\documentclass[10pt]{amsart}
 
\usepackage{amsthm,amsmath,amssymb}
\usepackage{graphicx}
\usepackage[comma, sort&compress]{natbib}
\usepackage{accents}
\usepackage{amssymb}
\usepackage{enumitem}
\RequirePackage{hyperref}

\newtheorem{theorem}{Theorem}
\newtheorem{convention}{Convention}

\newtheorem{property}[theorem]{Property}
\newtheorem{lemma}[theorem]{Lemma}

\theoremstyle{definition} 

\theoremstyle{remark}

\numberwithin{equation}{section}
\numberwithin{theorem}{section}
\numberwithin{example}{section}
\numberwithin{definition}{section}

\numberwithin{figure}{section}

\DeclareMathOperator{\var}{Var}
\DeclareMathOperator{\cS}{\mathcal{S}}
\DeclareMathOperator{\cE}{\mathcal{E}}

\DeclareMathOperator{\cH}{\mathcal{H}}

\DeclareMathOperator{\cX}{\mathcal{X}}
\DeclareMathOperator{\bE}{\mathbb{E}}
\DeclareMathOperator{\bR}{\mathbb{R}}

\newcommand{\barPhi}{\bar{\Phi}}

\newcommand{\bardelta}{\bar{\delta}}
\newcommand{\barD}{\bar{D}}
\newcommand{\E}{\mathbb{E}}

\newcommand{\barNR}{\bar{D}_{1,   x}}
\newcommand{\sigmahatstar}{\hat{\sigma}^*}
\newcommand{\sigmahat}{\hat{\sigma}}
\newcommand{\frakh}{\mathfrak{h}}
\newcommand{\frakc}{\mathfrak{c} } 
\newcommand{\frakD}{\mathfrak{D} } 
\newcommand{\frakcm}{\mathfrak{c}_m } 
\newcommand{\frakbm}{\mathfrak{b}_m } 
\newcommand{\cramer}{Cram$\acute{\text{e}}$r}
\newcommand{\anx}{a_{n, x}}

\newcommand{\DRplaceholder}{\mathfrak{D}_2} 
\newcommand{\barDRplaceholder}{\bar{\mathfrak{D}}_2} 

\newcommand{\secref}[1]{Section~\ref{sec:#1}}
\newcommand{\secsref}[1]{Sections~\ref{sec:#1}}
\newcommand{\secssref}[1]{\ref{sec:#1}}
\newcommand{\appref}[1]{Appendix~\ref{app:#1}}

\newcommand{\lemref}[1]{Lemma~\ref{lem:#1}}
\newcommand{\lemsref}[1]{Lemmas~\ref{lem:#1}}
\newcommand{\lemssref}[1]{\ref{lem:#1}}

\newcommand{\thmref}[1]{Theorem~\ref{thm:#1}}

\newcommand{\propertyref}[1]{Property~\ref{property:#1}}
\newcommand{\propertysref}[1]{Properties~\ref{property:#1}}
\newcommand{\propertyssref}[1]{\ref{property:#1}}

\title[]{Nonuniform Berry-Esseen bounds for Studentized U-statistics}

\author[D.~Leung]{Dennis Leung} 
\address{School of Mathematics and Statistics, University of Melbourne}
\email{dennis.leung@unimelb.edu.au}

\author[Q.~Shao]{Qi-Man Shao} 
\address{Department of Statistics and Data Science, SICM, National Center for Applied Mathematics Shenzhen, Southern University of Science and Technology}
\email{shaoqm@sustech.edu.cn}

\begin{document}

\begin{abstract}

We establish \emph{nonuniform} Berry-Esseen (B-E)  bounds for  Studentized U-statistics of the rate $1/\sqrt{n}$ under a third-moment assumption, which covers the t-statistic that corresponds to a kernel of degree $1$ as a special case. While an interesting data example  raised by  \citet{novak2005self} can show that   the  form of the nonuniform bound for standardized U-statistics  is actually \emph{invalid} for their Studentized counterparts,  our main results suggest that, the validity of such a bound can be restored by minimally augmenting it with an additive correction  term that decays exponentially in $n$. To our best knowledge, this is the first time that valid nonuniform B-E bounds for Studentized U-statistics have appeared in the literature. 
\end{abstract}



\maketitle

\section{Introduction} \label{sec:intro}

Let $X_1, \dots, X_n \in \cX$ be independent and identically distributed (i.i.d.) random variables taking values in a measurable space $(\cX, \Sigma_{\cX})$. A  U-statistic  \citep{hoeffding1948} of degree $m \geq 1$ is defined as 
\[
U_n ={n \choose m}^{-1} \sum_{1 \leq i_1 < \cdots < i_m \leq n} h(X_{i_1}, \dots, X_{i_m}),
\]
where $h: \cX^m \rightarrow \bR$ is a symmetric and measurable function in $m$ arguments, also known as a \emph{kernel} function.
 This important construction covers a wide range of statistics, including the sample mean  $n^{-1} \sum_{i=1}^n X_i$ as  the simplest example with $m=1$, for which   
\begin{equation} \label{sample_mean_as_ustat}
 h(x) = x \text{ and } \cX = \mathbb{R}.
\end{equation}
For the theorems stated in this article, we will throughout assume, without loss of generality,   that
\begin{equation} \label{zero_mean_kernel_assumption}
 \bE[h(X_1, \dots, X_m)] = 0,
\end{equation}
though knowing that such re-centering may not be done in practice because the mean of $h(\cdot)$ could be unknown. 
In the U-statistic literature, it is well established  that under the finite second-moment assumption 
 $\bE[h^2(X_1, \dots, X_m)]<\infty$
and the \emph{non-degeneracy condition}
\begin{equation} \label{non_degeneracy}
 \sigma^2 \equiv \var[g(X_1)] >0,
\end{equation}
where $g(\cdot)$ the first-order \emph{canonical function} defined by
\[
g(x)= \bE[h(X_1, X_2, \dots, X_m ) | X_1 = x],
\]
 one has the weak convergence  
\begin{equation} \label{standardized_one_sample_limit}
\frac{\sqrt{n}}{m \sigma} U_n  \longrightarrow_d N(0, 1) \text{ as } n \longrightarrow \infty,
\end{equation}
which extends the classical central limit theorem for the sample mean.

There has always been great interest in characterizing the normal approximation accuracy of \eqref{standardized_one_sample_limit} by Berry-Esseen (B-E) bounds; see \citet{filippova1962mises}, \citet{MR336788}, \citet{bickel1974edgeworth}, \citet{MR464359}, \citet{MR433551}, \citet{MR751580}, \citet{Friedrich},  \citet{chen2007normal} and \citet{bentkus1994lower} for an inexhaustive list of such works. For instance,  \citet{chen2007normal}'s results suggest that, under \eqref{zero_mean_kernel_assumption},  \eqref{non_degeneracy} and $\bE[|h(X_1, \dots, X_m)|^3]< \infty$, when $2m < n$, one has the bounds
\begin{equation}\label{standardized_ustat_unif_BE}
\sup_{x \in \bR}\left|P\left(\frac{\sqrt{n}}{m \sigma} U_n - x \right)  - \Phi(x)\right| \leq   C_1(m)\frac{ \bE[|h(X_1, \dots, X_m) |^3]}{\sqrt{n} \sigma^3}
\end{equation}
and
\begin{equation}\label{standardized_ustat_nonunif_BE}
\left|P\left( \frac{\sqrt{n}}{m \sigma} U_n \leq x\right) - \Phi(x)\right| \leq C_2(m) \frac{ \bE[|h(X_1, \dots, X_m) |^3]}{(1+ |x|)^3\sqrt{n} \sigma^3} \text{ for any }x \in \bR ,
\end{equation}
where $\Phi(x)$ is the standard normal distribution function, and $C_1(m)$ and $C_2(m)$ are  positive constants depending only on $m$\footnote{The moment quantities in \eqref{standardized_ustat_unif_BE} and \eqref{standardized_ustat_nonunif_BE}  have been simplified here for brevity; refer to \citet[Section 3.1]{chen2007normal} for  more sophisticated versions of such bounds.}. In contrast to the \emph{uniform} bound in \eqref{standardized_ustat_unif_BE}, \eqref{standardized_ustat_nonunif_BE} is known as a \emph{nonuniform} B-E bound, which is qualitatively more informative by having a "nonuniform" multiplicative factor that decays in the magnitude of $x$. 
 Without doubt, the sample mean from \eqref{sample_mean_as_ustat} has the richest literature since the works of \citet{berry1941accuracy} and \citet{esseen1942}, where even  the absolute constant's value is very well understood \citep{esseen1956moment, shevtsova2011absolute}.

Nevertheless, with some exceptions such as the rank-based Kendall's tau statistic \citep{kendall1938new}  for testing independence and  Wilcoxon signed rank statistic \citep{wilcoxon1992individual}  for testing medians, whose    respective degree-two kernels have $\sigma = 1/3$ and $\sigma = 1/12$ under a point null conditions like \eqref{zero_mean_kernel_assumption} and other regularity assumptions, $\sigma$ is typically unknown and cannot be directly used to standardize $U_n$.  It is hence more relevant  to develop a B-E bound for U-statistics that are \emph{Studentized}  with the data-driven Jackknife estimator of   $\sigma$ proposed by \citet{arvesen1969jackknifing}; in particular, for the special degree-one kernel in \eqref{sample_mean_as_ustat}, the resulting Studentized U-statistic is precisely the  t-statistic of Gosset \citep{student1908probable}. Other typical examples of U-statistics that must require Studentization are the sample variance and Gini's mean difference; see \citet[Section 1]{lai2011cramer} for the forms of their degree-two kernels. The quest for developing B-E bounds for such Studentized U-statistics
 has not gone unnoticed by researchers: Uniform B-E bounds of rate $1/\sqrt{n}$ analogous to \eqref{standardized_ustat_unif_BE} have been developed for Studentized U-statistics of degree $2$ by
\citet{helmers1985berry}, \citet{callaert1981order}, \citet{zhao1983rate} and \citet{jing2000berry}, respectively under $4.5$, $4+\varepsilon$ for any $\varepsilon > 0$, $4$ and $3$ finite absolute moments imposed on the kernel $h(X_1, X_2)$. Most recently, under $3$ finite absolute moments,   we  have  obtained a uniform B-E bound for Studentized U-statistics of any degree $m$, and also advocated \emph{variable censoring} as the appropriate technical device to prove such bounds under the  Stein-method approach \citep{leungshaounif}. 

To our best knowledge, a  nonuniform bound for Studentized U-statistics that is valid for  all  $x \in \bR$ in the same spirit as \eqref{standardized_ustat_nonunif_BE}  is  still eluding the literature,  even for  the t-statistic and the even simpler  \emph{self-normalized sum}  $S_n / V_n$, where
\begin{equation} \label{Sn_Vn_def}
S_n \equiv \sum_{i=1}^n X_i \text{ and }  V_n^2 \equiv \sum_{i=1}^n X_i^2 \text{ for i.i.d. } X_1, \dots, X_n \in \bR;
\end{equation}
see  \eqref{efron_relationship} below for a classical algebraic relationship between the t-statistic and the self-normalized sum.
In fact, an earlier nonuniform B-E bound for the self-normalized sum stated in   \citet[Corollary 2.3]{jing1999exponential} has  been latter disproved by an interesting binary data example  raised by  \citet[p.342-343]{novak2005self}, which also demonstrates it is in fact  \emph{impossible} to have a nonuniform B-E bound of the "usual form", 
\begin{equation}\label{wrong_sn_sum_nonunif_bdd}
\left|P\Bigg(\frac{S_n}{V_n} \leq x\Bigg)  - \Phi(x) \right|\leq  \frac{ C \bE[|X_1|^3]}{\sqrt{n} (\bE[|X_1|^2])^{3/2}} d(|x|) \text{ for all } x \in \bR,
\end{equation}
that holds for an absolute constant $C$ and any non-increasing function $d:\bR_{\geq 0} \rightarrow \bR_{\geq 0}$ with the property  $\lim_{x \rightarrow \infty} d(x) = 0$, assuming $\bE[X_1] = 0$.

This void is now filled by the new nonuniform B-E bound for  \emph{Studentized U-statistics} of any degree $m$ established in this paper. 
As we point out in  \secref{U_stat_review}, \citet{novak2005self}'s example also readily implies that, for a Studentized U-statistic $T_n$, it is similarly impossible to have a bound of the  form:
\begin{equation} \label{usual_form}
|P(T_n \leq x) - \Phi(x)| \leq \frac{C(m) \bE[|h(X_1, \dots, X_m)|^3]}{\sqrt{n} \sigma^3} d(|x|)
\end{equation}
that holds universally for all types of data distributions and kernels, where $C(m)$ is a positive  constant depending only on $m$ and $d$ is any non-increasing function with the same property as the one alluded to in \eqref{wrong_sn_sum_nonunif_bdd}. As such, our new nonuniform B-E bound for $T_n$  has to give up the form in \eqref{usual_form}, but, interestingly, not too much; our main theorem (\thmref{main}) suggests that, to restore the validity, it suffices to minimally augment the bound  with an additive correction term
\[
\exp\Bigg(-\frac{c(m)n \sigma^6 }{ (\bE[|h(X_1, \dots, X_m)|^3])^2} \Bigg)
\]
 that decays exponentially in $n$, for a small   constant $c(m) > 0$.

Our proof follows Stein's method, in a similar vein as our  work \citep{leungshaounif} on developing uniform B-E bounds for self-normalized nonlinear statistics. We  comment on two major departures in terms of techniques: First, to elicit the nonuniformity in $x$, considerably more delicate censoring techniques than the ones in \citet{leungshaounif} have to be employed. Secondly, to obtain the correction  term that decays exponentially in $n$, we analyze the  Jackknife estimate  of $\sigma$ by proving an exponential lower-tail bound developed for U-statistics with non-negative kernels (\lemref{chernoff_lower_tail_bdd_U_stat}); the latter result is a crucial technical tool, which naturally extends a similar result for a sum of non-negative random variables and is of independent interest.

{\bf Organization.} \secref{U_stat_review}  covers the basics of  Studentized U-statistics, and   revisits \citet{novak2005self}'s data example to deduce that the nonuniform bound in the usual form of \eqref{usual_form} cannot be valid. \secref{main} states our new nonuniform B-E bounds, including a general one for Studentized U-statistics and a further refined one for the t-statistic.  \secsref{mainpf} and \secssref{sn_sum_pf} respectively prove the two theorems in \secref{main}, with the appendices  covering additional technical proofs integral to them.

{\bf Notation}. For any $p \geq 1$,  we use $\|X\|_p = ( \mathbb{E}|X|^p)^{1/p}$ to denote the $L_p$-norm of  any real-valued random variable $X$;   if $f:\cX^L \rightarrow \bR$ is any function in $L \in \{1, \dots, n\}$ arguments, we may  use $\bE[f]$ as shorthand for $\bE[f(X_1, \dots, X_L)]$; likewise, we may use $\|f\|_p$ as a shorthand for the $p$-norm $\|f(X_1, \dots, X_L)\|_p$.
 If $a, b \in \bR$, we let $a \vee b = \max(a, b)$ and $a \wedge b = \min(a, b)$. $\barPhi (\cdot) \equiv 1 - \Phi(\cdot)$ is the standard normal survival function, $\phi(\cdot)$ denotes the standard normal density, and  $I(\cdot)$ denotes the indicator function.    For any subset $\cS \subset \{1, \dots, n\}$, we shall let  $X_{\cS} \equiv (X_s)_{s \in \cS}$ be a vector of variables from $X_1, \dots, X_n$ with sample indices in $\cS$, and  $x_{\cS} = (x_s)_{s \in \cS}$ be a similar sub-vector of any generic vector $(x_1, \dots, x_n) \in \bR^n$.
 $C, c, C_1, c_1, C_2, c_2 \dots$   denote  unspecified \emph{absolute} positive constants, where "absolute" means they are universal for all underlying distributions of the variables involved and do not depend on other quantities; if a  positive  constant does depend  on other quantities such as $a$ and/or $b$ \emph{exclusively},  it will be explicitly specified as $C(a)$, $C(a,b)$, $c(a)$, $c(a, b)$, etc. to emphasize the dependence on $a$, $(a,b)$, etc. \emph{All these absolute constants generally differ in values at different occurences}.

\section{Studentized U-statistics and \citet{novak2005self}'s example} \label{sec:U_stat_review}

We first review the basics of Studentized  U-statistics.   With 

\[
q_i = \frac{1}{{n-1 \choose m-1}} \sum_{\substack{1 \leq i_1 < \dots < i_{m-1} \leq n\\ i_l \neq i \text{ for } l = 1, \dots, m-1}} h(X_i, X_{i_1}, \dots, X_{i_{m-1}}), \qquad i = 1, \dots, n,
\]
serving as proxies  for the unknown quantities $g(X_1), \dots, g(X_n)$, the  "leave-one-out" Jackknife estimator \citep{arvesen1969jackknifing} for $\sigma^2$ is constructed as
\begin{equation} \label{real_studentizer}
\sigmahat^2 = \frac{n-1}{(n-m)^2} \sum_{i=1}^n (q_i- U_n)^2 
= \frac{n-1}{(n-m)^2}\Big( \sum_{i=1}^n q_i^2 - n U_n^2 \Big)
\end{equation}
to define the \emph{Studentized U-statistic}
\begin{equation}\label{student_utat}
\quad T_n  =  \frac{\sqrt{n}}{m \sigmahat} U_n.
\end{equation}
For the special case of $m =1$ and the kernel in \eqref{sample_mean_as_ustat}, one can check that $T_n$ is precisely the Student's t-statistic
\citep{student1908probable}
\[
T_{student} \equiv \frac{\sqrt{n} \bar{X}_n}{s_n},
\]
where $\bar{X}_n = n^{-1}\sum_{i=1}^n X_i$ and $s_n^2 = (n-1)^{-1} \sum_{i=1}^n (X_i - \bar{X}_n)^2$.  
It is instructive to clarify the value taken upon by $T_n$  when $\sigmahat$ is equal to zero, which could be the case for some realizations of the data.
The following convention is adopted:
\begin{convention} [Convention for $T_n$ when $\sigmahat = 0$] \label{con:convention1}
\ \
\begin{enumerate}
\item  If $\sigmahat = 0$ and $U_n \neq 0$,  $T_n$ is assigned the value $+\infty$ or $- \infty$ following the sign of $U_n$. 
\item  If $\sigmahat = 0$ and $U_n = 0$,  $T_n$ is assigned the value $0$. 
\end{enumerate}
\end{convention}

Under this convention, there is no ambiguity in understanding an event like $\{T_n \leq x\}$ for any $x \in \bR$ and  its probability. Recently, the following \emph{uniform} B-E bound has been established for $T_n$:
\begin{theorem}[Uniform B-E bound for Studentized U-statistics, \citet{leungshaounif}] \label{thm:BE_unif} 
Assume \eqref{zero_mean_kernel_assumption}-\eqref{non_degeneracy}, $2m < n$  and 
$
\bE[|h|^3] < \infty$. For a positive absolute constant $C(m) >0$ depending on $m$ only,  the following Berry-Esseen bound holds:
\begin{equation*}
\sup_{x \in \bR}|P(T_n \leq x) - \Phi(x)| \leq 
\frac{ C(m) }{\sqrt{n}}\left\{  \frac{  \|h\|_2^2}{\sigma^2} + \frac{\|g\|_3^2\|h\|_3}{\sigma^3}  \right\}.
\end{equation*}
In particular,  the bound above can be further simplified as 
\begin{equation}   \label{BE_Tn}
\sup_{x \in \bR}|P(T_n \leq x) - \Phi(x)| \leq 
\frac{ C(m) }{\sqrt{n}}\frac{\bE[|h|^3]}{\sigma^3}  .
\end{equation}
\end{theorem}

While the uniform bound in \eqref{BE_Tn} resembles the uniform bound for standardized U-statistics in \eqref{standardized_ustat_unif_BE}, as mentioned in \secref{intro}, it is impossible to obtain a nonuniform bound of the form in \eqref{usual_form} that resembles  the nonuniform bound in \eqref{standardized_ustat_nonunif_BE} for standardized U-statistics. To see this, we shall first revisit how \citet[p.342-343]{novak2005self}  refuted the prospective nonuniform bound for the self-normalized sum in  \eqref{wrong_sn_sum_nonunif_bdd}, via constructing $X_1, \dots, X_n$ as i.i.d. binary variables such that
\begin{equation} \label{bin_data_def}
P\Bigl(X_i = p^{1/2} (1 - p)^{-1/2}\Bigr) = 1- p \text{ and } P\Bigl(X_i = - (1 - p)^{1/2} p^{-1/2}\Bigr) = p
\end{equation}
for some  $p \in (0,1)$; the expectation, as well as the second and third absolute moments of $X_1$ is
 \begin{equation} \label{novak_moments}
\bE[X_1] = 0, \quad  \bE[X_1^2] = 1  \text{ and }
\bE[|X_1|^{3}] = p^{3/2} (1 - p)^{-1/2} + (1 - p)^{3/2} p^{-1/2}.
\end{equation} 
For such data, by letting
\begin{equation} \label{novak_param_choice}
p = p_n \equiv n^{-1} \text{ and }x = x_n \equiv \sqrt{n} - \epsilon \text{ for any small fixed constant } \epsilon > 0,
\end{equation}
the right hand side of \eqref{wrong_sn_sum_nonunif_bdd} is seen to be equal to
\begin{equation} \label{rhs_sn_presumed_bdd}
C \frac{ p_n^{\frac{3}{2}} (1 - p_n)^{-\frac{1}{2}} + (1 - p_n)^{\frac{3}{2}} p_n^{-\frac{1}{2}}}{\sqrt{n} } d(|x_n|)\\
=   C\left\{ n^{- 2 } (1 - n^{-1})^{- 1/2} + (1 - n^{-1})^{3/2}  \right\} d(|x_n|).
\end{equation}
Suppose, towards a contradiction, that the bound in \eqref{wrong_sn_sum_nonunif_bdd} \emph{does} hold.  Consider the event 
\begin{equation} \label{cE_event}
\cE_n \equiv \{X_1 = \dots = X_n = p_n^{1/2} (1 - p_n)^{-1/2}\},
\end{equation}
on which the self-normalized sum $S_n/V_n$ can be easily seen to take upon the value $\sqrt{n}$, which is greater than $x_n$; one can then consequently derive the lower bound $e^{-1} > 0$  for the  "liminf" of the left hand side in \eqref{wrong_sn_sum_nonunif_bdd} as:
\begin{align}
\liminf_{n \rightarrow \infty} \Bigl[P(S_n/V_n\leq x_n) - \Phi(x_n) \Bigr] 
 &= \liminf_{n \rightarrow \infty} \Bigl[P(S_n/V_n > x_n) - \barPhi(x_n) \Bigr] \notag\\
&\geq \liminf_{n \rightarrow \infty} \Bigl[P(\cE_n) - \barPhi(x_n) \Bigr]\notag\\
&=  \liminf_{n \rightarrow \infty} \Big[(1 -p_n)^n - \barPhi(x_n) \Big] = 1/e  \label{lhs_prp_sn}.
\end{align}
However, this contradicts the presumed bound in  \eqref{wrong_sn_sum_nonunif_bdd}, since the right hand side in \eqref{rhs_sn_presumed_bdd} converges to zero as $n \rightarrow \infty$,  given the assumed property $\lim_{x \rightarrow \infty}d(x) = 0$.

Likewise,  the nonuniform Berry-Esseen-type  bound  \eqref{usual_form} can't hold for Studentized U-statistics either. In fact, assuming the data are as in \citet{novak2005self}'s construction in  \eqref{bin_data_def} again, we hereby show that an even wider class of bounds that include \eqref{usual_form} as a special case cannot hold: We will show by contradiction that, it is impossible to have a bound of the form
\begin{equation}\label{impossible_general_form}
|P(T_n \leq x) - \Phi(x)| \leq C(m, n,  x, \mathcal{L}_{X_1}, h ),
\end{equation}
where the right hand side is an absolute term depending only on  $m$, $n$, $x$, the law $\mathcal{L}_{X_1}$ of the representative variable $X_1$ and (attributes of) the kernel $h$ in such a way that, 
\begin{equation} \label{prop_of_Cmnhx}
 \lim_{x \rightarrow \infty}C(m, n,  x, \mathcal{L}_{X_1}, h ) = 0 \text{ when the other parameters } (m, n,  \mathcal{L}_{X_1}, h) \text{ are held fixed}.
\end{equation}
 First, we define a special real-valued, symmetric kernel $h$ of degree $m \geq 1$ by
\begin{equation} \label{sum_kernel}
h: \bR^m \rightarrow \bR \text{ and }  h(x_1, \dots, x_m) \equiv x_1 +\dots + x_m.
\end{equation}
One can check with elementary calculations that
$
U_n = \frac{m}{n} \sum_{i=1}^n X_i$
and
\begin{equation*}
\sum_{i=1}^n q_i^2 - n U_n^2 
=  \bigg(\frac{n-m}{n-1} \bigg)^2 \bigg(\sum_{i=1}^n X_i^2 -  n (\bar{X}_n)^2 \bigg)
\end{equation*}
when the U-statistic is formed with this particular kernel  in \eqref{sum_kernel}; as such, from the definition of $\hat{\sigma}$ in \eqref{real_studentizer}, one can see that 
\begin{equation} \label{Tn_eq_Tstudent}
T_n = T_{student} \text{ for any $m \geq 1$ and the kernel in } \eqref{sum_kernel}.
\end{equation}
 Next,  recall   the classical relationship between the self-normalized sum and Student's t-statistics \citep{efron1969}:
\begin{equation} \label{efron_relationship}
T_{student}  = \frac{S_n}{V_n} \Bigg\{ \frac{n-1}{n - (S_n/V_n)^2} \Bigg\}^{1/2}.
\end{equation}
On the event $\cE_n$ in \eqref{cE_event},  $S_n/V_n$ is equal to $\sqrt{n}$ and hence $T_{student}$ takes the value $\infty$ in light of \eqref{efron_relationship}; as such,  by the equality in \eqref{Tn_eq_Tstudent},
\begin{equation} \label{Tn_eq_infty_on_En}
T_n = \infty \text{ on the event } \cE_n, \text{ for all } n.
\end{equation}
Since 
$
\lim_{n \rightarrow \infty}P(\cE_n)  = \lim_{n \rightarrow \infty}(1 -p_n)^n  = e^{-1}
$,
we will let $N \in \mathbb{N}$ be such that $P(\cE_N) \geq e^{-1}/2$. However, the fact in \eqref{Tn_eq_infty_on_En} implies that
\[
\liminf_{x \rightarrow \infty} (P(T_N \leq x) - \Phi(x)) = \liminf_{x \rightarrow \infty} (P(T_N > x) - \bar{\Phi}(x)) \geq P(\cE_N) - \lim_{x \rightarrow \infty} \bar{\Phi}(x) \geq e^{-1}/2,
\]
which in turns implies
$\liminf_{x \rightarrow \infty} |P(T_N \leq x) - \Phi(x)| \geq e^{-1}/2$. 
Apparently, the last fact breaks the bound  in \eqref{impossible_general_form} with the presumed property in \eqref{prop_of_Cmnhx}!

\section{Main results}\label{sec:main}

The moral of  \citet{novak2005self}'s example  in \secref{U_stat_review} is that, due to the way that the Jackknife Studentizer  $\hat{\sigma}$ in  \eqref{real_studentizer} is constructed, when the distribution of the data in question is such that $T_n$ can take its largest possible value (i.e. $\infty$) with a non-negligible probability, a bound like \eqref{usual_form}  may fail to hold. 
 We now state our main theorem, which contains what we consider to be the correct nonuniform  B-E bound for Studentized U-statistics; it suggests that it is enough to augment the  form in \eqref{usual_form} with an extra term that decays exponentially in $n$.

\begin{theorem}[Nonuniform B-E bounds for Studentized U-statistics]\label{thm:main}
Let $X_1, \dots, X_n \in \cX$ be independently and identically distributed random variables. Under \eqref{zero_mean_kernel_assumption}-\eqref{non_degeneracy}, $\max(2, m^2) < n$ and the moment condition $\bE[|h|^3] < \infty$, for any $x \in \bR$, there exist positive absolute constants $C(m)$, $c_1(m)$ and $c_2(m)$ such that
 \begin{multline}\label{BE_Tn_one_sample}
|P(T_n \leq x) - \Phi(x)| \leq \exp\Bigg(-\frac{c_1(m)n \sigma^6 }{ (\bE[|h|^3])^2} \Bigg) +
 \\
C(m)\Bigg\{ \frac{1}{ (1 +|x|^3)} \Bigg( \frac{ \bE[|h|^3]}{ n^{3/2} \sigma^3} + \frac{\bE[|g|^3]}{\sqrt{n} \sigma^3}\Bigg) +
 \frac{1 }{e^{c_2(m)|x|} \sqrt{n}} \Bigg(\frac{\|g\|_3^2 \|h\|_3}{\sigma^3} + \frac{  \|h\|_3^2}{\sigma^2 }\Bigg) \Bigg\};
\end{multline}
In particular, this implies, for some positive absolute constants $C(m)$ and $c(m)$, 
\begin{equation} \label{BE_Tn_one_sample_simpler}
|P(T_n \leq x) - \Phi(x)| \leq   \exp\Bigg(-\frac{c(m)n \sigma^6 }{ (\bE[|h|^3])^2} \Bigg) + \frac{C(m)\bE[|h|^3]}{(1+ |x|^3)\sqrt{n} \sigma^3}.
\end{equation}
\end{theorem}

\thmref{main} is proved in \secref{mainpf}. Note that \eqref{BE_Tn_one_sample_simpler} is a simple consequence of \eqref{BE_Tn_one_sample} because $\|g\|_3 \leq \|h\|_3$, due to the basic U-statistic property in \eqref{Jensen} below. For the  choice of the probability $p = p_n = n^{-1}$  in \eqref{novak_param_choice}, let us now re-examine \citet{novak2005self}'s binary data in \eqref{bin_data_def} and our special kernel $h$ in \eqref{sum_kernel}  to demonstrate three features of our new bounds  in \thmref{main}. Note that, 
by considering $x_n$ in  \eqref{novak_param_choice} and the fact established in \eqref{Tn_eq_infty_on_En} for the event  $\cE_n$ defined in \eqref{cE_event}, we must have 
\begin{equation} \label{lower_eminus1_bdd}
\liminf_{n \rightarrow \infty} [P (T_n \leq x_n ) - \Phi(x_n)] = \liminf_{n \rightarrow \infty} [P (T_n > x_n ) - \bar{\Phi}(x_n)] \geq \liminf_{n \rightarrow \infty} [P (\cE_n) - \Phi(x_n)] = e^{-1}.
\end{equation}
Moreover, we will also  leverage the following moment bounds  for the kernel in \eqref{sum_kernel},
\begin{equation}\label{lower_upper_rosenthal} 
8^{-1} m \bE[|X_1|^3]    
 \leq  \bE[|h|^{3}]  \leq   
C(m)\bE[|X_1|^3],
\end{equation}
where $C(m)>0$ is an absolute constant depending only on $m$; the bounds in \eqref{lower_upper_rosenthal}
are  direct consequences of the classical Rosenthal's inequalities \citep[Theorem 3]{rosenthal1970subspaces}.

\begin{enumerate} 
\item {\bf The new B-E bounds can  accommodate  more "unusual" data distributions}: 
By the lower bound $8^{-1} m \bE[|X_1|^3]  \leq  \bE[|h|^{3}]$ in  \eqref{lower_upper_rosenthal} and the moment calculations in \eqref{novak_moments},   
with  the   choice $p = p_n$  in \eqref{novak_param_choice}, we  get
\begin{align*}
\exp \Bigg( - \frac{c(m) n \sigma^6 }{(\bE[|h|^3])^2}\Bigg) &\geq  \exp \Bigg( - \frac{c(m) n  }{(8^{-1} m (n^{-3/2} (1 - n^{-1})^{-1/2} + (1 - n^{-1})^{3/2} n^{1/2}))^2}\Bigg) \\
&\geq \exp \Bigg( - \frac{64 \cdot c(m)  }{m^2 (1 - n^{-1})^{3} }\Bigg).
\end{align*}
 Hence,   given $n >2$, $  \exp ( - \frac{64 \cdot c(m)  }{m^2 (1 - n^{-1})^{3} })$ is  larger than the lower bound $e^{-1}$ in \eqref{lower_eminus1_bdd}  for a sufficiently small $c(m) > 0$; as such, unlike \eqref{impossible_general_form}, the new bounds \eqref{BE_Tn_one_sample} and \eqref{BE_Tn_one_sample_simpler} are not contradicted. 
\item {\bf The correction term could be crucial even when $|x|$ is not large relative  to $n$}: As  $\bE[|h|^3]/\sigma^3 = n^{- 3/2 } (1 - n^{-1})^{- 1/2} + (1 - n^{-1})^{3/2}n^{1/2} \sim \sqrt{n}$ as $n \rightarrow \infty$,  we have 
 \begin{equation} \label{need_correction}
\frac{C(m)\bE[|h|^3]}{(1+ |x|^3)\sqrt{n} \sigma^3} \sim \frac{C(m)}{(1 + |x|^3)}\text{ for large } n.
 \end{equation}
 The last display implies that, for  an unusual data distribution  where the moment ratio $\bE[|h|^3]/\sigma^3$ can be as large as $\sqrt{n}$, the need for having a correction term as in   \eqref{BE_Tn_one_sample_simpler} could arise as long as $x$ is of the order $O(n^a)$ for even a  small $a> 0$, because the term in \eqref{need_correction} could  be already too small to bound  
the left hand side of \eqref{BE_Tn_one_sample_simpler}, as suggested by the lower bound for the "$\liminf_{n \rightarrow \infty}$" in \eqref{lower_eminus1_bdd}.
\item \emph{\bf The  order of $n$ in the correction term is optimal}: \citet{novak2005self}'s example also illustrates the current order of $n$ in our additive correction proposal is optimal.  Suppose, toward a contradiction, that the correction term had instead taken the  form $\exp ( - \frac{c(m) n^a \sigma^6 }{(\bE[|h|^3])^2} )$ for a  power $a > 1$, with a faster decay in $n$.  In light of the upper bound $\bE[|h|^{3}]  \leq   
C(m)\bE[|X_1|^3]$  in \eqref{lower_upper_rosenthal}, such a correction term could then be further upper bounded by 
\begin{equation} \label{rosenthal_resulting_upperbdd}
 \exp \Bigg( - \frac{ c(m) n^{a-1}  }{(n^{-2}  (1 - n^{-1})^{-1/2}+ (1 - n^{-1})^{3/2} )^2 }\Bigg)
\end{equation}
for the binary data in \eqref{bin_data_def}, the parameter choice $p = p_n = n^{-1}$  in \eqref{novak_param_choice}, our kernel in \eqref{sum_kernel} and a sufficiently small constant $c(m) > 0$. Apparently, because $a - 1 >0$, the term in the \eqref{rosenthal_resulting_upperbdd} converges to zero as $n \rightarrow \infty$; this implies that the hypothetical correction term with $a > 1$ cannot bound  the "$\liminf_{n \rightarrow \infty}$" in \eqref{lower_eminus1_bdd}, because even its upper bound in \eqref{rosenthal_resulting_upperbdd} can't!
\end{enumerate}

While we believe \thmref{main} has nailed down the correct nonuniform B-E bounds for Studentized U-statistics, two aspects  related to the degree $m$ shall  warrant further investigation:
\begin{enumerate}
\item 
 In the   B-E bounds for the \emph{standardized} U-statistics,   \citet[Section 3.1]{chen2007normal} has actually  shown  that, as  $m$ increases, the absolute constants,  i.e. $C_1(m)$, $C_2(m)$ in \eqref{standardized_ustat_unif_BE} and \eqref{standardized_ustat_nonunif_BE}, only grow very pleasantly at the rate $\sqrt{m}$. However, for \thmref{main}, 
due to the fundamental  challenge posed by Studentization, it is unclear to us what the best possible (i.e. slowest) growth rate of the constants in $m$ should be. We defer  further  discussion of this to \secref{remark1}, after we have finished proving \thmref{main}.
%
\item The condition $\max(2, m^2) < n$ in \thmref{main} is  stronger than the typical $2m < n$ assumed for the uniform B-E bound in \thmref{BE_unif}. As will be seen in \secref{mainpf}, letting $\max(2, m^2) < n$  facilitates our analysis of the lower tail probability of the Studentizer $\sigmahat$ as a non-negative-kernel U-statistic  using the crucial \lemref{chernoff_lower_tail_bdd_U_stat}, which ultimately leads to our correction term with exponential decay in $n$. However, we believe establishing our theorem under  $2m < n$ is potentially feasible, and the related discussion will appear in \secref{remark2}.
\end{enumerate}

Aside from our general result in \thmref{main}, thanks to the  delicate \emph{\cramer-type} moderate deviation theorem  for the self-normalized sum $S_n/V_n$ established in  \citet{jing2003self},  a very refined nonuniform B-E bound for the  Student's t-statistic, the special case of $T_n$ with the kernel in  \eqref{sample_mean_as_ustat}, can be established  in \thmref{nonunif_BE_sn_sum} below; the proof is  \secref{sn_sum_pf}.   It says that the nonuniform term in $x$ can be further strengthened to be decaying exponentially in  $|x|$. It is an open question whether one can similarly strengthen the rate of decay in $|x|$ for our general result in \thmref{main}, as the current state-of-the-art in  \emph{\cramer-type} moderate deviation results for Studentized U-statistics  applies to a restricted class of kernels only \citep[Eqn. $(3.3)$]{shao2016cramer}.

\begin{theorem}[Nonuniform B-E bound for   Student's t-statistic] \label{thm:nonunif_BE_sn_sum}
Let $X_1, \dots, X_n$ be independent and identically distributed real-valued random variables such that $\bE[X_1] =0$, $0 < \bE[X_1^2] < \infty$ and $\bE[|X_1|^3] < \infty$. Assume $n \geq 2$. Then there exist positive absolute constants $C_1, C_2, c_1, c_2 >0$ such that

 \begin{equation*}\label{BE_sn_sum}
\bigg|P\bigg(T_{student}\leq x\bigg) - \Phi(x)\bigg| \leq  C_1\exp \Bigg( \frac{-  c_1 n ( \bE[X_1^2])^3}{ (\bE[X_1^3])^2} \Bigg) +  \frac{C_2}{e^{c_2x^2}} \frac{ \bE[|X_1|^3]}{\sqrt{n} (\bE[X_1^2])^{3/2}}.
\end{equation*}
The same bound can be stated with $T_n$ replaced by the self-normalized sum  $S_n/V_n$, where $S_n$ and $V_n$ are defined in \eqref{Sn_Vn_def}, for possibly different  constants $C_1, C_2, c_1, c_2 > 0$.
\end{theorem}


\section{\bf Proof of the nonuniform B-E bound for Studentized  U-statistics}\label{sec:mainpf}

This section lays out  the major steps of the proof for  \thmref{main}. It suffices to consider $x \geq 0$ only, or else one can replace the kernel $h$ with $- h$. Moreover, one can further just focus on $x \geq 1$; for the range $0 \leq x < 1$, because  $e^{- c(m)x} \geq e^{- c(m)}$ for any small positive constant $c(m)$, one can always inflate the constant $C(m)$ in \eqref{BE_Tn_one_sample} sufficiently so that \thmref{main} is true by virtue of the uniform bound in \thmref{BE_unif}.  Hence, this section focuses on proving
\begin{multline} \label{simplified_goal}
 |P(T_n \leq x) - \Phi(x)|  \leq \exp\Bigg(-\frac{c_1(m)n \sigma^6 }{ (\bE[|h|^3])^2} \Bigg) +
 \\
C(m)\Bigg\{ \frac{1}{ (1 +x^3)} \Bigg( \frac{ \bE[|h|^3]}{ n^{3/2} \sigma^3} + \frac{\bE[|g|^3]}{\sqrt{n} \sigma^3}\Bigg) +
 \frac{1 }{e^{c_2(m)x} \sqrt{n}} \Bigg(\frac{\|g\|_3^2 \|h\|_3}{\sigma^3} + \frac{  \|h\|_3^2}{\sigma^2 }\Bigg) \Bigg\}
   \text{ for } x \geq 1,
\end{multline}
for some absolute constant $C(m), c_1(m), c_2(m) >0$.

Without loss of generality, we assume
\begin{equation}\label{sigma_equal_1_assumption}
\sigma^2 = 1
\end{equation}
as  one can always replace $h(\cdot)$  and $g(\cdot)$ respectively with $h(\cdot)/\sigma$  and $g(\cdot)/\sigma$ without changing the definition of $T_n$. To prove \eqref{simplified_goal}, we  adopt the framework of 
 \emph{self-normalized} nonlinear   statistics, which amounts to writing $T_n$  as 
\begin{equation} \label{Tsn}
T_n = \frac{W + D_{1}}{(1 + D_{2})^{1/2}},
\end{equation}
where 
\begin{equation} \label{W_def}
W = W_n \equiv \sum_{i=1}^n \xi_i,
\end{equation}
 with $\xi_1, \dots, \xi_n$ being independent random variables such that 
\begin{equation} \label{simpAssumptions}
\bE[{\xi_i}] = 0 \text{ for all } i =1, \dots, n, \quad { and } \quad \sum_{i=1}^n\bE[\xi_i^2] = 1;
\end{equation}
$D_{1}$ and $D_{2}$ are random "remainder" terms that are  negligible when $n$ is large, with the additional property that 
\begin{equation} \label{D2_as_geq_minus1}
D_2 \geq -1 \text{ almost surely}.
\end{equation}
This is accomplished by first letting
\begin{equation} \label{xi_for_u_stat}
\xi_i = \frac{g(X_i)}{\sqrt{n}} \text{ for } i = 1, \dots, n;
\end{equation}
under both assumptions \eqref{zero_mean_kernel_assumption} and \eqref{sigma_equal_1_assumption}, it is seen that the properties in \eqref{simpAssumptions}are satisfied. 
Define
\begin{equation} \label{h_k_def}
\bar{h}_k(x_1 \dots, x_k) = h_k(x_1 \dots, x_k) - \sum_{i=1}^k g(x_i) \text{ for } k = 1, \dots, m,
\end{equation}
where 
\[
h_k(x_1, \dots, x_k) = \bE[h(X_1, \dots, X_m) |X_1 = x_1, \dots, X_k = x_k ];
\] 
in particular, $g(x) = h_1(x)$, $h(x_1, \dots, x_m) = h_m(x_1, \dots, x_m)$, and $\bar{h}_k$ for any $k \in \{1, \dots, m\}$ has the degeneracy property
\begin{equation} \label{degen_prop_bar_h_m}
\bE[\bar{h}_k (X_1, \dots, X_k)|X_i] = 0 \text{ for any } i = 1, \dots, k.
\end{equation}  
For  any $p \geq 1$, an important property of the functions $h_k$ is that 
\begin{equation} \label{Jensen}
\bE\Big[ |h_k|^p\Big] \leq \bE\Big[ |h_{k'}|^p\Big] \text{ for }k \leq k',\footnote{See \citet[Eqn. $(3.10)$]{leungshaounif} for instance.}
\end{equation}
which necessarily implies, for a constant $C(k) > 0$ depending only on $k$,
\begin{equation}\label{Jensen_consequence}
 \bE[|\bar{h}_k|^p] \leq C(k) \bE[|h_k|^p] \leq C(k) \bE[|h|^p] .
\end{equation}
 By the Hoeffding decomposition, for $W$ in \eqref{W_def} constructed with \eqref{xi_for_u_stat}, we have 
 \[
 \frac{\sqrt{n}}{m} U_n = W  +  {n -1\choose m -1}^{-1} \sum_{1 \leq i_1 < \dots < i_m\leq n} \frac{\bar{h}_m(X_{i_1}, \dots, X_{i_m})}{\sqrt{n}}.
 \]
Hence,  the "numerator remainder"  for $T_n$ can be defined as
 \begin{equation} \label{D1}
D_1 =  D_1(X_1, \dots, X_n)\equiv {n -1\choose m -1}^{-1} \sum_{1 \leq i_1 < \dots < i_m\leq n} \frac{\bar{h}_m(X_{i_1}, \dots, X_{i_m})}{\sqrt{n}},
\end{equation}
and the "denominator remainder" can be taken as
\begin{equation} \label{D2_dual_def}
D_2 = D_2(X_1, \dots, X_m)\equiv 
 \sigmahat^2  -1  .
\end{equation}
By further defining 
\begin{multline*}
 \quad V_n^2 = \sum_{i=1}
^n \xi_i^2, \quad 
\Psi_{n, i} =  \sum_{\substack{1 \leq i_1 < \dots < i_{m-1} \leq n\\ i_l \neq i \text{ for } l = 1, \dots, m-1}} \frac{\bar{h}_m(X_i, X_{i_1}, \dots, X_{i_{m-1}}) }{\sqrt{n}}
\quad \text{ and } \quad  \Lambda_n^2 = \sum_{i=1}^n \Psi_{n, i}^2 ,
\end{multline*}
as well as 
\begin{multline}\label{delta_1_star}
\delta_1^* = \delta_{1n}^*
  \equiv \left[ \frac{ n(m-1)^2}{(n-m)^2} + \frac{2(m-1)}{(n-m)}\right] W^2 +  \frac{(n-1)^2}{ {n-1 \choose m-1}^2 (n-m)^2} \Lambda_n^2 + \frac{2(n-1)(m-1) }{(n-m)^2 {n-1 \choose m-1}} \sum_{i = 1}^n  W \Psi_{n, i},
\end{multline}
 one can then also write $D_2$ as
\begin{equation} \label{s_star_re_expr}
D_2 = d_n^2( V_n^2 + \delta_1 + \delta_2) -1 \text{ for }  d_n = \sqrt{\frac{n}{n-1}}, 
\end{equation}
with 
\begin{equation}\label{delta_1}
\delta_1 \equiv
  \delta_1^* -   \frac{(n-1)^2}{(n-m)^2}U_n^2  
\end{equation}

and
\[
\delta_2   \equiv \frac{2 (n-1) }{(n-m)} {n-1 \choose m-1}^{-1} \sum_{i=1}^n \xi_i \Psi_{n,i}; 
\]
see our related work \citet[Section 3]{leungshaounif} for the derivation of \eqref{s_star_re_expr}.
\footnote{Specifically, it was showed that ${\sigmahatstar}^2 = d_n^2 (V_n^2 + \delta_1^* + \delta_2)$; see the self-normalized U-statistic in \eqref{sn_ustat}. One can then  deduce from \eqref{real_studentizer} that  $\sigmahat^2 = d_n^2 (V_n^2 + \delta_1^*  -   \frac{(n-1)^2}{(n-m)^2}U_n^2 + \delta_2)$.} 
We also note that, although defining $ V_n^2 = \sum_{i=1}
^n \xi_i^2$ slightly abuses the definition of $V_n^2$ for the self-normalized sum in  \eqref{Sn_Vn_def}, one can think of $\xi_i$'s as analogous to the real-valued $X_i$'s in the self-normalized sum.

In \citet{leungshaounif},  replacing the more common truncation technique, \emph{variable censoring} is  advocated as the appropriate  device to establish B-E bounds for self-normalized nonlinear statistics under the Stein-method approach.  In this paper, variable censoring is also adopted; in particular, the censored summands 
\begin{equation*} \label{xi_censored}
\xi_{b,i} \equiv \xi_i I(|\xi_i| \leq 1) +  I(\xi_i > 1) -   I(\xi_i <- 1) \text{ for each } i = 1, \dots, n,
\end{equation*}
as well as their sum
\begin{equation*} \label{xi_i_upper_censored}
W_b \equiv \sum_{i=1}^n \xi_{b, i}
\end{equation*}
will also figure in our proof. However, the other two remainder terms $D_1$ and $D_2$ have to be censored in a considerably more delicate manner as described next. First, we shall define a special  positive  constant "$\mathfrak{c}_m$"   via its square:
\begin{equation} \label{frakcm_def}
\mathfrak{c}_m^2   \equiv  \Big( 1 - \frac{m^2}{m^2 +1}\Big) \times \mathfrak{b}_m,
\end{equation}
where $\mathfrak{b}_m$ is a constant  depending also only  on $m$ defined as
\begin{equation} \label{frakbm_def}
\mathfrak{b}_m  \equiv \begin{cases} 
& \frac{1}{2} \text{ if } m = 1 \text{ or } 2; \\
 & \underbrace{ \left(\frac{m}{2m-2}\right) \cdot  \left(\frac{m-1}{2m-3}\right) \cdots\left(\frac{4}{m+2}\right) \cdot\left(\frac{3}{m+1}\right)}_{(m-2) \text{ many terms }}\text{ if }  m \geq 3.
  \end{cases}
  \end{equation}
Later,  it will 
 become clear later why $\mathfrak{c}_m$ is defined  in this specific way; for now, it is enough to know that $\mathfrak{c}_m$ only depends on $m$ and has the property that
\[
0 < \frakcm < 1.
\]
The censored version of the numerator remainder  $D_1$ is defined to be
\begin{equation} \label{D1x_upper_censored}
\bar{D}_{1, x} = D_1 I\left(|D_1| \leq \frac{\frakcm x}{4} \right) +\frac{\frakcm x}{4} I\left(D_1 >\frac{\frakcm x}{4}\right) - \frac{\frakcm x}{4} I\left(D_1 < - \frac{ \frakcm  x}{4}\right).
\end{equation}
For the denominator remainder, replacing certain $\xi_i$'s with $\xi_{b, i}$'s in \eqref{s_star_re_expr} we first define
\[
D_{2, V_b, \delta_1, \delta_{2,b}} =  d_n^2 (V_b^2  + \delta_1 + \delta_{2, b}) -1.
\]
where
 \begin{equation} \label{Vb2_and_delta2b_def}
V^2_b = V^2_{n, b} \equiv \sum_{i =1}^n \xi_{b, i}^2  \text{ and } \delta_{2, b} \equiv \frac{2 (n-1) }{(n-m)} {n-1 \choose m-1}^{-1} \sum_{i=1}^n \xi_{b, i}\Psi_{n,i}
\end{equation}
We further censor $\delta_1$ and $\delta_{2, b}$ as
\begin{equation} \label{bardelta1_def}
\bar\delta_{1} =  \delta_1 I(| \delta_1| \leq n^{-1/2}) + n^{-1/2}I( \delta_1 > n^{-1/2}) -  n^{-1/2}I( \delta_1 < -n^{-1/2})
\end{equation}
and
\[
 \bar{\delta}_{2, b} =  \delta_{2, b} I(| \delta_{2, b}| \leq 1) + I( \delta_{2, b} > 1) -  I( \delta_{2, b} < -1), 
\]
and define
\begin{equation} \label{DR_censored_def}
D_{2, V_b, \bardelta_1, \bardelta_{2,b}}  =  d_n^2( V_b^2 + \bar{\delta}_1  + \bar{\delta}_{2, b}) - 1.
\end{equation}
Finally, we  censor $D_{2, V_b, \bardelta_1, \bardelta_{2,b}}$ as 
\begin{multline} \label{barDR_censored_def}
\barD_{2, V_b, \bardelta_1, \bardelta_{2,b}}\equiv D_{2, V_b, \bardelta_1, \bardelta_{2,b}}  I\Bigg(\frac{9\frakc_m^2}{16}- 1 \leq D_{2, V_b, \bardelta_1, \bardelta_{2,b}}   \leq 1 \Bigg)\\
 +I\Bigg(D_{2, V_b, \bardelta_1, \bardelta_{2,b}} >1\Bigg) + \Big(\frac{9\frakc_m^2}{16}- 1 \Big)I\Bigg(D_{2, V_b, \bardelta_1, \bardelta_{2,b}}   < \frac{9\frakcm^2}{16}- 1\Bigg).
\end{multline}
 
With these censoring constructions, now we start to prove \eqref{simplified_goal}: First rewrite
\[
P(T_n > x) = P(W + D_1 > x(1 + D_2)^{1/2} )
\]
and 
define the   events:
\begin{align*}
\cE_1 &\equiv  \left\{W + D_1 > x(1 + D_2)^{1/2}  , \;\  |D_1|> \frac{\frakcm x}{4} \right\}   \bigcup \left\{W + D_1 > x(1 + D_2)^{1/2}  , \;\  D_2 < \frac{9\frakc_m^2}{16}- 1\right\} ;   \\
\cE_2&\equiv  \left\{W + \barD_{1,x} > x\Bigg(1 + \max\bigg( \frac{9\frakc_m^2}{16}- 1, D_2\bigg)\Bigg)^{1/2} , \;\max_{1 \leq i \leq n} |\xi_i|  > 1 \right\} ;\\
\cE_3&\equiv \left\{ W_b + \barD_{1,x} > x\Bigg(1 + \max\bigg( \frac{9\frakc_m^2}{16}- 1 , D_{2, V_b, \delta_1, \delta_{2, b}}\bigg)\Bigg)^{1/2} , \;\   |\delta_1|> \frac{1}{\sqrt{n}}\right\} \\
 &\hspace{2cm} \bigcup \left\{ W_b + \barD_{1,x} > x\Bigg(1 + \max\bigg(\frac{9\frakc_m^2}{16}- 1, D_{2, V_b, \delta_1, \delta_{2, b}}\bigg)\Bigg)^{1/2} , \;\   |\delta_{2, b}|> 1\right\} ; \\
\cE_4&\equiv \left\{ W_b + \barD_{1,x} > x\Bigg(1 + \max\bigg(\frac{9\frakc_m^2}{16}- 1, D_{2, V_b, \bardelta_1, \bardelta_{2, b}}\bigg)\Bigg)^{1/2}  , \;\   |D_{2, V_b, \bardelta_1, \bardelta_{2,b}}| > 1\right\}.
 \end{align*}
The following sequence of  inclusions are then seen to hold by progressively using 
\[
\tilde{\cE}_\ell \equiv \cup_{i=1}^{\ell} \cE_i , \quad \ell = 1, \dots, 4,
\]
 as covering events:
 \[
 \{ T_n > x\} \backslash \cE_1
 \subset \Bigg\{ P\Big(W + \barD_{1, x} > x\Big(1 + \max\Big(\frac{9\frakc_m^2}{16}- 1,  D_2\Big)\Big)^{1/2} \Big) \Bigg\} 
 \subset  \{ T_n > x \} \cup \cE_1
\]
\[
\downarrow
\]
 \[
 \{ T_n > x\} \backslash  \tilde{\cE}_2
 \subset \Bigg\{ P\Big(W_b + \barD_{1, x} > x\Big(1 + \max\Big( \frac{9\frakc_m^2}{16}- 1, D_{2, V_b, \delta_1, \delta_{2,b}}\Big)\Big)^{1/2} \Big) \Bigg\} 
 \subset  \{ T_n > x \} \cup   \tilde{\cE}_2
\]
\[
\downarrow
\]
 \[
 \{ T_n > x\} \backslash\tilde{\cE}_3
 \subset \Bigg\{ P\Big(W_b + \barD_{1, x} > x\Big(1 + \max\Big(\frac{9\frakc_m^2}{16}- 1, D_{2, V_b, \bardelta_1, \bardelta_{2,b}}\Big)\Big)^{1/2} \Big) \Bigg\} 
 \subset  \{ T_n > x \} \cup \tilde{\cE}_3
\]
\[
\downarrow
\]
 \begin{equation} \label{last_event_inclusion}
 \{ T_n > x\} \backslash \tilde{\cE}_4
 \subset \Bigg\{ P\Big(W_b + \barD_{1, x} > x\Big(1 +  \barD_{2, V_b, \bardelta_1, \bardelta_{2,b}} \Big)^{1/2} \Big) \Bigg\} 
 \subset  \{ T_n > x \} \cup \tilde{\cE}_4.
\end{equation}
The  last event inclusion \eqref{last_event_inclusion} implies the inequality
\begin{multline} \label{distill_ineq}
\Big| P(T_n > x) - \barPhi(x)\Big| \leq \\
 R_x 
  +  P\left(D_2 < \frac{9\frakc_m^2}{16}- 1\right) + 
\Big| P\Big(W_b + \barD_{1, x} > x\Big(1 +  \barD_{2, V_b, \bardelta_1, \bardelta_{2,b}}  \Big)^{1/2} \Big)- \barPhi(x)\Big|
 ,
\end{multline}
 where 
 \[
 R_x \equiv P\Bigg( W + D_1 > x(1 + D_2)^{1/2}  , \;\  |D_1|> \frac{\frakcm x}{4}  \Bigg) + \sum_{i =2}^4P(\cE_i) 
 \]
 is the sum of the probabilities of all covering events except 
 \[
 \Big\{ W + D_1 > x(1 + D_2)^{1/2}  , \;\  D_2 < \frac{9\frakc_m^2}{16}- 1  \Big\}
 \]
 whose probability can be bounded by $P(D_2 < \frac{9\frakc_m^2}{16}- 1)$. Hence, the proof boils down to proving bounds for the three terms on the right hand side of \eqref{distill_ineq}. We first state the bound for the  "$x$-dependent" term $R_x$.
 In light of the inequality:
\[
x \Bigg(1+ \Bigg(\frac{9 \frakcm^2}{16} - 1\Bigg) \Bigg)^{1/2}  - \bar{D}_{1, x}\geq x\Big(\frac{9\frakcm^2}{16}\Big)^{1/2} - \frac{\frakcm x}{4} =\frac{ \frakcm x}{2}
\]
which is true by the definition of $\barD_{1,x}$,
it can be seen that 
\begin{multline} \label{Rx_bdd}
R_x \leq 
P\left(|D_1| > \frac{\frakcm x}{4}  \right)  + 
P\left( W  \geq \frac{\frakcm x}{2}, \max_{1 \leq i \leq n} |\xi_i| >1\right) + 
P\left( W_b  \geq  \frac{\frakcm x}{2}, |\delta_1| > \frac{1}{\sqrt{n}}\right)   \\
 +
  P\left( W_b  \geq  \frac{\frakcm x}{2}, |\delta_{2, b} |> 1\right)   + 
 P\left( W_b  \geq  \frac{\frakcm x}{2}, | D_{2, V_b, \bardelta_1, \bardelta_{2,b}}  | > 1\right),
\end{multline}
leading us to the following bound:

\begin{lemma}[Nonuniform bound for $R_x$] \label{lem:bdd_for_remnants}
For $x\geq 1$, assuming  \eqref{zero_mean_kernel_assumption} and \eqref{sigma_equal_1_assumption}, we have the following bounds of rate no larger than $1/\sqrt{n}$:
\begin{enumerate} [label=(\roman*)]
\item $P\left(|D_1| > \frac{\frakcm x}{4}  \right)  
\leq \frac{C(m) \bE[|h|^3]}{\frakcm^3 n^{3/2}(1 +x^3)} $;
\item $P\left( W  \geq \frac{\frakcm x}{2}, \max_{1 \leq i \leq n} |\xi_i| >1\right)   
\leq  \frac{ C \bE[|g|^3]}{\frakcm^3(1 + x^3)\sqrt{n}}$;
\item $P\left( W_b  \geq  \frac{\frakcm x}{2}, |\delta_1| > \frac{1}{\sqrt{n}}\right)
\leq  C(m) e^{-\frakcm x/2} \frac{\|h\|_3^2}{\sqrt{n}}$;
 \item $P\left( W_b  \geq  \frac{\frakcm x}{2}, |\delta_{2, b} |> 1\right)
   \leq C(m) e^{-\frakcm x/2} \frac{\|g\|_3 \|h\|_3}{\sqrt{n}} $;
\item $
 P\left( W_b  \geq  \frac{\frakcm x}{2}, | D_{2, V_b, \bardelta_1, \bardelta_{2,b}}  | > 1 \right)  \leq
  Ce^{-\frakcm x/2} \Big( \frac{\bE[|g|^3] +   \|g\|_3 \|h\|_3}{\sqrt{n}} \Big)
  $.
\end{enumerate}
In particular, via \eqref{Rx_bdd} these bounds together imply
\begin{equation} \label{Rx_final_bdd}
R_x \leq  \frac{C_1(m)}{ (1 +x^3)} \Bigg( \frac{ \bE[|h|^3]}{n^{3/2}} + \frac{\bE[|g|^3]}{\sqrt{n}} \Bigg) + \frac{C_2(m)}{ e^{\frakcm x/2}}\Bigg( \frac{ \|h\|_3^2}{\sqrt{n}} \Bigg),
\end{equation}
 for some absolute constants $C_1(m), C_2(m) >0$.
\end{lemma}
 The proof of \lemref{bdd_for_remnants} in \appref{remnant_pf} follows fairly standard arguments, and we note that 
the proofs for $(iii)$-$(v)$ repeatedly use the Chernoff-type bound
$
I\Big(W_b \geq \frac{ \frakcm x}{2}\Big) \leq e^{W_b - \frac{\frakcm x}{2}}
$
to result in the exponentially nonuniform terms in $x$. Next, we bound the other $x$-dependent term $| P(W_b + \barD_{1, x} > x (1 +   \barD_{2, V_b, \bardelta_1, \bardelta_{2,b}}   )^{1/2} )- \barPhi(x)|$ from  \eqref{distill_ineq} in  \lemref{intermediate} below. Its proof in \appref{Stein_pf}  involves Stein's method, and is  largely similar to the proof of the uniform B-E bounds in   \citet{leungshaounif}, except that the properties of the solution to the Stein equation is more thoroughly exploited,  to ensure the nonuniformity in $x$ of the bound.

\begin{lemma}[Intermediate nonuniform bound by Stein's method] \label{lem:intermediate} For $x \geq 1$, there exist absolute constants $C(m) ,c(m) >0$ depending only on $m$ such that
\begin{equation} \label{intermediate_nonunif_bdd}
\bigg| P\bigg(W_b + \barD_{1, x} > x \big(1 +   \barD_{2, V_b, \bardelta_1, \bardelta_{2,b}}   \big)^{1/2} \bigg)- \barPhi(x)\bigg| \leq \frac{C(m) }{e^{c(m)x}} \Bigg( \frac{\bE[|g|^3] + \|g\|_3^2 \|h\|_3}{\sqrt{n}}\Bigg).
\end{equation}
\end{lemma}

To summarize \lemsref{bdd_for_remnants} and \lemssref{intermediate} in a nutshell: The delicate internal/external censoring operations applied to the terms $W$, $D_1$ and $D_2$ allow for  a desired nonuniform bound of rate $1/\sqrt{n} $ to be established for 
$
| P(W_b + \barD_{1, x} > x (1 +   \barD_{2, V_b, \bardelta_1, \bardelta_{2,b}}   )^{1/2} )- \barPhi(x)|
$
under minimal moment conditions, 
 while ensuring that a bound depending on $x$ can be established for $R_x$ by way of the inequality in \eqref{Rx_bdd}. We also remark that the crude censoring techniques in \citet{leungshaounif} are insufficient to prove a nonuniform bound since they would have severed the dependence on $x$.

With \eqref{distill_ineq}, \eqref{Rx_final_bdd} and \eqref{intermediate_nonunif_bdd}, to finish proving \eqref{simplified_goal} under \eqref{sigma_equal_1_assumption},
it remains to show, for a small  constant $c(m) >0$ depending only on $m$,  the exponential lower bound
\begin{equation} \label{exponential_lower_bdd_D2}
P\Big(D_2 < \frac{9\frakcm^2}{16}- 1 \Big) = P\Big(1+D_2  < \frac{9\frakcm^2}{16} \Big) \leq \exp\Bigg(-\frac{c(m)n }{ (\bE[|h|^3])^2} \Bigg)
\end{equation}
for $D_2$ in \eqref{s_star_re_expr}. 
The key observation is that, overall,  $\sigmahat^2$ can  be  understood as a  U-statistic constructed with  a \emph{non-negative} kernel.
 First,  write
\begin{align*}
\sum_{i=1}^n q_i^2
&= \sum_{i =1}^n \left( {n-1 \choose m-1}^{-1} \sum_{\substack{1 \leq i_1 < \cdots < i_{m-1}\leq n\\ i_l \neq i \text{ for } l = 1, \dots, m-1}}h(X_i, X_{i_1}, \dots, X_{i_{m-1}})\right)^2\\
&=  {n-1 \choose m-1}^{-2} \sum_{i = 1}^n 
\sum_{ \substack{1 \leq i_1 < \cdots < i_{m-1}\leq n\\ 1 \leq j_1 < \cdots < j_{m-1}\leq n \\ i_l ,j_l\neq i \text{ for } l = 1, \dots, m-1}}  h(X_i, X_{i_1}, \dots, X_{i_{m-1}})h(X_i, X_{j_1}, \dots, X_{j_{m-1}})   \\
&= {n-1 \choose m-1}^{-2} \left( \sum_{k = m}^{2m-1} \sum_{1 \leq i_1 < \dots < i_k\leq n}  \tilde{\cH}_k ( X_{i_1}, \dots, X_{i_k})\right)\\
& = {n-1 \choose m-1}^{-2} \sum_{1 \leq i_1 < \dots < i_{2m} \leq n} \tilde{\frakh}(X_{i_1},\dots, X_{i_{2m}}),
\end{align*}
where $\tilde{\cH}_k :\mathbb{R}^k \longrightarrow \mathbb{R}$  is a symmetric kernel of degree $k$ induced by $h(\cdot)$ defined as
\begin{multline} \label{tilde_cH_k}
\tilde{\cH}_k (x_1, \dots, x_k) \equiv  (2m  - k) \times \sum_{\substack{\cS_1, \cS_2, \cS_3 \subset \{1, \dots, k\}: \\ |\cS_1| = 2m- k \\ |\cS_2| =  |\cS_3|=  k - m  \\  \cS_1, \cS_2, \cS_3 \text{disjoint}  }}     h(x_{\cS_1}, x_{\cS_2}) h(x_{\cS_1}, x_{\cS_3}) ,\\
\text{ for each } k = m, \dots, 2m-1,
\end{multline}
and $\tilde{\frakh}$ is  
 the symmetric kernel of degree $2m$ further derived from \eqref{tilde_cH_k} defined as
\begin{equation} \label{tilde_frak_h}
\tilde{\frakh} (x_1,\dots, x_{2m}) \equiv \sum_{k=m}^{2m-1} {n - k\choose 2m  - k}^{-1}\sum_{1 \leq l_1 < \dots < l_k \leq 2m} \tilde{\cH}_k(x_{l_1},\dots, x_{l_k}).
\end{equation}
Next, upon expansion,
\begin{align*}
U_n^2 &= {n \choose m}^{-2} \sum_{ \substack{1 \leq i_1 < \cdots < i_m\leq n\\ 1 \leq j_1 < \cdots < j_m\leq n }}  h(X_{i_1}, \dots, X_{i_m})h(X_{j_1}, \dots, X_{j_m}) \\
&= {n \choose m}^{-2} \left( \sum_{k = m}^{2m} \sum_{1 \leq i_1 < \dots < i_k\leq n}  \breve{\cH}_k ( X_{i_1}, \dots, X_{i_k})\right)\\
& = {n \choose m}^{-2} \sum_{1 \leq i_1 < \dots < i_{2m} \leq n} \breve{\frakh}(X_{i_1},\dots, X_{i_{2m}}),
\end{align*}
where  $\breve{\cH}_k :\mathbb{R}^k \longrightarrow \mathbb{R}$  is a symmetric kernel of degree $k$ induced by $h(\cdot)$ defined as
\begin{equation} \label{breve_cH_k}
\breve{\cH}_k (x_1, \dots, x_k) \equiv  \sum_{\substack{\cS_1, \cS_2, \cS_3 \subset \{1, \dots, k\}: \\ |\cS_1| = 2m- k \\ |\cS_2| =  |\cS_3|=  k - m  \\  \cS_1, \cS_2, \cS_3 \text{disjoint}  }}     h(x_{\cS_1}, x_{\cS_2}) h(x_{\cS_1}, x_{\cS_3}) \text{, for each } k = m, \dots, 2m,
\end{equation}
and $\breve{\frakh}$ is  
 the symmetric kernel of degree $2m$ further derived from \eqref{breve_cH_k} defined as
\begin{equation} \label{breve_frak_h}
\breve{\frakh} (x_1,\dots, x_{2m}) \equiv \sum_{k=m}^{2m} {n - k\choose 2m  - k}^{-1}\sum_{1 \leq l_1 < \dots < l_k \leq 2m} \breve{\cH}_k(x_{l_1},\dots, x_{l_k}).
\end{equation}
With the above expressions for $\sum_{i=1}^n q_i^2$ and $U_n^2$ both as U-statistics of degree $2m$, from \eqref{real_studentizer}, one can write 
\begin{equation} \label{sigmahat_as_u_stat}
\sigmahat^2 = A(m, n)   \frac{\sum_{1 \leq i_1 < \dots < i_{2m} \leq n} \mathfrak{h}(X_{i_1},\dots, X_{i_{2m}}) }{{n \choose 2m}},
\end{equation}
where 
\begin{equation} \label{Anm_def}
A(n, m) \equiv    \frac{n-1}{(n-m)^2 (n-2m+1)} {n \choose 2m} {n-1 \choose m-1}^{-2}  
\end{equation}
and
\begin{equation}\label{frak_h_def}
\frakh (x_1, \dots, x_{2m}) \equiv (n-2m+1) \Bigg\{  \tilde{\frakh} (x_1,\dots, x_{2m}) - \frac{m^2}{n} \breve{\frakh}(x_1, \dots, x_{2m})\Bigg\};
\end{equation}
Hence, up to the multiplicative factor $A(n,m)$, $\sigmahat^2$ is  a U-statistic of degree $2m$.
Moreover, 
 it is not hard to see that 
\begin{equation} \label{non_negative_kernel_property}
\frakh(x_1, \dots, x_{2m})\geq 0 \text{ for all values of $x_1, \dots, x_{2m}$};
\end{equation}
when $n = 2m$, from the original definition of $\sigmahat^2$ in \eqref{real_studentizer}, it is seen that,  irrespective of the values of  $X_1, \dots, X_{2m}$,
\[
 A(m, n) \frakh(X_1, \dots, X_{2m}) = {\sigmahat}^2 =  \frac{n-1}{(n-m)^2} \sum_{i=1}^n (q_i- U_n)^2  \geq 0,
\]
so $\frakh$ can only take on non-negative value since $A(m, n)  > 0$.

With the insights above, we are primed to leverage the following exponential lower tail bound for non-negative kernel U-statistics to develop the exponential bound in \eqref{exponential_lower_bdd_D2}. This result is of independent interest, and it naturally extends a known   exponential lower tail bound for a sum of independent non-negative variables in the literature \citep[Theorem 2.19]{victor2009self}; surprisingly, we could not locate a result similar to \lemref{chernoff_lower_tail_bdd_U_stat} elsewhere. Its proof is included in \appref{otherproofs}, which uses a well-known trick by  \citet{hoeffding1963probability}.

\begin{lemma}[Exponential lower tail bound for U-statistics with non-negative kernels] \label{lem:chernoff_lower_tail_bdd_U_stat}
Assume that $U_n = {n \choose m}^{-1} \sum_{1\leq i_1<\dots <i_m \leq n} h(X_{i_1}, \dots, X_{i_m})$ is a U-statistic of degree $m$, and $h : \cX^m \longrightarrow \bR_{\geq0}$ can only take non-negative values, with the property that $\bE[h^p(X_1, \dots, X_m)] < \infty$ for some $p \in (1, 2]$. Then for $0 < x \leq  \bE[h]$,
\[
P(U_n \leq x) \leq \exp\left(\frac{- [ n/m ] (p-1)  (\bE[h]  -x)^{p/(p-1)}}{p (\bE[h^p])^{1/(p-1)} }\right),
\]
where $[n/m]$ is defined as the greatest integer less than $n/m$.
\end{lemma}

Since
\[
  \frac{n ((n -m - 1)!)^2}{(n-2)! (n - 2m +1)!}  = 
  \begin{cases} 
& \frac{n}{n-m} \text{ if } m = 1 \text{ or } 2;\\
 &\frac{n}{n-2}\times \underbrace{\big(\frac{n-m-1}{n-3} \big) \cdot \big(\frac{n-m-2}{n-4} \big)\cdots \big(\frac{n-2m+2}{n-m}\big)}_{m-2 \text{ many terms}}
  \text{ if }  m \geq 3,
  \end{cases}
\] 
in light of the assumption that $n > \max(2, m^2)$ (which implies $n > 2m$) and the definition of  $\mathfrak{b}_m$ in \eqref{frakbm_def},  one can derive the lower bound
\begin{equation} \label{lower_bdd_Amn}
A(n, m) = \frac{\{ (m-1)!\}^2}{(2m)!} \frac{n((n-m-1)!)^2}{ (n-2)! (n-2m +1)!}\geq \frac{\{ (m-1)! \}^2}{(2m)!} \mathfrak{b}_m .
\end{equation}
Moreover, because  for any disjoint subsets $\cS_1, \cS_2, \cS_3 \subset \{1, \dots, k\}$ such that $|\cS_1| = 2m-k$, $|\cS_2| = |\cS_3| = k-m$ , 
\[
 \bE[h(X_{\cS_1}, X_{\cS_2}) h(X_{\cS_1}, X_{\cS_3}) ] = \bE[\bE[ h(X_1, \dots, X_m)|X_{\cS_1}]^2 ] = \bE[ h_{2m-k}^2 (X_1, \dots, X_{2m-k})],
\]
we have
\begin{equation*}
\bE[ \tilde{\cH}_k  ]  =  (2m-k) {k \choose 2m- k} { 2k - 2m \choose k - m} \bE[h_{2m-k}^2]  \text{ and }
  \bE[ \breve{\cH_k}]  =  {k \choose 2m- k} { 2k - 2m \choose k - m} \bE[h_{2m-k}^2].
\end{equation*}  
As such, the expectation of $\frakh$ can be computed as
\begin{align}
&\bE[\frakh] \notag\\
 &=   (n-2m+1)  \Bigg\{  \bE[\tilde{\frakh}] - \frac{m^2}{n} \bE[\breve{\frakh}]\Bigg\} \notag \\
 &= (n-2m+1) \Bigg\{    \sum_{k=m}^{2m-1} {n - k\choose 2m  - k}^{-1} {2m \choose k} \bE[ \tilde{\cH}_k  ] - \frac{m^2}{n}  \sum_{k=m}^{2m} {n -k \choose 2m- k}^{-1} {2m \choose k} \bE[ \breve{\cH_k}] \Bigg\} \notag \\
&=(n-2m+1)\sum_{k=m}^{2m-1} {n -k \choose 2m- k}^{-1} {2m \choose k} {k \choose 2m- k} { 2k - 2m \choose k - m} \Big(2m - k - \frac{m^2}{n}\Big)  \bE[h_{2m-k}^2] \notag \\
&= (n-2m+1) \sum_{k = m}^{2m-1}  \frac{(2m)!}{\{(2m-k)!(k-m)!\}^2}{n-k \choose 2m-k}^{-1}   \Big(2m - k - \frac{m^2}{n}\Big)  \bE[h_{2m-k}^2] \notag \\
&=  \Big(1 - \frac{m^2}{n}\Big)    \frac{(2m)!}{\{(m-1)!\}^2} +  \sum_{k = m}^{2m-2}  \frac{(2m)! (n-2m+1)}{\{(2m-k)!(k-m)!\}^2}{n-k \choose 2m-k}^{-1}   \Big(2m - k - \frac{m^2}{n}\Big)  \bE[h_{2m-k}^2] \label{hfrak_explicit_mean}
,
\end{align}
where the third equality uses that $\bE[h_0^2] = \bE[h(X_1, \dots, X_m) h(X_{m+1}, \dots, X_{2m})] = 0$, and 
 the last equality uses $\bE[h_1^2] = \bE[g^2] = 1$. Under our assumption $n > m^2$, because the quantities $(1 - m^2/n)$ and $(2m - k - m^2/n)$ for all $k = m, \dots, 2m-2$ are positive, all the summands in \eqref{hfrak_explicit_mean} are positive. In particular, this implies 
 \begin{equation}\label{kernel_mean_lower_bound}
\bE[\frakh]
 \geq  \frac{(2m)!}{\{(m-1)!\}^2} \Big( 1 - \frac{m^2}{n}\Big) \geq  \frac{(2m)!}{\{(m-1)!\}^2} \Big( 1 - \frac{m^2}{m^2 +1}\Big).
 \end{equation}
 Hence, with the lower bound for $A(m,n)$ in \eqref{lower_bdd_Amn}, 
  \begin{align}
P\Big({ \sigmahat}^2 \leq \frac{9\frakc_m^2}{16}\Big) 
&\leq  P\Bigg(\frac{\sum_{1 \leq i_1 < \dots < i_{2m} \leq n} \mathfrak{h}(X_{i_1},\dots, X_{i_{2m}}) }{{n \choose 2m}} \leq \frac{9}{16} \cdot  \frac{(2m)!}{ \{(m-1)!\}^2} \Big( 1 - \frac{m^2}{m^2 +1}\Big) \Bigg) \notag\\
&\leq \exp\left( -\frac{[n/m] \Big\{   \frac{7(2m)!}{16\{(m-1)!\}^2}  \Big(1 - \frac{m^2}{m^2+1}\Big)  \Big\}^3 }{ 3(\bE[\frakh^{3/2}])^2}\right) \label{sigmahatstar_sq_lower_tail_bdd}
 \end{align}
 where the last  inequality comes from applying \lemref{chernoff_lower_tail_bdd_U_stat} to  ${n \choose 2m}^{-1} \sum \frakh$ by taking 
 $x =  \frac{9(2m)!}{16\{(m-1)!\}^2} (1- \frac{m^2}{m^2 +1})$ and $p = 3/2$, using the kernel non-negativity in \eqref{non_negative_kernel_property} and  the kernel mean lower bound in \eqref{kernel_mean_lower_bound}. 
 The following  moment bound for centered U-statistics proved in \appref{otherproofs} can  be used to further understand $\bE[|\frakh|^{3/2}]$.

\begin{lemma}[General moment bound of U-statistics]\label{lem:non_integral_moment_bdd_u_stat}
Suppose $h(x_1, \dots, x_m)$ is a real-valued  symmetric kernel, with $\bE[h(X_1, \dots, X_m)] = 0$ and $\bE[|h(X_1, \dots, X_m)|^p] < \infty$ for some $p \in [1, \infty)$.  Let $r \geq 1$ be the order of degeneracy  of $U_n$, i.e.  $r$ is the first integer for which, as functions,
\[
h_k(x_1, \dots,x_k) = 0\text{ for } k = 1, \dots, r-1, \text{ and } \quad h_r(x_1, \dots, x_r) \neq 0.
\]
For  positive constants $C(m, r, p) >0$, we have
\begin{equation} \label{simple_u_stat_moment_bdd}
 \bE[|U_n|^p] \leq  
  \begin{cases} 
\frac{ C(m, r, p) \bE[|h|^p]}{ n^{r(p -1)}} & \text{if } 1 \leq p \leq 2; \\
\frac{C(m, r, p) \bE[|h|^p]}{n^{(rp)/2}} & \text{if }  p \geq 2.
  \end{cases}
\end{equation}
\end{lemma}

 With \lemref{non_integral_moment_bdd_u_stat}, by the Cauchy inequality 
\[
 \bE[|h(X_{\cS_1}, X_{\cS_2}) h(X_{\cS_1}, X_{\cS_2})|^{3/2}] \leq  \sqrt{\bE[|h(X_{\cS_1}, X_{\cS_2}) |^3]}  \sqrt{\bE[|h(X_{\cS_1}, X_{\cS_3}) |^3]} = \bE[|h|^3] , 
\]
for any $\cS_1, \cS_2, \cS_3 \subset \{1, \dots, k\}$, 
one can derive
\begin{align*}
&\bE[|\frakh|^{3/2}]\\
 &\leq C n^{3/2} (\bE[|\tilde{\frakh}|^{3/2}] + \frac{m^3}{n^{3/2}}\bE[|\breve{\frakh}|^{3/2}])\\
 &\leq C(m)n^{3/2} \Bigg( \sum_{k=m}^{2m-1} {n -k \choose 2m-k}^{-3/2} \bE[|\tilde{\cH}_k|^{3/2}] + n^{-3/2} \sum_{k = m}^{2m} {n -k \choose 2m-k}^{-3/2} \bE[|\breve{\cH}_k|^{3/2}]  \Bigg)\\
&\leq C(m) n^{3/2}  \Bigg( \sum_{k=m}^{2m-1} {n -k \choose 2m-k}^{-3/2} \bE[|h|^3] + n^{-3/2} \sum_{k = m}^{2m} {n -k \choose 2m-k}^{-3/2} \bE[|h|^3]  \Bigg)\\
&= C(m) \bE[|h|^3], 
\end{align*}
which can then conclude \eqref{exponential_lower_bdd_D2}  from \eqref{sigmahatstar_sq_lower_tail_bdd}.

\subsection{How the constants scale in $m$} \label{sec:remark1}
The special constants $\frakcm$ and $\frakbm$ defined in \eqref{frakcm_def} and \eqref{frakbm_def} play critical roles in arriving at our lower tail bound for $\hat{\sigma}^2$ in \eqref{sigmahatstar_sq_lower_tail_bdd}, which ultimately induces the correction term with exponential decay in $n$, in the final B-E bounds \eqref{BE_Tn_one_sample} and \eqref{BE_Tn_one_sample_simpler}. For $m \geq 3$, by  Stirling's formula, we get that
\[
\frakbm = \frac{(m!)^2}{2 \cdot (2m-2)!}  \sim \frac{2\sqrt{\pi} }{e^2 4^{m}} \cdot \Big(\frac{m}{m-1}\Big)^{2m} \cdot m \cdot(m-1)^{3/2},
\]
where "$\sim$" means their ratio tends to $1$ as $m \rightarrow \infty$. As such, $\frakbm$, and therefore $\frakcm$, decays exponential in $m$, due to the factor $4^m$ on the right hand side of the prior display. By observing where the constant $\frakcm$ figures in the inequalities of \lemref{bdd_for_remnants}$(i) - (v)$, our proof methodology indicates that the big constants appearing in \eqref{BE_Tn_one_sample} and \eqref{BE_Tn_one_sample_simpler}, denoted by $C(m)$, could potentially grow exponentially in $m$!   This is in stark contrast to the B-E bounds of the \emph{standardized} U-statistics in \eqref{standardized_ustat_unif_BE} and \eqref{standardized_ustat_nonunif_BE}, where the constants are known to only scale like $\sqrt{m}$ \citep[Theorem 3.1]{chen2007normal}. It is not clear to us whether there could be a different proof that can bring down the order of these constants in $m$. Note that,  $\frakbm$ and $\frakcm$ are direct by-products of the factor $A(n, m)$  from analyzing the Studentizer $\hat{\sigma}$ with the tight exponential lower tail bound of \lemref{chernoff_lower_tail_bdd_U_stat}, and $A(n, m)$ is in itself intrinsic to the structure of the Studentizer $\hat{\sigma}$ as seen in \eqref{sigmahat_as_u_stat}. As such, the possible exponential dependence on $m$ of our constants in \thmref{main} could well be a unique, perhaps undesirable, nature of Studentized U-statistics.  

\subsection{The required sample size  relative to $m$}  \label{sec:remark2}

 In our proof above, we have effectively used the assumed condition $m^2 < n$ to demonstrate that all the  summands  in \eqref{hfrak_explicit_mean} are positive, and then established a positive lower bound for the expectation of the kernel $\frakh$ of $\hat{\sigma}^2$ in \eqref{kernel_mean_lower_bound}; this gives way to using \lemref{chernoff_lower_tail_bdd_U_stat} to establish the exponential lower bound in  \eqref{sigmahatstar_sq_lower_tail_bdd}.

 To  weaken the condition to the more typical $2m < n$ assumed for the uniform B-E bound in \thmref{BE_unif},  a possible avenue is to first establish a nonuniform B-E bound for the \emph{self-normalized U-statistic}
\begin{equation}\label{sn_ustat}
T_n^*  \equiv  \frac{\sqrt{n}}{m \sigmahatstar} U_n,
\end{equation}
where ${\sigmahatstar}^2 \equiv \frac{n-1}{(n-m)^2} \sum_{i=1}^n q_i^2$,
 i.e. establishing a bound of the form
\begin{equation} \label{nonunif_BE_selfnorm_ustat}
|P(T_n^* \leq x) - \Phi(x)| \leq   \exp\Bigg(-\frac{c(m)n \sigma^6 }{ (\bE[|h|^3])^2} \Bigg) + \frac{C(m)\bE[|h|^3]}{(1+ |x|^3)\sqrt{n} \sigma^3} \text{ for } x \in \mathbb{R}
\end{equation}
analogous to \eqref{BE_Tn_one_sample_simpler}, by employing a similar  strategy to how our current B-E bound for the Studentized $T_n$ was established, in which case an exponential lower bound for ${\sigma^*}^2$ analogous to \eqref{sigmahatstar_sq_lower_tail_bdd} has to be established by using  \lemref{chernoff_lower_tail_bdd_U_stat}. In some unreported calculations, we found that the weaker assumption $2m< n$  suffices to derive the said exponential lower bound. To leverage \eqref{nonunif_BE_selfnorm_ustat} as a "bridge" to establish the nonuniform bound for the Studentized U-statistic $T_n$ in  \eqref{BE_Tn_one_sample_simpler}, one can then potentially exploit the well-known equity of the events 
\begin{equation}\label{equity_of_events}
\{ T_n > x\} = \{ T_n^* > x b_{m, n}(x) \} \text{ for any } x \geq 0
\end{equation}
that results from  the algebraic relationship
\begin{equation} \label{algebraic_relationship_ustat}
T_n = \frac{T_n^*}{\Big(1 - \frac{m^2 (n-1)}{(n-m)^2} {T_n^*}^2\Big)^{1/2} },
\end{equation}
 where we have defined \[
b_{m, n}(x) \equiv \bigg(1 + \frac{m^2 (n-1) x^2}{(n-m)^2} \bigg)^{-1/2};
\]
we note in passing that the relationship in \eqref{algebraic_relationship_ustat} is analogous to the relationship between  the self-normalized sum $S_n/V_n$ and  $T_{student}$ in \eqref{efron_relationship}, and has been used in \citet{lai2011cramer} and \citet{shao2016cramer} to establish  \emph{\cramer-type} moderate deviation results for Studentized U-statistics. 
From \eqref{nonunif_BE_selfnorm_ustat} and \eqref{equity_of_events}, one can write 
\begin{equation}  \label{last_relevant_term}
|P(T_n \leq  x) - {\Phi}(x) | 
\leq  \exp\Bigg(-\frac{c(m)n \sigma^6 }{ (\bE[|h|^3])^2} \Bigg)+ |\bar{\Phi}(xb_{m,n}(x)) + \bar{\Phi}(x)|  + \frac{C(m)\bE[|h|^3]}{(1+ |xb_{m,n}(x)|^3)\sqrt{n} \sigma^3},
\end{equation}
and further bound the last two terms on the right; 
without loss of generality, we can focus on the range $x \geq 0$. The  term  $|\bar{\Phi}(xb_{m,n}(x)) + \bar{\Phi}(x)|$ is quite easy to bound, but we skip the details and refer to \secref{anotheranothersubsection} for similar arguments used to handle an analogous quantity for the t-statistic, where we prove \thmref{nonunif_BE_sn_sum}. 
However, a bottleneck arises when attempting to bound the last term in \eqref{last_relevant_term}: Under $2m < n$ where one has
$
 b_{m, n}(\sqrt{n}) = (1 + \frac{m^2 (n-1) n}{(n-m)^2} )^{-1/2}
  \geq  (1 + \frac{m^2 n^2}{(n-m)^2} )^{-1/2} 
 \geq  (1 + 4m^2)^{-1/2}$,
while the nonuniform multiplicative factor in $x$   is seen to be such that 
\[
\frac{1}{1 + (x b_{m, n}(x))^3} \leq \frac{1}{1 + (x b_{m, n}(\sqrt{n}))^3} \leq \frac{1}{1 + (x (1 + 4m^2)^{-1/2})^3} \text{ for } 0 \leq x \leq \sqrt{n},
\]
 the factor doesn't vanish  as $x \rightarrow \infty$ because $\lim_{x \rightarrow \infty}  x b_{m, n}(x) = \frac{n-m}{m \sqrt{n-1}}$. This means, for the range $x \geq \sqrt{n}$, one has to show that the absolute difference $|P(T_n^* > xb_{m,n}(x)) - \bar{\Phi}(xb_{m,n}(x)) |$ is no larger than our exponential correction factor $\exp\big(-\frac{c(m)n \sigma^6 }{ (\bE[|h|^3])^2} \big)$, perhaps up to  an  absolute multiplicative constant in $m$. We believe this is possible since both $\bar{\Phi}(xb_{m,n}(x))$ and $P(T_n^* > xb_{m,n}(x))$  are expected to be small for $x \geq \sqrt{n}$.  By a standard upper bound of the normal survival function $\barPhi(\cdot)$ \citep[p.16, $(2.11)$]{chen2011normal} and the fact that $x b_{m, n}(x)$ is increasing in $x \geq 0$,
\begin{align*}
\bar{\Phi}(x b_{m,n}(x))  &\leq \min \Bigg(\frac{1}{2} , \frac{1}{xb_{m,n}(x) \sqrt{2 \pi}} \Bigg) \exp \Bigg( - \frac{(xb_{m,n}(x))^2 }{2}\Bigg)\\
&\leq  \   \exp \Bigg( -  \frac{n}{2}  \Bigg( 1 +  \frac{m^2 (n-1)n}{(n-m)^2}\Bigg)^{-1} \Bigg)  \quad \text{ for }  x \geq \sqrt{n},
\end{align*}
which has the desired exponential rate of decay in $n$. Our intuition is that the term $P(T_n^* > xb_{m,n}(x))$ is also expected to have some form of exponential decay in $x$ to induce an exponentially decaying term in $n$, but important Hoeffding-type inequalities  comparable to those available for the self-normalized sum are missing in the literature; see \lemref{subGauss_sn_sum} below. A quest for such inequalities is an important problem that deserves independent investigation. 


\section{Proof of the nonuniform B-E bound for Student's t-statistic} \label{sec:sn_sum_pf}

We now prove the refined nonuniform bound for the Student's t-statistic and self-normalized sum in \thmref{nonunif_BE_sn_sum}.
It suffices to consider $x \geq 0$, whether we are aiming to establish the theorem for the self-normalized sum $S_n/V_n$ or the t-statistic $T_{student}$, otherwise one can  replace the $X_i$'s with $- X_i$'s instead. 
 A technical tool that we will use is the following Hoeffding-type bound  that can be found in \citet[Theorem 2.16, p.12]{victor2009self}:
 
\begin{lemma}[Sub-Gaussian property for self-normalized sums] \label{lem:subGauss_sn_sum}
Under the same assumptions as  \thmref{nonunif_BE_sn_sum}, it is true that, for any $x \geq 0$,
\[
P\Big(S_n > x (4 \sqrt{n} \|X_1\|_2 +V_n) \Big) \leq 2 e^{-x^2/2}.
\]
\end{lemma}

\subsection{Proof for the self-normalized sum}
We will first prove  a more general bound for the self-normalized sum:
\begin{equation} \label{general_sn_sum_be_bdd}
|P(S_n / V_n > x) - \barPhi(x)| \leq
  \begin{cases} 
&C \frac{(1 + x)^{2}}{e^{x^2/2}} \frac{ \bE[|X_1|^3]}{\sqrt{n} (\bE[X_1^2])^{3/2}} \text{ for } 0 \leq x \leq n^{1/6} \frac{\|X_1\|_2 }{\|X_1\|_3}; \\
 & \exp \Bigg( \frac{-   n ( \bE[X_1^2])^3}{16 (\bE[X_1^3])^2} \Bigg) +  2 \exp \Bigg( -\frac{x^2}{162} \Bigg)   \text{ for } x > n^{1/6} \frac{\|X_1\|_2 }{\|X_1\|_3}.
  \end{cases}
\end{equation}

Now we prove \eqref{general_sn_sum_be_bdd}.
As a simple consequence of \emph{\cramer-type moderate deviation} for self-normalized sums \citep[Theorem 2.3]{jing2003self}, one can derive the nonuniform B-E bound
\begin{equation} \label{consequence_of_cramer}
|P(S_n / V_n > x) - \barPhi(x)| \leq  C \frac{(1 + x)^{2}}{e^{x^2/2}} \frac{ \bE[|X_1|^3]}{\sqrt{n} (\bE[X_1^2])^{3/2}} \text{ for } 0 \leq x \leq n^{1/6} \frac{\|X_1\|_2 }{\|X_1\|_3};
\end{equation}
see \citet[Eqn. $(2.11)$, p.2171]{jing2003self}.
For  any $x > n^{1/6} \|X_1\|_2/\|X_1\|_3$, 
\begin{align}
P\bigg(S_n / V_n  > x\bigg) 
 &\leq  P\bigg( V_n\leq \sqrt{n} \|X_1\|_2/2 \bigg) + P(S_n > x V_n, V_n > \sqrt{n} \|X_1\|_2/2) \notag\\
&\leq P\Bigg(\frac{V_n^2}{n} \leq \frac{ \bE[X_1^2]}{4} \Bigg) + P\Bigg(S_n > \frac{x(4 \sqrt{n} \|X_1\|_2 + V_n ) }{9} \Bigg) \notag\\
&\leq \exp \Bigg( \frac{-\frac{n}{2} (\frac{3}{4} \bE[X_1^2])^3}{1.5 (\bE[X_1^3])^2} \Bigg) + 2 \exp \Bigg( -\frac{x^2}{162} \Bigg) \text{ by \lemsref{chernoff_lower_tail_bdd_U_stat} and \lemssref{subGauss_sn_sum}} \notag\\
&\leq \exp \Bigg( \frac{- n (\bE[X_1^2])^3}{16(\bE[X_1^3])^2} \Bigg) + 2 \exp \Bigg( -\frac{x^2}{162} \Bigg)  \label{consequence_of_sub_gauss_sn}.
\end{align}
Moreover,  by the standard normal tail bound,
\begin{equation} \label{gauss_tail_bdd_consequence}
\barPhi(x) \leq  \frac{1}{2}e^{-x^2/2} \leq 2 \exp \Bigg( -\frac{x^2}{162} \Bigg)
\end{equation}
Combining \eqref{consequence_of_cramer} for $0 \leq x \leq n^{1/6} \|X_1\|_2 /\|X_1\|_3$ along with \eqref{consequence_of_sub_gauss_sn} and \eqref{gauss_tail_bdd_consequence} for $x > n^{1/6} \|X_1\|_2/\|X_1\|_3$, we get that the bound \eqref{general_sn_sum_be_bdd} for the self-normalized sum.

Lastly, 
the term $ 2 e^{-x^2/162} $ in the bound from \eqref{general_sn_sum_be_bdd} can be bounded as 
\begin{multline} \label{lucky}
2 e^{-x^2/162}  = 2e^{- x^2/324} \underbrace{ \exp\left(- \frac{ x^2}{324} + 3 \log (x)\right)}_{\leq C} x^{-3} \\
 \leq C e^{-cx^2}  \frac{\bE[|X_1|^3]}{\sqrt{n} (\bE[X_1^2])^{3/2}} \text{ for } x > n^{1/6} \frac{\|X_1\|_2 }{\|X_1\|_3}.
\end{multline}
Combining \eqref{general_sn_sum_be_bdd} and \eqref{lucky} and suitably adjusting the absolute constants, we get the desired bound 
\begin{equation} \label{desired_bdd_sn_sum}
|P(S_n / V_n > x) - \barPhi(x)| \leq  \frac{C}{e^{cx^2}} \frac{ \bE[|X_1|^3]}{\sqrt{n} (\bE[X_1^2])^{3/2}} + \exp \Bigg( \frac{-   n ( \bE[X_1^2])^3}{16 (\bE[X_1^3])^2} \Bigg)   \text{ for all } x\geq 0. 
\end{equation}
(In particular, this means, for the self-normalized sum, the constant $C_1$ in  \thmref{nonunif_BE_sn_sum} can be simply taken to be $1$.)

\subsection{Proof for the t-statistic}

To prove the theorem for $T_{student}$, we will adapt a "bridging" argument found in  \citet{jing1999exponential}: Define the function
\[
a_n(x) = a_{n,x}\equiv \Bigg(\frac{n }{n + x^2 - 1} \Bigg)^{1/2},
\]
which has the property that
\begin{equation} \label{lower_upper_bdd_anx}
1/\sqrt{2} \leq \anx \leq \sqrt{2} \text{ for } 0 \leq x \leq \sqrt{n},
\end{equation}
considering that $n \geq 2$. Moreover, the function $x a_n(x)$ is increasing in $x$ because 
\begin{equation} \label{ascending_prop_sn}
\frac{d}{dx} x a_n(x) =  \bigg(\frac{n}{n+x^2 -1}\bigg)^{1/2} \bigg(1- \frac{x^2}{n + x^2 -1}\bigg) >0 
\end{equation}
for $n \geq 2$. 
 Using  the  well-known algebraic relationship in \eqref{efron_relationship}, we have the event equivalence
\[
\Big\{T_{student} > x \Big\} = \Bigg\{   \frac{S_n}{V_n} > x  a_n(x)\Bigg\} \text{ for any } x \geq 0.
\]
Then, by the triangular inequality we have
\begin{equation} \label{tri_ineq_bridge}
|P(T_{student} > x) - \barPhi(x)| \leq |P(S_n/V_n >  x a_n(x)) - \barPhi(x a_n(x))| + |\barPhi(x a_n(x)) - \barPhi(x)|.
\end{equation}

\ \

\subsubsection{Bounding $|P(S_n/V_n >  x a_n(x)) - \overline{\Phi}(x a_n(x))| $ }

From \eqref{lower_upper_bdd_anx}, it must be true that for any small constant $c >0$,
\[
\frac{(1 + x a_{n,x})^{2}}{e^{c(x a_{n,x})^2}} \leq \frac{ (1 + \sqrt{2} x)^2}{ e^{cx^2/2}}\text{ for } 0 \leq x \leq \sqrt{n},
\]
which also implies
 \begin{multline} \label{something1}
\big|P( S_n/V_n > x a_{n, x}) - \barPhi( x a_{n, x})\big| \leq  
\exp \Bigg( \frac{-   n ( \bE[X_1^2])^3}{16 (\bE[X_1^3])^2} \Bigg) + C \frac{(1 + x)^{2}}{e^{cx^2}} \frac{ \bE[|X_1| ^3]}{\sqrt{n} (\bE[X_1^2])^{3/2}}\\
 \text{ for } 0 \leq x \leq \sqrt{n}
\end{multline}
from \eqref{desired_bdd_sn_sum} by adjusting the constants $C, c$. Since $x a_{n,x}$ is increasing by \eqref{ascending_prop_sn}, 
\begin{equation} \label{something}
x a_{n,x} \geq  \sqrt{n} a_n(\sqrt{n}) =  \frac{n}{\sqrt{2n-1}}\text{ for } x \geq \sqrt{n}.
\end{equation}
As $n/\sqrt{2n-1} \geq n^{1/6} \|X_1\|_2/\|X_1\|_3$, we can then apply \eqref{general_sn_sum_be_bdd}  to get
\begin{align}
&\big|P( S_n/V_n > x a_{n, x}) - \barPhi( x a_{n, x})\big| \notag \\
&\leq \exp \Bigg( \frac{-  n ( \bE[X_1^2])^3}{ 16(\bE[X_1^3])^2} \Bigg) +   2 \exp \Bigg( -\frac{x^2 a_{n, x}^2}{162} \Bigg)    \notag\\
&\leq  \exp \Bigg( \frac{-  n ( \bE[X_1^2])^3}{ 16(\bE[X_1^3])^2} \Bigg) +  2 e^{- 162^{-1} n^2/(2n-1)} \text{ for } x \geq \sqrt{n}, \text{ by } \eqref{something}
\label{somethingsomething}
\end{align}
Combining \eqref{something1} and \eqref{somethingsomething}, as well as $(\bE[X_1^2])^3/ (\bE[X_1^3])^2 \leq 1$, upon adjusting the absolute constants we have
 \begin{multline}  \label{pre_bridge}
\big|P( S_n/V_n > x a_{n, x}) - \barPhi( x a_{n, x})\big| \leq  
C_1\exp \Bigg( \frac{-  c_1 n ( \bE[X_1^2])^3}{ (\bE[X_1^3])^2} \Bigg) +  \frac{C_2}{e^{c_2x^2}} \frac{ \bE[|X_1| ^3]}{\sqrt{n} (\bE[X_1^2])^{3/2}} \\
\text{ for all } x \geq 0.
\end{multline}

 \subsubsection{Bounding $|\overline{\Phi}(x a_n(x)) - \overline{\Phi}(x)|$ } \label{sec:anotheranothersubsection}
 First write  the inequality
\begin{equation*}
|x a_{n, x} - x| = 
\left|\frac{ (a_{n, x}^2 - 1)x}{a_{n,x } +1}\right| = 
\left| \left( \frac{1 - x^2}{n + x^2 -1} \right) \left(\frac{x}{ a_{n, x} + 1}   \right)\right| 
\leq  \frac{(1 + x^2)x}{(n -1) (a_{n, x} +1)};
\end{equation*}
the prior inequality in turns implies, via Taylor's theorem, 
\begin{multline} \label{prebridge1}
|\Phi(x a_{n, x}) - \Phi(x)| \leq \phi \Big(x (a_{n, x} \wedge 1)  \Big) \Bigl|xa_{n, x} - x \Bigr|\\
 \leq  \phi \Big(x (a_{n, x} \wedge 1) \Big)\frac{(1 + x^2)x}{(n -1) (a_{n, x} +1)} =\frac{(1 + x^2)x}{\sqrt{2\pi}(n -1) (a_{n, x} +1)} \exp\left( \frac{-x^2 (a_{n, x} \wedge 1)^2}{2}\right)  .
\end{multline}
by the mean-value theorem. 
At the same time, we also have 
\begin{equation} \label{prebridge2}
|\Phi(x a_{n, x} ) - \Phi(x)| \leq \barPhi(x a_{n, x} ) +\barPhi(x)
 \leq  \exp\left(\frac{- x^2  (a_{n, x} \wedge 1)^2}{2} \right)
\end{equation}
by the typical normal tail bound; see \citet[Eqn. (2.11)]{chen2011normal} for instance. 
Combining \eqref{prebridge1} and \eqref{prebridge2}, we have
\begin{equation} \label{triangle_bridge_2nd_part_bdd_prelim}
|\Phi(x a_{n, x}) - \Phi(x)| \leq \\
\min \left(  \frac{(1 + x^2)x}{\sqrt{2\pi}(n -1) (a_{n, x} +1)} ,  1  \right)\exp\left( \frac{-x^2 (a_{n, x} \wedge 1)^2}{2}\right).
\end{equation}
Now, for the range $0 \leq x \leq n^{1/6}$, from \eqref{triangle_bridge_2nd_part_bdd_prelim} and \eqref{lower_upper_bdd_anx} one get
\begin{equation}\label{bridge1}
|\Phi(x a_{n, x}) - \Phi(x)| 
\leq  \frac{(1 + n^{1/3})n^{1/6}}{\sqrt{\pi} (1 + \sqrt{2}) (n-1)} \exp\Big( \frac{ -x^2}{4}\Big)
 \leq \frac{C}{\sqrt{n}}  \exp\Big( \frac{ -x^2}{4}\Big) \text{ for }  0 \leq x \leq n^{1/6}.
\end{equation}
For the range $n^{1/6} < x \leq n^{1/2}$,  since $a_{n, n^{1/6}} \leq 1$,  from \eqref{triangle_bridge_2nd_part_bdd_prelim} one get
\begin{align}
|\Phi(x a_{n, x}) - \Phi(x)| &\leq \exp \Big( - \frac{ x^2 a_{n, x}^2}{4} \Big)  \exp \Big( - \frac{ x^2 a_{n, x}^2}{4} \Big)\notag\\
&= \underbrace{  \exp \Big( - \frac{ x^2 a_{n, x}^2}{4} + 3 \log ( x a_{n,x} )\Big) }_{\leq C}\frac{1}{x^3 a_{n, x}^3}  \exp \Big( - \frac{ x^2 a_{n, x}^2}{4} \Big)\notag\\
&\leq \frac{C}{n^{1/2} a_{n, x}^3}  \exp \Big( - \frac{ x^2 a_{n, x}^2}{4} \Big) \notag\\
&\leq \frac{C}{n^{1/2} a_{n, n^{1/2}}^3}  \exp \Big( - \frac{ x^2 a_{n, n^{1/2}}^2}{4} \Big)\notag\\
& \leq  \frac{ 2 ^{3/2}C}{\sqrt{n}} \exp \Big( - \frac{ x^2}{8} \Big) \text{ for } n^{1/6} < x \leq n^{1/2} \label{bridge2},
\end{align}
where the last inequality uses \eqref{lower_upper_bdd_anx}. For the range $x > n^{1/2}$, 
using that $x a_{n,x}$ is increasing in $x$ from \eqref{ascending_prop_sn} again, from \eqref{triangle_bridge_2nd_part_bdd_prelim} we get that
\begin{equation} \label{bridge3}
|\Phi(x a_{n, x}) - \Phi(x)| \leq \exp \Big(  - \frac{ x^2 a_{n, x}^2}{2} \Big) \leq \exp \Big(  - \frac{ n  a_{n, n^{1/2}}^2}{2} \Big) \leq \exp \big( - n/4\big) \text{ for } x > n^{1/2},
\end{equation}
where the last inequality again uses \eqref{lower_upper_bdd_anx}. Combining \eqref{bridge1}, \eqref{bridge2} and \eqref{bridge3}, we get
\begin{equation} \label{whole_bridge}
|\Phi(x a_{n, x}) - \Phi(x)| \leq  \exp\Big(-\frac{n}{4}\Big) +  \frac{ C}{\sqrt{n}} \exp \Big( - \frac{ x^2}{8} \Big)  \text{ for } x\geq 0 .
\end{equation}

Lastly, combining \eqref{tri_ineq_bridge}, \eqref{pre_bridge} and \eqref{whole_bridge},  we get
\[
|P(T_{student} > x) - \barPhi(x)| \leq C_1\exp \Bigg( \frac{-  c_1 n ( \bE[X_1^2])^3}{ (\bE[X_1^3])^2} \Bigg) + \frac{C_2}{e^{c_2x^2}} \frac{ \bE[|X_1| ^3]}{\sqrt{n} (\bE[X_1^2])^{3/2}} \text{ for all } x \geq 0
\]
because $\|X_1\|_3/\|X_1\|_2 \geq 1$, and \thmref{nonunif_BE_sn_sum}  for $T_{student}$ is proved.

\section*{Acknowledgments}
This research is partially supported by National Nature Science Foundation of China NSFC 12031005 and Shenzhen Outstanding Talents Training Fund, China. 


\section*{Appendices}
\appendix

These appendices are organized as follows: \appref{tech_lemmas} first list and prove some  supporting lemmas, with the remaining appendices covering the proofs for:

\begin{itemize}
\item  \appref{remnant_pf}: \lemsref{bdd_for_remnants}
\item  \appref{Stein_pf}: \lemref{intermediate} 
\item  \appref{otherproofs}:  \lemsref{chernoff_lower_tail_bdd_U_stat} and \lemssref{non_integral_moment_bdd_u_stat}
\end{itemize}

\section{Technical lemmas} \label{app:tech_lemmas}

This appendix lists out a few sets of useful results that, for the most part,  have already been established  in our related work \citet{leungshaounif}, except for \lemref{delta1_moment_bdd}.
For any $x \in \mathbb{R}$, recall that the Stein equation \citep{stein1972bound}
\begin{equation} \label{steineqt}
f'(w) - wf(w) = I(w \leq x) - \Phi(x),
\end{equation}
has the unique bounded solution $f(w) = f_x(w)$ of the form
\begin{equation} \label{fx}
f_x (w) \equiv 
  \begin{cases} 
  \sqrt{2 \pi} e^{w^2/2} \Phi(w)\bar{\Phi}(x) &  w \leq x ;\\
 \sqrt{2 \pi} e^{w^2/2} \Phi(x) \bar{\Phi}(w)        &  w > x;
     \end{cases}
\end{equation}
see \citet[p.14]{chen2011normal}. Since $f_x$ as in \eqref{fx} is not differentiable at $w = x$, we customarily define 
\begin{equation}\label{fx'atx}
f_x'(x) \equiv xf_x(x) +1- \Phi(x),
\end{equation}
so \eqref{steineqt} holds for all $w \in \mathbb{R}$.

\subsection{Properties of the solution to Stein's equation}

This section provides some useful bounds related to $f_x$ in \eqref{fx}. We will define 
\begin{equation} \label{gx_def}
g_x(w) \equiv (wf_x(w))'= f_x(w) + w f_x'(w),
\end{equation}
where  it is understood that $g_x(x) \equiv  f_x(x) + xf_x'(x)$ for $f_x'(x)$  defined in \eqref{fx'atx}. Precisely,
\begin{equation} \label{fx'}
f_x' (w) = 
  \begin{cases} 
\left(  \sqrt{2 \pi}w e^{w^2/2} \Phi(w) + 1\right)\bar{\Phi}(x)& \text{ for }\quad w \leq  x;\\
\left(  \sqrt{2 \pi}w e^{w^2/2} \bar{\Phi}(w) - 1\right) \Phi(x)        & \text{ for }\quad w > x;
     \end{cases}
\end{equation}
\begin{equation} \label{wfx'}
g_x(w) = 
  \begin{cases} 
 \sqrt{2 \pi}\bar{\Phi}(x) \left( (1 + w^2)  e^{w^2/2} \Phi(w) + \frac{w}{\sqrt{2\pi}}\right) & \text{ for }\quad w \leq x;\\
  \sqrt{2\pi}\Phi(x) \left(   (1 + w^2) e^{w^2/2}    \bar{\Phi}(w) - \frac{w}{\sqrt{2 \pi}} \right)  &\text{ for }\quad  w > x.
     \end{cases}
\end{equation}

\begin{lemma}[Uniform bounds] \label{lem:helping_unif}
For $f_x$ and $f_x'$, the following bounds are true:
\[
 |f_x'(w)| \leq 1, \quad  \quad  0 < f_x(w) \leq 0.63 \quad  \text{ and }\quad 0 \leq g_x(w) \quad \text{ for all } \quad w, x \in \bR.
\]
Moreover, for any $x \in [0, 1]$, $g_x(w) \leq 2.3$ for all $w \in \bR$.
\end{lemma}

\begin{lemma}[Nonuniform bounds when $x \geq 1$] \label{lem:helping}
For $x\geq 1$, 
the following are true: 

 \begin{equation} \label{fxbdd}
 f_x (w) \leq  
  \begin{cases} 
  1.7 e^{-x} &  \text{ for }\quad w \leq x - 1;\\
  1/x     &    \text{ for }\quad x  - 1 < w \leq x; \\
   1/w     &   \text{ for }\quad x  < w; 
     \end{cases}
     \end{equation}
     and
    \begin{equation} \label{fx'bdd}
      | f_x'(w)| \leq  
  \begin{cases} 
    e^{1/2 -x } &   \text{ for }\quad w \leq x - 1;\\
   1       &    \text{ for }\quad x  - 1 < w \leq x ;\\
      (1 + x^2)^{-1}    &  \text{ for }\quad w > x.  
     \end{cases}
\end{equation}
Moreover, $g_x(w) \geq 0$ for all $w \in \mathbb{R}$, 
\begin{equation}  \label{gbdd}
 g_x(w) \leq  
  \begin{cases} 
    1.6 \;\bar{\Phi}(x) &  \text{ for }\quad w \leq 0;\\
   1/w      &  \text{ for }\quad w > x,
     \end{cases}
\end{equation}
and $g_x(w)$ is increasing for $0 \leq w  \leq x$ with 
\[
g_x(x - 1) \leq x e^{1/2 -x} \;\ \text{ and } \;\ g_x(x)  \leq  x+ 2.
\]

\end{lemma}


\begin{lemma}[Bound on expectation of  $f_x(W_b)$ when $x \geq 1$]   \label{lem:expect_fxWB}
Let $\xi_1, \dots, \xi_n$ be independent random variables with $\bE[\xi_i] = 0$ for all $i = 1, \dots, n$ and $\sum_{i=1}^n \bE[\xi_i^2] \leq 1$, and define $\xi_{b, i} = \xi_i I(|\xi_i| \leq 1) + 1 I(\xi_i >1) - 1  I(\xi_i <-1)$  and $W_b =  \sum_{i=1}^n \xi_{b, i}$.
For $x \geq 1$, then there exists an absolute constant $C >0$ such that 
\[
|\bE[f_x (W_b)]| \leq C e^{-x}.
\]
\end{lemma}

\begin{proof}[Proof of \lemref{expect_fxWB}]
From \eqref{fxbdd}  in \lemref{helping} and $|f_x| \leq 0.63$ in \lemref{helping_unif}, 
\begin{align*}
|\bE[f_x (W_b)]| 
&\leq 1.7 e^{-x }  + |\bE[f_x(W_b)I(W_b  > x - 1)] | \leq  1.7 e^{-x } +  e^{1 - x} 0.63\bE[e^{W_b }],
\end{align*}
then apply the Bennett inequality in \lemref{bennett_censored} below.
\end{proof}

%
%
%

\subsection{Bounds for the  censored summands $\xi_{b,i}$'s and their sum $W_b$}
The following bounds for the censored summands $\xi_{b,i}$'s and their sum $W_b$ will be useful.

\begin{lemma}[Bound on expectation of $\xi_{b, i}$] \label{lem:exp_xi_bi_bdd} Let $\xi_{b, i} = \xi_i I(|\xi_i| \leq 1) + 1 I(\xi_i >1) - 1  I(\xi_i <-1)$ with $\bE[\xi_i] = 0$. Then
\[
  \left|\bE[ \xi_{b, i}]\right| \leq   \bE[ \xi_i^2I(|\xi_i| > 1) ]  \leq \bE[|\xi_i|^3] \wedge\bE[ \xi_i^2 ]
\]
\end{lemma}

\begin{lemma} [Bennett's inequality for a sum of censored random variables]  \label{lem:bennett_censored} 
Let $\xi_1, \dots, \xi_n$ be independent random variables with $\bE[\xi_i] = 0$ for all $i = 1, \dots, n$ and $\sum_{i=1}^n \bE[\xi_i^2] \leq 1$, and define $\xi_{b, i} = \xi_i I(|\xi_i| \leq 1) + 1 I(\xi_i >1) - 1  I(\xi_i <-1)$. 
For any $t >0$ and $W_b =  \sum_{i=1}^n \xi_{b, i}$, we have
\[
\bE[e^{t W_b}] \leq  \exp\left( e^{2t}/4 - 1/4 + t/2\right)
\]
\end{lemma}

%

\begin{lemma}[Exponential randomized concentration inequality  for a sum of  censored random variables] \label{lem:modified_RCI_bdd}
Let $\xi_1, \dots, \xi_n$ be independent random variables  with mean zero and finite second moments, and for each $i =1, \dots, n$, define 
\[
\xi_{b, i} = \xi_i I(|\xi_i| \leq 1) + 1 I(\xi_i >1) - 1  I(\xi_i <-1),
\]
an upper-and-lower censored version of $\xi_i$; moreover, let $W =\sum_{i=1}^n \xi_i $ and $W_b = \sum_{i=1}^n \xi_{b, i}$ be their corresponding sums, and $\Delta_1$ and $\Delta_2$ be two random variables on the same probability space. Assume there exists $c_1 > c_2 >0$ and  $\delta \in(0, 1/2)$ such that 
\[
\sum_{i=1}^n \bE[\xi_i^2] \leq c_1
\]
and 
\[
\sum_{i=1}^n \bE[|\xi_i| \min(\delta, |\xi_i|/2 )] \geq c_2.
\]
Then for any $\lambda \geq 0$, it is true that
\begin{align*}
&\bE[e^{\lambda W_b} I(\Delta_1 \leq W_b \leq \Delta_2)] \\
&\leq 
\left( \mathbb{E}\left[e^{2\lambda W_b}\right] \right)^{1/2}\exp\left( - \frac{c_2^2}{16 c_1 \delta^2}\right) \\
&+ \frac{2 e^{\lambda \delta}}{c_2} \Biggl\{ 2\sum_{i =1}^n \mathbb{E} [ |\xi_{b,  i}|e^{\lambda W_b^{(i)} } (|\Delta_1 - \Delta_1^{(i)}| + |\Delta_2 - \Delta_2^{(i)}|)] \\
& + 
\mathbb{E}[|W_b|e^{\lambda W_b }(|\Delta_2 - \Delta_1| + 2 \delta)]\\
& +   \sum_{i=1}^n  \Big|\bE[ \xi_{b, i}]\Big| \bE[e^{\lambda W_b^{(i)} }(|\Delta_2^{(i)} - \Delta_1^{(i)}| + 2 \delta)]\Biggr\} ,
\end{align*}
where $\Delta_1^{(i)}$ and $\Delta_2^{(i)}$ are any random variables on the same probability space such that $\xi_i$ and $(\Delta_1^{(i)}, \Delta_2^{(i)},  W^{(i)}, W_b^{(i)})$ are independent, where $W^{(i)} = W - \xi_i$ and $W_b^{(i)} = W_b - \xi_{b, i}$. In particular, if $\sum_{i=1}^n \bE[\xi_i^2] = 1$, one can take 
\[
\delta = \frac{\beta_2 + \beta_3}{4}, \quad \lambda = \frac{1}{2}, \quad c_1 = 1, \quad c_2 = \frac{1}{4},
\]
where $\beta_2 \equiv \sum_{i=1}^n \bE[\xi_i^2 I(|\xi_i|>1)]$ and $\beta_3 \equiv \sum_{i=1}^n \bE[\xi_i^3 I(|\xi_i|\leq1)]$.
 \end{lemma}
 
 \subsection{Bounds related to the  components of the censored denominator remainder in  \secref{mainpf}}
 This subsection supplements \secref{mainpf}, and in particular  \underline{\eqref{zero_mean_kernel_assumption} and \eqref{sigma_equal_1_assumption} are assumed to hold}. 
Given 
\[
\sum_{i=1}^n \bE[\xi_{b, i}^2] + \sum_{i=1}^n \bE[(\xi_i^2 - 1) I(|\xi_i |>1)] = \sum_{i=1}^n \bE[\xi_i^2]= 1
\]
and $d_n^2 = n/(n-1)$, we shall rewrite $ D_{2, V_b, \bardelta_1, \bardelta_{2,b}}  $ in \eqref{DR_censored_def} as 
 \begin{equation} \label{DRbar_rewrite}
D_{2, V_b, \bardelta_1, \bardelta_{2,b}}  =  d_n^2\Bigg(\delta_{0,b} + \bardelta_1 + \bardelta_{2, b} \Bigg) + \frac{\sum_{i=1}^n  \bE[\xi_{b, i}^2] }{n-1}-  \sum_{i=1}^n \bE\Big[(\xi_i^2 - 1) I(|\xi_i |>1)\Big],
 \end{equation}
 where
 \[
\delta_{0, b}  \equiv \sum_{i=1}^n \Big(\xi_{b, i}^2 - \bE[\xi_{b, i}^2]\Big).
\]
This section includes some useful properties related to the components $\delta_{0,b}$, $\delta_1$ and $\delta_{2, b}$ in \eqref{DRbar_rewrite}; recall $f_x$ is the solution to the Stein equation in \eqref{fx}.

  \begin{lemma}[Properties of $\delta_{0, b}$] \label{lem:delta_0_property}
 There exists positive absolute constants $C >0$ and $x \geq 1$, 
 \begin{equation}  \label{bdd_exp_delta0b_sq}
\bE[ \delta_{0, b}^2] \leq  \sum_{i=1}^n \bE[|\xi_{b, i}|^3]
 \end{equation}
 \begin{equation}\label{bdd_exp_ew_delta0b}
  \bE\Big[e^{W_b} \delta_{0, b}^2 \Big] \leq C  \sum_{i=1}^n\bE[|\xi_{b,i}|^3]
   \end{equation}
and
 \begin{equation} \label{bdd_fx_delta0b}
 \Big|\bE[\delta_{0, b} f_x(W_b)]\Big| \leq  C e^{-x} \sum_{i=1}^n \bE[|\xi_{b, i}|^3].
 \end{equation}
 \end{lemma}
 \begin{proof}[Proof of \lemref{delta_0_property}]
 The proofs of \eqref{bdd_exp_delta0b_sq} and \eqref{bdd_exp_ew_delta0b} can be found in \citet[Appendix D.1, p.30-31]{leungshaounif}; the proof of \eqref{bdd_fx_delta0b} can be  found in  \citet[Appendix D.2, p.33]{leungshaounif}. Note that $\delta_{0, b}$ is same as the quantity "$\Pi_1$"  in \citet{leungshaounif}.
 \end{proof}
 
  \begin{lemma}[Properties of $\delta_1$] \label{lem:delta1_moment_bdd}
  Assume $\bE[h(X_1, \dots, X_m)] = 0$ and $\sigma^2 =1$.
 For some positive absolute constants $C(m) >0$,
 
 \begin{equation} \label{delta1_1st_abs_moment_bdd}
  \bE\Big[|\delta_1| \Big] \leq    \frac{C m^2\bE[h^2]}{n}
   \end{equation}
 and
 \begin{equation}  \label{delta1_3over2_th_abs_moment_bdd}
  \bE\Big[|\delta_1|^{3/2} \Big] \leq C(m)  \frac{\bE[|h|^3]}{n^{3/2}}.
   \end{equation}
 \end{lemma}
 \begin{proof}[Proof of \lemref{delta1_moment_bdd}] 

 In \citet[Section 3, p.11]{leungshaounif}, we have established that
 \begin{equation*}
 \bE[|\delta_1^*|] \leq    2 \left[\frac{ m (m-1) (n -1)}{(n-m)^2} \right] + 
  \frac{4(n-1)^2(m-1)^2}{ (n-m)^2(n-m+1)m }
 \bE \left[ h^2\right].
 \end{equation*}
 Moreover, by \citet[Lemma 5.2.1.A$(i)$, p.183]{serfling1980},  $\bE[U_n^2] \leq  \frac{m \bE[h^2]}{n}$. Collecting these facts give  \eqref{delta1_1st_abs_moment_bdd} because $\bE[|h|^2] \geq 1$.

 For \eqref{delta1_3over2_th_abs_moment_bdd}, we need to first establish certain higher moment bounds for $W_n$ and $\Lambda_n$.
 For $W_n$, by  \citet[Theorem 3]{rosenthal1970subspaces}'s inequality, we have that 
 \begin{equation} \label{Wn_3rd_moment_bdd}
 \bE\Big[|W_n|^3\Big] \leq C \Bigg\{ \Big(\sum_{i=1}^n \bE[\xi_i^2]\Big)^{3/2}  + \sum_{i=1}^n \bE[|\xi_i|^3] \Bigg\} \leq C \Bigg(1 + \frac{\bE[|g|^3]}{\sqrt{n}} \Bigg).
 \end{equation}
 For $\Lambda_n$, upon rescaling $\Lambda_n^2$ with the factor $n{n \choose 2m-1}^{-1}$, we write:
  \begin{align}
&\tilde{\Lambda}_n^2 \equiv n {n \choose 2m-1}^{-1}\Lambda_n^2 \notag\\
&={n \choose 2m-1}^{-1} \sum_{i =1}^n \left(  \sum_{\substack{1 \leq i_1 < \cdots < i_{m-1}\leq n\\ i_l \neq i \text{ for } l = 1, \dots, m-1}} \bar{h}_m(X_i, X_{i_1}, \dots, X_{i_{m-1}})\right)^2 \notag\\
&= {n \choose 2m-1}^{-1}  \sum_{i = 1}^n 
\sum_{ \substack{1 \leq i_1 < \cdots < i_{m-1}\leq n\\ 1 \leq j_1 < \cdots < j_{m-1}\leq n \\ i_l, j_l\neq i \text{ for } l = 1, \dots, m-1}}  \bar{h}_m(X_i, X_{i_1}, \dots, X_{i_{m-1}}) \bar{h}_m(X_i, X_{j_1}, \dots, X_{j_{m-1}})   \notag \\
&= {n \choose 2m-1}^{-1} \left( \sum_{k = m}^{2m-1} \sum_{1 \leq i_1 < \dots < i_k\leq n}  \bar{\cH}_k ( X_{i_1}, \dots, X_{i_k})\right) \label{Lambda_n_reexpr},
\end{align}
where for each $k = m, \dots, 2m-1$, $\bar{\cH}_k :\mathbb{R}^k \longrightarrow \mathbb{R}$ is a symmetric kernel of degree $k$ defined as
\begin{equation} \label{bar_H_k}
\bar{\cH}_k (x_1, \dots, x_k) \equiv  (2m  - k) \times \sum_{\substack{\cS_1, \cS_2, \cS_3 \subset \{1, \dots, k\}: \\ |\cS_1| = 2m- k \\ |\cS_2| =  |\cS_3|=  k - m  \\  \cS_1, \cS_2, \cS_3 \text{disjoint}  }}     \bar{h}_m(x_{\cS_1}, x_{\cS_2}) \bar{h}_m(x_{\cS_1}, x_{\cS_3})
\end{equation}
induced by $h(\cdot)$. Hence, up to scaling factors, $\Lambda_n^2$  can be seen as a sum of $2m-1$ U-statistics with kernels of degree $k = m, \dots, 2m -1$. Moreover, for  $k = 2m -1$, 
$
\bar{\cH}_{2m-1} (x_1, \dots, x_m)
$
is seen to have 
the \emph{second-order} degeneracy property
\begin{equation} \label{2nd_order_degen_barHm}
\bE[\bar{\cH}_{2m-1} ( X_1, \dots, X_{2m-1}) |X_i, X_j] = 0 \text{ for } \{i, j\} \subset \{1, \dots, 2m-1\} \text{ and } i \neq j,
\end{equation}
which, in particular, implies 
\begin{equation}\label{mean_zero_barHm}
\bE[\bar{\cH}_{2m-1} ( X_1, \dots, X_{2m-1})] = 0;
\end{equation}
 we will prove \eqref{2nd_order_degen_barHm} at the end.

Now, by taking the $(3/2)$-th absolute central moment of $\tilde{\Lambda}_n^2 $, from \eqref{Lambda_n_reexpr} one get
\begin{align}
& \bE\Bigg[\Bigg|\tilde{\Lambda}_n^2 - \bE[\tilde{\Lambda}_n^2]\Bigg|^{3/2}\Bigg] \notag\\
&\leq C(m)\sum_{k=m}^{2m-1} \frac{1}{n^{3(2m-1-k)/2}} \bE\Bigg[ \Bigg|\frac{ \sum_{1 \leq i_1 < \dots < i_k\leq n}  \bar{\cH}_k ( X_{i_1}, \dots, X_{i_k})}{{n \choose k}} - \bE[\bar{\cH}_k ]\Bigg|^{3/2}\Bigg]\notag\\
&\leq C(m)  \Bigg\{n^{-3/2} \bE\Big[  |\bar{\cH}_{2m-1} (X_1, \dots, X_k)|^{3/2}\Big] +  \sum_{k=m}^{2m-2} \underbrace{ \frac{1}{n^{3m-1 -3k/2}} }_{\leq n^{-2}} \bE \Big[  |\bar{\cH}_k (X_1, \dots, X_k)|^{3/2}\Big] \Bigg\}\notag\\
&\leq C(m) \frac{\bE[|h|^3]}{n^{3/2}} \label{tilde_Lambda_sq_central_mom_bdd}
\end{align}
where the second last inequality uses the moment bound \eqref{simple_u_stat_moment_bdd} for centered U-statistics 
in \lemref{non_integral_moment_bdd_u_stat} and the  degeneracy of $\bar{\cH}_{2m-1}$ in \eqref{2nd_order_degen_barHm}; the last inequality in \eqref{tilde_Lambda_sq_central_mom_bdd} uses the definition in \eqref{bar_H_k}, the Cauchy inequality
\[
\bE[|\bar{h}_m(X_{\cS_1}, X_{\cS_2}) \bar{h}_m(X_{\cS_1}, X_{\cS_3})|^{3/2}] \leq \bE[|\bar{h}_m(X_1, \dots, X_m)|^3]
\]
and that
$
 \bE[|\bar{h}_m|^3] \leq C(m)  \bE[|h_m|^3]$ from \eqref{Jensen_consequence}. On the other hand,  from  the definition in \eqref{bar_H_k} and that $\bE[\bar{h}_m^2] \leq C(m)\bE[h_m^2]$ from \eqref{Jensen_consequence}  \footnote{Actually it can also be shown that $\bE[\bar{h}_m^2] \leq \bE[h_m^2]$; see $(10.76)$ in \citet[p.284]{chen2011normal}.}, it follows that
\[
\Big|\bE[\bar{\cH}_k (X_1, \dots, X_k)]\Big| \leq C(m) \bE[h^2(X_1, \dots, X_m)] \text{ for } k = 1, \dots, 2m-2,
\]
which, together with \eqref{mean_zero_barHm}, implies that
\begin{equation} \label{tilde_Lambda_sq_mean_bdd}
\bE[\tilde{\Lambda}_n^2 ] \leq C(m) \frac{ \bE[h^2]}{n}
\end{equation}
Combining \eqref{tilde_Lambda_sq_central_mom_bdd} and \eqref{tilde_Lambda_sq_mean_bdd}, by $\|h(X_1, \dots, X_m)\|_2 \leq \|h(X_1, \dots, X_m)\|_3$, we have
\begin{equation} \label{lucky_bdd_minus_order_3_voer_2}
\bE\Big[|\tilde{\Lambda}_n|^{3/2} \Big]  = \bE\Bigg[\Bigg(n {n \choose 2m-1}^{-1}\Bigg)^{3/2}|\Lambda_n|^3\Bigg] \leq \frac{C(m)\bE[|h|^3]}{n^{3/2}}
\end{equation}

Now we can finish proving the lemma, by using the basic property that
\begin{multline}\label{WPsi_bdd}
 \frac{2(n-1)(m-1) }{(n-m)^2 {n-1 \choose m-1}} \Bigg|\sum_{i = 1}^n  W_n \Psi_{n, i}\Bigg|
\leq \frac{n (m-1)^2}{(n-m)^2} W_n^2 + \frac{(n-1)^2}{{n-1 \choose m-1}^2 (n-m)^2 } \Lambda_n^2,
\end{multline}
a consequence of the Cauchy's inequality $2 |W_n   \sum_{i=1}^n \Psi_{n, i}|  \leq 2 \sqrt{n} |W_n| \Lambda_n$.
 Recalling the definition of $\delta_1^*$ in \eqref{delta_1_star}, we can apply \eqref{Wn_3rd_moment_bdd} and \eqref{lucky_bdd_minus_order_3_voer_2} to get
 \begin{align*}
 \bE[|\delta_1^*|^{3/2}] &\leq C\Bigg\{  \left[ \frac{ n(m-1)^2}{(n-m)^2} + \frac{2(m-1)}{(n-m)}\right]^{3/2} \bE[ |W_n|^3] +  \Bigg[ \frac{(n-1)^2}{ {n-1 \choose m-1}^2 (n-m)^2} \Bigg]^{3/2}\bE[|\Lambda_n|^3] \Bigg\}\\
 &\leq C(m) \Bigg\{n^{-3/2} \Bigg(1 + \frac{\bE[|g|^3]}{\sqrt{n}} \Bigg) +  \Bigg[ \frac{(n-1)^2 {n \choose 2m-1}}{ n{n-1 \choose m-1}^2 (n-m)^2} \Bigg]^{3/2}\frac{\bE[|h|^3]}{n^{3/2}} \Bigg\}\\
 &\leq C(m)  \frac{\bE[|h|^3]}{n^{3/2}}.
 \end{align*}
Since $\bE[|U_n|^{3}] \leq C(m) n^{-3/2}\bE[|h|^3]$ by the  bound \eqref{simple_u_stat_moment_bdd} in \lemref{non_integral_moment_bdd_u_stat},  we get from \eqref{delta_1} that
 \[
 \bE[|\delta_1|^{3/2}] \leq C(m) \left(\bE[|\delta_1^*|^{3/2}] +  \bE[|U_n|^{3}]\right) \leq C(m)  \frac{\bE[|h|^3]}{n^{3/2}};
 \]
 \eqref{delta1_3over2_th_abs_moment_bdd} is proved.

 It remains to show  \eqref{2nd_order_degen_barHm}, for which we will leverage  the degenerate property of $\bar{h}_m$ in  \eqref{degen_prop_bar_h_m}.
 From the definition in \eqref{bar_H_k}, it suffices to show that each summand of $\bar{\cH}_{2m-1} (X_1, \dots, X_{2m-1})$ has the same property, i.e. 
 \[
 \bE[   \bar{h}_m(X_{\cS_1}, X_{\cS_2}) \bar{h}_m(X_{\cS_1}, X_{\cS_3})| X_i, X_j] = 0
 \]
 for any disjoint $\cS_1, \cS_2, \cS_3 \subset \{1, ,\dots, 2m-1\}$ such that $|\cS_1| = 1$. We consider three cases \footnote{It's possible that $i \in \cS_3$, $j \in \cS_2$ for case $(ii)$ and $i, j \in \cS_3$ for case $(iii)$, with the same proof.}: 
 \begin{enumerate}
 \item 
 If $\{i, j\}\cap \cS_1 \neq \emptyset$, without loss of generality, we can assume that  $i \in \cS_1$ and $j \in \cS_2$. Then, by \eqref{degen_prop_bar_h_m},
\begin{align*}
&\bE[\bar{h}_m (X_{\cS_1}, X_{\cS_2}) \bar{h}_m (X_{\cS_1}, X_{\cS_3} )| X_i , X_j] \\
&= \bE[ \bar{h}_m (X_{\cS_1}, X_{\cS_2})| X_i, X_j] \bE[ \bar{h}_m (X_{\cS_1}, X_{\cS_3})| X_{i}] \\
&= \bE[ \bar{h}_m (X_{\cS_1}, X_{\cS_2})| X_i, X_j] \cdot 0 = 0.
\end{align*}

\item  If $\{i, j\}\cap \cS_1 = \emptyset$ and $i \in \cS_2$ and $j \in \cS_3$,  then by \eqref{degen_prop_bar_h_m},
\begin{align*}
&\bE[\bar{h}_m (X_{\cS_1}, X_{\cS_2}) \bar{h}_m (X_{\cS_1}, X_{\cS_3} )| X_i , X_j] \\
&= \bE[ \bar{h}_m (X_{\cS_1}, X_{\cS_2})| X_i] \bE[ \bar{h}_m (X_{\cS_1}, X_{\cS_3})| X_j] \\
&= 0 \cdot 0 = 0.
\end{align*}

\item  If $\{i, j\}\cap \cS_1 = \emptyset$ and $i, j \in \cS_2$,  then
\begin{align*}
&\bE[\bar{h}_m (X_{\cS_1}, X_{\cS_2}) \bar{h}_m (X_{\cS_1}, X_{\cS_3} )| X_i , X_j] \\
&= \bE[ \bar{h}_m (X_{\cS_1}, X_{\cS_2})| X_i, X_j] \bE[ \bar{h}_m (X_{\cS_1}, X_{\cS_3})] \\
&= \bE[ \bar{h}_m (X_{\cS_1}, X_{\cS_2})| X_i, X_j] \cdot 0 = 0.
\end{align*}

\end{enumerate}
 \end{proof}
 
 \begin{lemma}[Properties of $\delta_{2, b}$]  \label{lem:delta2b_property}
  Assume $\bE[h(X_1, \dots, X_m)] = 0$ and $\sigma^2 =1$.  There exists a positive absolute constant $C(m) >0$ depending only on $m$ such that
\begin{equation}\label{delta2b_l2_bdd}
\|\delta_{2, b}\|_2\leq C(m) \Bigg\{  \frac{\|g\|_3 \|h\|_3}{\sqrt{n}}\Bigg\}
\end{equation}
Moreover, for an absolute constant $C >0$,
\begin{equation} \label{bdd_fx_delta2b}
|\bE[\bardelta_{2, b} f_x(W_b)]| \leq  C e^{-x} \|\delta_{2, b}\|_2 \text{ for } x\geq 1.
\end{equation}

\end{lemma}

\begin{proof}[Proof of \lemref{delta2b_property}]
$\delta_{2, b}$ is precisely the quantity "$\Pi_{22}$" in \citet[Appendix E.1]{leungshaounif}, and  the bound  \eqref{delta2b_l2_bdd}  is shown as  equation $(E.3)$ in \citet{leungshaounif} which states
\[
\|\delta_{2, b}\|_2^2 \leq C(m) \left(\frac{ \bE[         
    h^2 
   ]}{n} + \frac{ \|g\|_3^2 \|h\|_3^2}{n} \wedge \frac{ \|h \|_3^2}{n^{2/3}}\right),
\]
and can be further simplified as  \eqref{delta2b_l2_bdd} because $\|h\|_2 \leq \|h\|_3$ and $1=  \|g\|_2 \leq \|g\|_3$. 

 \eqref{bdd_fx_delta2b} can be easily proved using a technique from \citet[Appendix D.2]{leungshaounif} as follows: By \eqref{fxbdd} in \lemref{helping} and  $ 0 < f_x \leq 0.63$ in \lemref{helping_unif},
  \begin{align*}
\Big| \bE\Big[\bardelta_{2,b} f_x(W_b)\Big]\Big| &= \Big|  \bE\Big[\bardelta_{2,b}  f_x(W_b)I(W_b \leq x -1) \Big]  + \bE\Big[\bardelta_{2,b}  f_x(W_b) I(W_b > x -1) \Big] \Big|\\
&\leq  1.7 e^{-x} \bE\Big[|\delta_{2,b} |  \Big] +  0.63\bE\Big[ |\delta_{2,b} | \frac{e^{W_b}}{e^{x-1}} \Big]\leq C e^{-x} \|\delta_{2,b} \|_2,
\end{align*}
where the last inequality uses Bennett's inequality (\lemref{bennett_censored}).
\end{proof}

\section{Proof of \lemref{bdd_for_remnants}} \label{app:remnant_pf}

\begin{proof}[Proof of \lemref{bdd_for_remnants}$(i)$]
Rewrite $D_1$ in \eqref{D1} as 
 \begin{equation*} 
D_1 = \frac{\sqrt{n}}{m} \times {n \choose m }^{-1} \sum_{1 \leq i_1 < \dots < i_m\leq n} \bar{h}_m(X_{i_1}, \dots, X_{i_m}).
\end{equation*}
In light of \eqref{degen_prop_bar_h_m}, we recognize that $ {n \choose m }^{-1} \sum_{1 \leq i_1 < \dots < i_m\leq n} \bar{h}(X_{i_1}, \dots, X_{i_m})$ is a mean-0 degenerate U-statistic of rank $2$. By the bound for  the  central absolute moment  of U-statistics in \lemref{non_integral_moment_bdd_u_stat} and  that $\bE[|\bar{h}_m|^3] \leq C(m)\bE[|h_m|^3]$ from \eqref{Jensen_consequence}, we have
\[
\bE[|D_1|^3] \leq  \frac{n^{3/2}}{m^3} \frac{C(m)}{n^3} \bE[|h|^3] = \frac{C(m)}{n^{3/2}} \bE[|h|^3].
\]
By Markov's inequality, we hence get
\[
P\Big(|D_1|>\frac{\frakcm x}{4} \Big) \leq \frac{64\bE[|D_1|^3] }{ \frakcm^3 x^3} \leq \frac{C(m)}{\frakcm^3 n^{3/2}(1 +x^3)} \bE[|h|^3].
\]

\end{proof}

\begin{proof}[Proof of \lemref{bdd_for_remnants}$(ii)$]
Let $W^{(i)} = W - \xi_i$, which satisfies 
\[
P\Big(W^{(i)} \geq \frac{\frakcm x}{4} \Big)  \leq P\left(\max_{1 \leq i \leq n} |\xi_i| > \frac{\frakcm x}{6}\right) + e^{3/2} \Big(1 + \frac{(\frakcm x)^2}{24}\Big)^{-3/2}
\]
by \citet[Lemma 8.2]{chen2011normal}(taking $p = 3/2$ in that lemma). This implies 
\begin{align*}
&P\Big(W \geq \frac{\frakcm x}{2}, \max_{1 \leq i \leq n} |\xi_i| >1 \Big)\\
&\leq \sum_{i=1}^n P\Big(W \geq \frac{\frakcm x}{2},  |\xi_i| >1\Big)\\
&\leq \sum_{i=1}^n P\Big(|\xi_i | > \frac{\frakcm x}{4} \Big) + \sum_{i=1}^n P\Big(W^{(i)} \geq \frac{\frakcm x}{4} \Big) P(|\xi_i| > 1)\\
&\leq \sum_{i=1}^n  P\Big(|\xi_i| > \frac{\frakcm x}{4} \Big) +  \left\{  P\left(\max_{1 \leq i \leq n} |\xi_i| > \frac{\frakcm x}{6}\right) + e^{3/2} \Big(1 + \frac{(\frakcm x)^2}{24}\Big)^{-3/2}\right\} \sum_{i=1}^n P(|\xi_i|>1)\\
&\leq \sum_{i=1}^n  P\Big(|\xi_i| > \frac{\frakcm x}{4} \Big) +  \left\{  P\left(\max_{1 \leq i \leq n} |\xi_i| > \frac{\frakcm x}{6}\right) + e^{3/2} \Big(1 + \frac{(\frakcm x)^2}{24}\Big)^{-3/2}\right\}  \sum_{i=1}^n  \bE[|\xi_i|^2 I(|\xi_i|>1)]\\
&\leq  2\sum_{i=1}^n P\Big(|\xi_i| > \frac{\frakcm x}{6} \Big) + e^{3/2} \Big(1 + \frac{(\frakcm x)^2}{24}\Big)^{-3/2}  \sum_{i=1}^n  \bE[|\xi_i|^2 I(|\xi_i|>1)] \text{ given  \eqref{sigma_equal_1_assumption}} \\
&\leq 2\sum_{i=1}^n P\Big(|\xi_i| > \frac{\frakcm x}{6} \Big) + \frac{(24 \cdot e)^{3/2}}{ \frakcm^3 (24 + x^2)^{3/2}} \sum_{i=1}^n  \bE[|\xi_i|^2 I(|\xi_i|>1)]  \text{ given $0 < \frakcm < 1$}  \\
&\leq \frac{C}{\frakcm^3(1 + x^3)} \sum_{i=1}^n \bE[|\xi_i|^3] =  \frac{C\bE[|g|^3]}{\frakcm^3(1 + x^3)\sqrt{n}}.
\end{align*}
\end{proof}

\begin{proof}[Proof of \lemref{bdd_for_remnants}$(iii)$]
 \begin{align*} \label{iii_bd_prelim}
&P\left( W_b  \geq  \frac{ \frakcm x}{2}, |\delta_1| > n^{-1/2}\right) \\
&\leq  e^{-\frakcm x/2} \bE[e^{W_b} I(|\delta_1| >n^{-1/2})]   \\
&\leq e^{-\frakcm x/2} \|e^{W_b}\|_3 \|I(|\delta_1| >n^{-1/2})\|_{3/2} \text{ by H\"older's inequality}\\
&\leq C e^{-\frakcm x/2} \|n^{1/2} |\delta_1|\|_{3/2} \text{ by Bennett's inequality (\lemref{bennett_censored}) and $I(|\delta_1| >n^{-1/2}) \leq n^{1/2}|\delta_1|$} \\
&\leq  C(m) e^{-\frakcm x/2} \frac{\|h(X_1, \dots, X_m)\|_3^2}{\sqrt{n}} \text{  by \eqref{delta1_3over2_th_abs_moment_bdd} in \lemref{delta1_moment_bdd} }.
 \end{align*}
\end{proof}
\begin{proof} [Proof of \lemref{bdd_for_remnants}$(iv)$]
\begin{align*}
&P\left( W_b  \geq  \frac{ \frakcm x}{2}, |\delta_{2,b} |> 1\right)  \\
&\leq  \bE[e^{W_b - \frac{\frakcm x}{2}} I(|\delta_{2, b}| > 1)] \\
&\leq e^{-\frakcm x/2}  \bE[e^{W_b}|\delta_{2,b}|] \\
&\leq C e^{-\frakcm x/2} \|\delta_{2,b}\|_2 \text{ by Bennett's inequality (\lemref{bennett_censored}) }\\
&\leq  C(m) e^{-\frakcm x/2} \frac{\|g\|_3 \|h\|_3}{\sqrt{n}}  \text{ by  \lemref{delta2b_property}}.
\end{align*}
\end{proof}

\begin{proof}[Proof of \lemref{bdd_for_remnants}$(v)$]

 From the alternative form of $D_{2, V_b, \bardelta_1, \bardelta_{2,b}} $ in \eqref{DRbar_rewrite}, 
\begin{align*}
&  P\left( W_b  \geq  \frac{\frakcm x}{2}, | D_{2, V_b, \bardelta_1, \bardelta_{2,b}}  | > 1 \right)  \\
&\leq e^{-  \frakcm x/2} \bE[   e^{W_b}   D_{2, V_b, \bardelta_1, \bardelta_{2,b}}  ^2]\\
&\leq C e^{-  \frakcm x/2} \Bigg\{\bE[e^{W_b} \delta_{0, b}^2] + \bE[e^{W_b} \bardelta_1^2 ]+ \bE[e^{W_b} \bardelta_{2, b}^2] \\
&\hspace{1cm}+ \bE[e^{W_b}] \Bigg(\frac{\sum_{i=1}^n  \bE[\xi_{b, i}^2] }{n-1}+  \sum_{i=1}^n \bE\Big[(\xi_i^2 - 1) I(|\xi_i| >1)\Big] \Bigg)\Bigg\} ,
\end{align*}
where we have also used that $d_n \leq 2$ and both
\[
\frac{\sum_{i=1}^n  \bE[\xi_{b, i}^2] }{n-1} \text{ and } \sum_{i=1}^n \bE\Big[(\xi_i^2 - 1) I(|\xi_i |>1)\Big]
\]
are less than 1. Continuing,  we get
\begin{align*}
&  P\left( W_b  \geq  \frac{\frakcm x}{2}, | D_{2, V_b, \bardelta_1, \bardelta_{2,b}}  | > 1 \right)  \\
&\leq   Ce^{-\frakcm x/2}  \{\bE[e^{W_b} \delta_{0, b}^2] + \|e^{W_b}\|_2  \|\bardelta_1^2 \|_2+ \|e^{W_b} \|_2 \|\bardelta_{2, b}^2\|_2 + n^{-1} + \sum_{i=1}^n \bE[\xi_i^2 I(|\xi_i| >1)]\}\\
&\hspace{10cm} \text{ (by \lemref{bennett_censored})}\\
&\leq   Ce^{-\frakcm x/2}  \Big\{\bE[e^{W_b} \delta_{0, b}^2] + \|e^{W_b}\|_2  \sqrt{\bE[|\bardelta_1|]}+ \|e^{W_b} \|_2 \|\bardelta_{2, b}\|_2 +  n^{-1}+\sum_{i=1}^n \bE[|\xi_i|^3]\Big\}\\
&\hspace{10cm} \text{ (by $|\bardelta_1| \vee |\bardelta_{2,b}| \leq 1$)}.
\end{align*}
To wrap up the proof, apply Bennett's inequality (\lemref{bennett_censored}),  \lemsref{delta_0_property}-  \lemssref{delta2b_property}, $n^{-1} \leq n^{-1/2} \bE[|g|^3]$ and that  $\|h\|_2 \leq \|g\|_3 \|h\|_3$ to the last line and get
\[
P\left( W_b  \geq  \frac{\frakcm x}{2}, | D_{2, V_b, \bardelta_1, \bardelta_{2,b}}  | > 1 \right) \leq  Ce^{-\frakcm x/2} \Bigg( \frac{\bE[|g|^3] +   \|g\|_3 \|h\|_3}{\sqrt{n}} \Bigg).
\]
\end{proof}

\begin{proof} [Proof of \eqref{Rx_final_bdd}]
 Given \eqref{Rx_bdd}, simply putting $(i)-(v)$ together and use the fact $\|h\|_3 \|g\|_3 \leq \|h\|_3^2$.
\end{proof}

\section{Proof of \lemref{intermediate}} \label{app:Stein_pf}

Before proving \lemref{intermediate}, we shall review a useful property of variable censoring discussed in \citet[Section 2.0.3]{leungshaounif}:
\begin{property} \label{property:censoring_property}
 Suppose $Y$ and $Z$ are any two real-value variables censored in the \emph{same manner}, i.e. for some  $a, b \in \bR \cup \{-\infty, \infty\}$ with $a \leq b$, we define their censored versions 
\[
\bar{Y} \equiv a I(Y < a) + Y I(a \leq Y \leq b)+bI(Y>b)
\] and 
\[
\bar{Z}  \equiv a I(Z < a) + Z I(a \leq Z \leq b)+bI(Z>b).
\]
Then it must be that 
$
|\bar{Y} - \bar{Z} | \leq |Y- Z|.
$
\end{property}

Now we begin the proof. We first let 
\[
D_{2, b, n^{-1/2}, \bardelta_{2, b}} =  d_n^2( V_{n, b}^2 + n^{-1/2} + \bar{\delta}_{2, b}) - 1
\text{ and }
D_{2, b, -n^{-1/2}, \bardelta_{2, b}} =  d_n^2( V_{n, b}^2 -n^{-1/2} + \bar{\delta}_{2, b}) - 1,
\]
where we respectively replaced $\bardelta_1$ with its lower and upper bounds $-n^{-1/2}$ and $n^{-1/2}$  from the definition of $D_{2, V_b, \bardelta_1, \bardelta_{2,b}}$ in \eqref{DR_censored_def}. Analogously to \eqref{barDR_censored_def}, we also let
\begin{multline*}
\barD_{2, V_b, n^{-1/2}, \bardelta_{2,b}}\equiv D_{2, V_b, n^{-1/2}, \bardelta_{2,b}}  I\Bigg(\frac{9\frakc_m^2}{16}- 1 \leq D_{2, V_b, n^{-1/2}, \bardelta_{2,b}}   \leq 1 \Bigg)\\
 +I\Bigg(D_{2, V_b, n^{-1/2}, \bardelta_{2,b}} >1\Bigg) + \Big(\frac{9\frakc_m^2}{16}- 1 \Big)I\Bigg(D_{2, V_b, n^{-1/2}, \bardelta_{2,b}}   < \frac{9\frakcm^2}{16}- 1\Bigg).
\end{multline*}
and
\begin{multline*}
\barD_{2, V_b, - n^{-1/2}, \bardelta_{2,b}}\equiv D_{2, V_b, - n^{-1/2}, \bardelta_{2,b}}  I\Bigg(\frac{9\frakc_m^2}{16}- 1 \leq D_{2, V_b, - n^{-1/2}, \bardelta_{2,b}}   \leq 1 \Bigg)\\
 +I\Bigg(D_{2, V_b, - n^{-1/2}, \bardelta_{2,b}} >1\Bigg) + \Big(\frac{9\frakc_m^2}{16}- 1 \Big)I\Bigg(D_{2, V_b, - n^{-1/2}, \bardelta_{2,b}}   < \frac{9\frakcm^2}{16}- 1\Bigg).
\end{multline*}
With respect to these, we define the  "placeholder" denominator remainder 
\begin{multline} \label{placeholder_D2_def}
\DRplaceholder = \DRplaceholder(X_1, \dots, X_n) \equiv  d_n^2( V_{n, b}^2 + (-n^{-1/2}|n^{-1/2}) + \bar{\delta}_{2, b}) - 1 \\
= d_n^2\Bigg(\delta_{0,b} + (-n^{-\frac{1}{2}}|n^{-\frac{1}{2}})+ \bardelta_{2, b} \Bigg) + \frac{\sum_{i=1}^n  \bE[\xi_{b, i}^2] }{n-1}-  \sum_{i=1}^n \bE\Big[(\xi_i^2 - 1) I(|\xi_i |>1)\Big]
\end{multline}
 (where the second line comes from \eqref{DRbar_rewrite}) and 
its censored version
\begin{multline} \label{placeholder_D2_def_censored}
\barDRplaceholder \equiv
 \DRplaceholder I\Bigg(\frac{9\frakc_m^2}{16}- 1 \leq \DRplaceholder   \leq 1 \Bigg)
 +I\Bigg(\DRplaceholder >1\Bigg) + \Big(\frac{9\frakc_m^2}{16}- 1 \Big)I\Bigg(\DRplaceholder   < \frac{9\frakcm^2}{16}- 1\Bigg),
\end{multline}
where for any $a, b \in \bR$, $(a|b)$ represents either $a$ or $b$, which means that
\[
\DRplaceholder \text{ represents either } D_{2, b, n^{-1/2}, \bardelta_{2, b}} \text{ or } D_{2, b, - n^{-1/2}, \bardelta_{2, b}}.
\]
 Since $- n^{-1/2} \leq \bardelta_1 \leq n^{-1/2}$ by its definition \eqref{bardelta1_def}, it is easy to see that
\begin{multline} \label{sandwich}
P(W_b + \barD_{1, x} > x (1 +  \bar{D}_{2, V_b, n^{-1/2}, \bardelta_{2, b}} )^{1/2}) \\
\leq P(W_b + \barD_{1, x} > x (1 +  \bar{D}_{2, V_b, \bardelta_1, \bardelta_{2, b}} )^{1/2})\\
 \leq P(W_b + \barD_{1, x} > x (1 +   \bar{D}_{2, V_b, -n^{-1/2}, \bardelta_{2, b}} )^{1/2}).
\end{multline}
Therefore, to show  \lemref{intermediate}, it suffices to show the same bound \eqref{intermediate_nonunif_bdd} with $ \bar{D}_{2, V_b, n^{-1/2}, \bardelta_{2, b}} $ replaced by $\bar{D}_{2, V_b, n^{-1/2}, \bardelta_{2, b}}$ or $\bar{D}_{2, V_b, -n^{-1/2}, \bardelta_{2, b}}$, i.e., 
\begin{equation} \label{intermediate_nonunif_bdd_2}
\Big|P\Big(W_b + \barD_{1, x} > x\Big(1 +  \barDRplaceholder \Big)^{1/2}\Big) - \barPhi(x)\Big| \leq \frac{C(m)}{ e^{c(m)x}} \Bigg( \frac{\bE[|g|^3]}{\sqrt{n}} +   \frac{\|g\|_3^2 \|h\|_3}{\sqrt{n}}\Bigg).
\end{equation}
As will be seen later, transforming the problem into one that handles $\DRplaceholder$ instead of ${D}_{2, V_b, \bardelta_1, \bardelta_{2, b}}$ has the advantage of obviating the need to deal with the variability of $\bardelta_1$;  a similar strategy has also been employed in our related work \citet{leungshaounif} for proving uniform B-E bounds.

  Since $\barDRplaceholder  > -1$ almost surely, by applying the elementary inequality that
\[
 (1 + s)^{1/2} \leq 1 + s/2 \text{ for all } s \geq -1,
\]
one get the two event  inclusions
\begin{multline*} \label{event1}
\Bigg\{W_b + \barD_{1, x} > x\Big(1 +  \barDRplaceholder \Big)^{1/2}\Bigg\} \\
\subset \left\{W_b  + \barD_{1,x} - \frac{x}{2}\barDRplaceholder  > x\right\} \bigcup \left\{x (1 + \barDRplaceholder)^{1/2} < W_b + \barNR \leq x ( 1 +  \barDRplaceholder/2)\right\} 
\end{multline*}
and 
\begin{equation*} \label{event2}
\Bigg\{W_b + \barD_{1, x} > x\Big(1 +  \barDRplaceholder \Big)^{1/2}\Bigg\} \supset \left\{W_b  + \barD_{1,x} - \frac{x}{2} \barDRplaceholder  > x\right\}.
\end{equation*}
These imply
\begin{multline} \label{two_quantities_to_bdd}
\Big|P\Big(W_b + \barD_{1, x} > x\Big(1 +  \barDRplaceholder \Big)^{1/2}\Big) - \barPhi(x)\Big| \leq \\ 
P\left( x (1 + \barDRplaceholder)^{1/2} < W_b + \barNR \leq x ( 1 +  \barDRplaceholder/2) \right)  + 
 \left|P\left(W_b  + \barD_{1,x} - \frac{x}{2} \barDRplaceholder > x\right) - \bar{\Phi}(x)\right|.
\end{multline}
Hence, proving \eqref{intermediate_nonunif_bdd_2} boils down to bounding the two terms on the right of \eqref{two_quantities_to_bdd}

To do this, we shall first define the "leave-one-out" variants for some of the variables involved. Let $i \in \{1,\dots, n\}$ be any sample point. For the numerator remainder, we define the variant of $D_1$ with all terms involving $X_i$ omitted, i.e. 
\begin{equation} \label{D1i_def}
D_1^{(i)} = D_1^{(i)}(X_1, \dots, X_{i-1}, X_i \dots, X_n)
\equiv {n -1 \choose m -1}^{-1} \sum_{\substack{1 \leq i_1 < \dots < i_m \leq n\\ i_l \neq i \text{ for } l = 1, \dots, m}} \frac{\bar{h}_m(X_{i_1}, X_{i_2}, \dots, X_{i_m}) }{\sqrt{n}},
\end{equation}
and its corresponding censored version
\[
\barD_{1, x}^{(i)} \equiv D_1^{(i)} I\left(|D_1^{(i)}| \leq \frac{\frakcm x}{4} \right) +\frac{\frakcm x}{4} I\left(D_1^{(i)} >\frac{\frakcm x}{4}\right) - \frac{\frakcm x}{4} I\left(D_1^{(i)} < - \frac{ \frakcm  x}{4}\right).
\]
For the denominator remainder,  we first define
\[
\delta_{2, b}^{(i)} \equiv \frac{2 (n-1) }{\sqrt{n}(n-m)} {n-1 \choose m-1}^{-1} \sum_{\substack{j =1\\ j \neq i}}^n \xi_{b, j} \sum_{\substack{1 \leq i_1 < \dots < i_{m-1} \leq n \\
i_l \neq j, i \text{ for } l  = 1, \dots, m-1}} \bar{h}_m (X_j, X_{i_1}, \dots, X_{i_{m-1}}).
\]
and  its censored variant
\[
\bardelta_{2, b}^{(i)}  = \delta_{2, b}^{(i)}I(| \delta_{2, b}^{(i)}| \leq 1) + I(\delta_{2, b}^{(i)}> 1) -  I( \delta_{2, b}^{(i)} < -1).
\]
Base on them, we can define the "leave-one-out" denominator remainder
\begin{equation}
\DRplaceholder^{(i)} =  \DRplaceholder^{(i)} (X_1, \dots, X_{i-1}, X_i, \dots, X_n) \equiv 
 d_n^2 \Big(\sum_{\substack{j =1 \\j \neq i}}^n \xi_{b, j}^2 + (-n^{-1/2}|n^{-1/2}) + \bardelta_{2, b}^{(i)} \Big) - 1 
\end{equation}
that omits all terms involving $X_i$ or $\xi_i$, 
and its censored version
\begin{equation*}
\barDRplaceholder^{(i)} \equiv \DRplaceholder^{(i)}  I\Bigg(\frac{9\frakc_m^2}{16}- 1 \leq \DRplaceholder^{(i)}    \leq 1 \Bigg) 
 +I\Bigg(\DRplaceholder^{(i)} >1\Bigg) + \Bigg(\frac{9\frakc_m^2}{16}- 1 \Bigg)I\Bigg(\DRplaceholder^{(i)}   < \frac{9\frakcm^2}{16}- 1\Bigg).
\end{equation*}
With these notions, we can state the bounds for the right-hand-side terms of \eqref{two_quantities_to_bdd}.

\begin{lemma}[Randomized concentration inequality] \label{lem:lem1}
 Let $W$, $D_1$, $D_2$ be as defined in \secref{mainpf} for $T_n$ and $T_n^*$. Under the assumptions of 
\thmref{main} and \eqref{sigma_equal_1_assumption}, for any  $x \geq 1$,
\begin{multline*} 
P\left( x \left(1 +\barDRplaceholder\right)^{1/2}  \leq W_b + \bar{D}_{1, x} \leq x \left( 1+   \bar{\frakD}_2/2\right)\right) \leq  C x e^{- \frakcm x/4} \times\\
\Bigg\{
 \bE\Big[(1 +e^{W_b})  \bar{\frakD}_2^2\Big]  +
\sum_{i=1}^n \Big(\bE[|\xi_i|^3] +
  \Big\|\xi_{i}\Big\|_2 \Big\|\bar{D}_{1, x} - \bar{D}_{1, x}^{(i)}\Big\|_2  + 
  \Big\| \xi_{i}  \Big\|_3 \Big\|\bar{\frakD}_2 - \bar{\frakD}_2^{(i)} \Big\|_{3/2} 
  \Big)
\Bigg\}
\end{multline*}
 where $D_1^{(i)}$, $D_2^{(i)}$ are random variables such that $\xi_i$ is independent of $(W - \xi_i, D_1^{(i)}, D_2^{(i)})$. 
\end{lemma}

\begin{lemma}[Nonuniform Berry-Esseen bound for $W_b  + \barD_{1,x} - \frac{x}{2}  \bar{\frakD}_2 $] \label{lem:lem2}
 Assuming $\max_{1 \leq i \leq n} \|\xi_i\|_3< \infty$, for any  $x \geq 1$,
\begin{multline}
 \left|P\left(W_b  + \barD_{1,x} - \frac{x}{2} \bar{\frakD}_2> x\right) - \bar{\Phi}(x)\right|\leq x \Bigg|\bE[\frakD_2 f_x(W_b)]\Bigg| +   C(m) e^{-c(m)x} \times\\
 \Bigg\{  \sum_{i=1}^n \bE[|\xi_i|^3]+ \sum_{i=1}^n \Big(  \|\xi_i\|_2  \| \bar{D}_{1,x} - \bar{D}_{1, x}^{(i)}\|_2 +  \|\xi_i\|_3  \| \bar{\frakD}_2 - \bar{\frakD}_2^{(i)}\|_{3/2}  \Big) +  \|\barD_{1, x}\|_2 +  \bE[ (1 + e^{W_b})\frakD_2^2  ]\Bigg\}
 \end{multline}
\end{lemma}

Combining \lemsref{lem1} and  \lemssref{lem2} with \eqref{two_quantities_to_bdd}, we get
\begin{multline} \label{last_lap}
\Big|P\Big(W_b + \barD_{1, x} > x\Big(1 +  \barDRplaceholder \Big)^{1/2}\Big) - \barPhi(x)\Big| \leq x \Bigg|\bE[\frakD_2 f_x(W_b)]\Bigg|  +   C(m) e^{-c(m)x}  \times \\
\Bigg\{
 \bE\Big[(1 +e^{W_b})  \frakD_2^2\Big]  + \| D_1\|_2 +
\sum_{i=1}^n \Bigg(\bE[|\xi_i|^3] +
  \Big\|\xi_{ i}\Big\|_2 \Big\| D_{1} - D_{1}^{(i)}\Big\|_2  + 
  \Big\| \xi_{i}  \Big\|_3 \Big\| \frakD_2 -  \frakD_2^{(i)} \Big\|_{3/2} \Bigg)
\Bigg\}
\end{multline}
by \propertyref{censoring_property} and $\|\bar{D}_{1, x}\| \leq \|D_1\|_2$. At this point, we define the  typical quantities:
\[
 \quad \beta_2 \equiv  \sum_{i =1}^n \mathbb{E}[ \xi_i^2 I(|\xi_i| > 1 )]
\quad \text{ and } \quad \beta_3 \equiv \sum_{i=1}^n \bE[\xi_i^3 I(|\xi_i| \leq 1)],
\]
which has the property
$
\beta_2 + \beta_3 \leq \sum_{i=1}^n\bE[|\xi_i|^3]
$. The following bounds allow us  to arrive at  \eqref{intermediate_nonunif_bdd_2} from \eqref{last_lap}, 
\begin{equation} \label{D1l2norm_bdd}
\|D_1\|_2 \leq \frac{(m-1) \|h\|_2}{\sqrt{m(n-m+1)}}  \leq \frac{C(m)\|h\|_2}{\sqrt{n}}.
\end{equation}
\begin{equation} \label{D1minusD1i_l2_norm_bdd}
\|D_1 - D_1^{(i)}\|_2 \leq  \frac{\sqrt{2}(m-1)  \|h\|_2}{\sqrt{nm(n-m+1) }}  \leq \frac{C(m)\|h\|_2}{n}.
\end{equation}
\begin{equation} \label{expD2square_bdd}
\bE[\frakD_2^2] \leq   C(m)  \Bigg( \frac{\bE[|g|^3]}{\sqrt{n}} +   \frac{\|g\|_3 \|h\|_3}{\sqrt{n}}\Bigg)
\end{equation}
\begin{equation} \label{expect_expWbD22_bdd}
\bE[e^{W_b}\frakD_2^2] \leq C(m)  \Bigg( \frac{\bE[|g|^3]}{\sqrt{n}} +   \frac{\|g\|_3 \|h\|_3}{\sqrt{n}}\Bigg)
\end{equation}
\begin{equation} \label{expect_D2fxWb_bdd}
\bigg|\bE[\frakD_2 f_x(W_b)]\bigg| \leq C(m) e^{-x} \Bigg( \frac{\bE[|g|^3]}{\sqrt{n}} +   \frac{\|g\|_3 \|h\|_3}{\sqrt{n}}\Bigg)
\end{equation}
\begin{equation} \label{D2minusD2i_l3over2_norm_bdd}
\|\frakD_2 - \frakD_2^{(i)}\|_{3/2} \leq C(m) \Bigg( \frac{\|g\|_3^2}{n}  +  \frac{\|g\|_3 \|h\|_3}{n}\Bigg)
\end{equation}
These bounds are proved as follows:

\eqref{D1l2norm_bdd} and \eqref{D1minusD1i_l2_norm_bdd}: The proofs can be found in   \citet[Lemma 10.1]{chen2011normal}.

\eqref{expD2square_bdd}: From \eqref{placeholder_D2_def}, we have
\begin{align*}
\bE[\frakD_2^2] &\leq   C \bE[\delta_{0, b}^2  + \bardelta_{2, b}^2 +  n^{-1} + \beta_2^2] \\
&\leq C \bE[\delta_{0, b}^2  + |\delta_{2, b}| +  n^{-1} + \beta_2]  \text{ since $\bardelta_{2, b}, \beta_2 \leq 1$}\\
&\leq C (\sum_{i=1}^n\bE[|\xi_i|^3] + \|\delta_{2, b}\|_2 + n^{-1}) \text{ by \eqref{bdd_exp_delta0b_sq},  $\beta_2 \leq \sum_{i=1}^n \bE[|\xi_i|^3]$ } \\
&\leq C(m)  \Bigg( \frac{\bE[|g|^3]}{\sqrt{n}} +   \frac{\|g\|_3 \|h\|_3}{\sqrt{n}}\Bigg)\text{ by \eqref{delta2b_l2_bdd} and $\|g\|_3 \|h\|_3 \geq 1$}.
\end{align*}

\eqref{expect_expWbD22_bdd}:  From \eqref{placeholder_D2_def} again, we have
\begin{align*}
&\bE[e^{W_b}\frakD_2^2] \\
&\leq C \bE[e^{W_b} (\delta_{0, b}^2  + |\delta_{2, b}| +  n^{-1} + \beta_2)]  \text{ since $\bardelta_{2, b}, \beta_2 \leq 1$}\\
&\leq C \Big(\sum_{i=1}^n\bE[|\xi_{ i}|^3] +  \|\delta_{2, b}\|_2 + n^{-1}\Big)\\
& \hspace{1cm} \text{by \eqref{bdd_exp_ew_delta0b}, Bennett's inequality (\lemref{bennett_censored}) and $\beta_2 \leq \sum_{i=1}^n \bE[|\xi_i|^3]$}\\
&\leq C(m)  \Big\{ \frac{\bE[|g|^3]}{\sqrt{n}} +   \frac{\|g\|_3 \|h\|_3}{\sqrt{n}}\Big\} \text{ by \eqref{delta2b_l2_bdd}}
\end{align*}

\eqref{expect_D2fxWb_bdd}: From \eqref{placeholder_D2_def} again, we get
\begin{align*}
\bE[\frakD_2 f_x(W_b)] &\leq C \Big\{ |\bE[f_x(W_b)(\delta_{0, b}  + \bardelta_{2, b})] |+  e^{-x}( n^{-1/2} + \beta_2) \Big\} \text{ by  \lemref{expect_fxWB}}\\
&\leq C(m) e^{-x} \Bigg\{  \frac{\bE[|g|^3]}{\sqrt{n}} +   \frac{\|g\|_3 \|h\|_3}{\sqrt{n}}\Bigg\} \text{ by \eqref{bdd_fx_delta0b}, \eqref{delta2b_l2_bdd} and  \eqref{bdd_fx_delta2b}} 
\end{align*}

\eqref{D2minusD2i_l3over2_norm_bdd}: Since 
$
\frakD_2 - \frakD_2^{(i)} = d_n^2 (\xi_{b, i}^2 + \bardelta_{2, b} -  \bardelta_{2, b}^{(i)})
$,
\[
\|\frakD_2 - \frakD_2^{(i)}\|_{3/2} \leq C \Bigg\{ (\bE[|\xi_{b, i}|^3])^{2/3}  + \|\delta_{2, b} -  \delta_{2, b}^{(i)}\|_{3/2}\Bigg\}.
\]
In our related work \citet[Appendix E]{leungshaounif}, we have already shown that 
\[
\|\delta_{2, b} -  \delta_{2, b}^{(i)}\|_{3/2} \leq  \frac{C(m)\|g\|_3 \|h\|_3}{n},
\]
so we get
\[
\|\frakD_2 - \frakD_2^{(i)}\|_{3/2} \leq C(m) \Bigg( \frac{\|g\|_3^2}{n}  +  \frac{\|g\|_3 \|h\|_3}{n}\Bigg).
\]
(Note that $\delta_{2, b} -  \delta_{2, b}^{(i)}$ is precisely the quantity "$A + B$" appearing in \citet[Appendix E.2]{leungshaounif})

\hspace{3cm}

It remains to prove  \lemsref{lem1} and  \lemssref{lem2}, which is the focus next.

\subsection{Proof of \lemref{lem1}} \label{sec:pf_lem_1} 

If $\sum_{i=1}^n \bE[|\xi_i|^3] \geq 2$, we will have 
\begin{align*}
&P\left( x \left(1 +\bar{\frakD}_2 \right)^{1/2}  \leq W_b + \bar{D}_{1, x} \leq x \left( 1+   \bar{\frakD}_2/2\right)\right) \\
&\leq P\Big(\frac{3\frakcm  x }{4} \leq W_b + \bar{D}_{1, x}\Big)\\
&\leq P\Big(W_b \geq  \frac{\frakcm x}{2} \Big) 
\leq e^{- \frakcm x/2} \mathbb{E}[e^{W_b}] \leq  C  e^{- \frakcm x/2} \sum_{i=1}^n \bE[|\xi_i|^3],
\end{align*}
since the Bennett's inequality (\lemref{bennett_censored}) implies $\mathbb{E}[e^{W_b}] \leq C  \bE[|\xi_i|^3]$ for some absolute constant $C > 0$, and  \lemref{lem1}  follows because $x \geq 1$.

If 
$\sum_{i=1}^n \bE[|\xi_i|^3] < 2$,
since $x (1 + \bar{\frakD}_2)^{1/2} \geq \frac{3 \frakcm x}{4}$, it must be less that 
\begin{align}
&P\left( x \left(1 + \bar{\frakD}_2 \right)^{1/2}  \leq W_b + \bar{D}_{1, x} \leq x \left( 1+   \bar{\frakD}_2 /2\right)\right)  \notag\\
&e^{- (3 \frakcm x)/8} \mathbb{E}\left[e^{ (W_b + \bar{D}_{1, x})/2 }I\left\{x \Big(1 + \bar{\frakD}_2\Big)^{1/2}  \leq W_b  + \bar{D}_{1, x} \leq x \Big( 1+   \bar{\frakD}_2/2\Big)\right\}\right]\notag\\
&\leq e^{-( \frakcm  x)/4} \mathbb{E}\left[e^{ W_b /2 }I\left\{x \Big(1 + \bar{\frakD}_2\Big)^{1/2}  \leq W_b  + \bar{D}_{1, x} \leq x \Big( 1+  \bar{\frakD}_2/2\Big)\right\}\right]\notag\\
&\leq e^{-(\frakcm x)/4} \bE\Bigg[e^{W_b/2} I\Big(x + \frac{x(\bar{\frakD}_2 -  \bar{\frakD}_2^2)}{2} - \bar{D}_{1, x}  \leq W_b \leq x + \frac{x  \bar{\frakD}_2}{2} - \bar{D}_{1, x}\Big)\Bigg], \label{must_be_less_than}
\end{align}
the last inequality follows from the fact that
\[
1 + s/2 - s^2/2 \leq (1 + s)^{1/2} \text{ for all } s \geq -1.
\]
Continuing from \eqref{must_be_less_than}, by the exponential randomized concentration inequality for a sum of censored variables (\lemref{modified_RCI_bdd}), we have
\begin{align}
&e^{(\frakcm x)/4}P\left( x \left(1 + \bar{\frakD}_2 \right)^{1/2}  \leq W_b + \bar{D}_{1, x} \leq x \left( 1+   \bar{\frakD}_2/2\right)\right)\notag\\
&\leq 
\left( \mathbb{E}\left[e^{ W_b}\right] \right)^{1/2}\exp\left( - \frac{1}{16(\beta_2 + \beta_3)^2}\right) \notag\\
&+ C e^{(\beta_2 + \beta_3)/8}
\Biggl\{ \sum_{i=1}^n
\mathbb{E} [  |\xi_{b,  i}|e^{W_b^{(i)} /2} (|\bar{D}_{1, x} -\bar{D}_{1, x}^{(i)} | + x |\bar{\frakD}_2 - \bar{\frakD}_2^{(i)}| )  ] 
\notag\\
&
+ \mathbb{E}\left[|W_b|e^{W_b/2 }\left(x \bar{\frakD}_2^2 + \beta_2 + \beta_3\right) \right]\notag\\
& +   \sum_{i=1}^n  \Big|\bE[ \xi_{b, i}]\Big| \bE\left[e^{ W_b^{(i)} /2 }\left(x (\bar{\frakD}_2^{(i)})^2  + \beta_2 + \beta_3\right)\right]
\Biggr\}, \label{1st_applied_RCI}
\end{align}
where we have used the fact that $|\bar{\frakD}_2|\vee| \bar{\frakD}_2^{(i)}| \leq 1$, which implies
\begin{align*}
|\bar{\frakD}_2^2 - (\bar{\frakD}_2^{(i)})^2| &= |(\bar{\frakD}_2- \bar{\frakD}_2^{(i)})(\bar{\frakD}_2 + \bar{\frakD}_2^{(i)})| \leq 2| \bar{\frakD}_2- \bar{\frakD}_2^{(i)}|.
\end{align*}
It remains to bound the  terms on the right hand side of \eqref{1st_applied_RCI}.

 First,
\begin{equation} \label{bennett_exp_ewb}
\bE[e^{W_b} ] \leq C \text{ for some } C >0 \text{ by  Bennett's inequality (\lemref{bennett_censored})}
\end{equation}
and
\begin{align}
\exp\left(\frac{-1}{16 (\beta_2 + \beta_3)^2} \right) 
&\leq C (\beta_2 + \beta_3) \leq C \sum_{i=1}^n \bE[|\xi_i|^3].
\end{align}
Secondly, 
\begin{align}
&\mathbb{E} [  |\xi_{b,  i}|e^{W_b^{(i)} /2} (|\bar{D}_{1, x} -\bar{D}_{1, x}^{(i)} | + x |\bar{\frakD}_2 - \bar{\frakD}_2^{(i)}| )  ]  \notag\\
&\leq \| \xi_{b,  i} e^{W_b^{(i)} /2}\|_2 \|\barD_{1, x} -{D}_{1, x}^{(i)}\|_2 + x \| \xi_{b,  i} e^{W_b^{(i)} /2} \|_3 \|\bar{\frakD}_2 - \bar{\frakD}_2^{(i)}) \|_{3/2} \notag\\
&\leq C   \Big\{\Big\|\xi_{b, i}\Big\|_2 \Big\|{D}_{1, x} - {D}_{1, x}^{(i)}\Big\|_2  + x\| \xi_{b,  i}  \|_3 \|\bar{\frakD}_2 - \bar{\frakD}_2^{(i)}) \|_{3/2}\Big\} , \label{bdd_sth_else1}
\end{align}
where the last inequality uses that for any $2 \leq p \leq 3$, $\| \xi_{b,  i} e^{W_b^{(i)} /2}\|_p \leq C \| \xi_{b,  i}\|_p$ by the independence of $\xi_{b, i}$  and $e^{W_b^{(i)}/2}$, as well as Bennett's inequality (\lemref{bennett_censored}). Thirdly, since 
$
|W_b|/2 \leq e^{|W_b|/2} \leq e^{W_b/2} + e^{-W_b/2}
$,
by the Bennett's inequality (\lemref{bennett_censored}) again,
\begin{equation}
\mathbb{E}\left[|W_b|e^{W_b/2 }\left(x \bar{\frakD}_2^2 + \beta_2 + \beta_3\right) \right] \leq C\Big( x\bE[(1 +e^{W_b})  \bar{\frakD}_2^2]  + \sum_{i=1}^n \bE[|\xi_i|^3] \Big)    \label{bdd_sth_else2}
\end{equation}
Lastly, by  \lemref{exp_xi_bi_bdd},  Bennett's inequality (\lemref{bennett_censored}) and the current assumption that  $\beta_2 + \beta_3 \leq \sum_{i=1}^n \bE[|\xi_i|^3] < 2$, we have
\begin{equation}
 \sum_{i=1}^n  \Bigg|\bE[ \xi_{b, i}]\Bigg| \bE\left[e^{ W_b^{(i)} /2 } \underbrace{\left(x (\bar{\frakD}_2^{(i)})^2  + \beta_2 + \beta_3\right)}_{\leq Cx \text{ as } x \geq 1}\right] 
 \leq C x \sum_{i=1}^n\bE[ \xi_i^3 ] . \label{bdd_sth_else3}
\end{equation}
Combining \eqref{1st_applied_RCI}-\eqref{bdd_sth_else3} with $x\geq 1$,  \lemref{lem1} is proved when $\sum_{i=1}^n \bE[|\xi_i|^3] < 2$.

\subsection{Proof of \lemref{lem2}} \label{app:pf_lem_2} 
We shall equivalently bound
\begin{equation} \label{bdd_with_stein}
|P( W_b  + \barD_{1,x} - \frac{x}{2} \barDRplaceholder  \leq x) - \Phi(x)|.
\end{equation}

We first let $X_1^*, \dots, X_n^*$ be independent copies of $X_1, \dots, X_n$ and define
\begin{multline*}
D_{1, i^*}\equiv D_1 (X_1, \dots, X_{i-1}, X_i^*, X_{i+1}, \dots, X_n) \text{ and }\\
 \frakD_{2, i^*} \equiv \frakD_2(X_1, \dots, X_{i-1}, X_i^*, X_{i+1}, \dots, X_n) \text{ for each } i \in \{1,\dots, n\}, 
\end{multline*}
which are versions of $D_1$ and $\frakD_2$ with $X_i^*$ replacing $X_i$ as input. In analogy to \eqref{D1x_upper_censored} and \eqref{placeholder_D2_def_censored}, also define their correspondingly censored versions
\begin{multline*}
\barD_{1, i^*, x}\equiv D_{1, i^*} I\left(|D_{1, i^*}| \leq \frac{\frakcm x}{4} \right) +\frac{\frakcm x}{4} I\left(D_{1, i^*} >\frac{\frakcm x}{4}\right) - \frac{\frakcm x}{4} I\left(D_{1, i^*} < - \frac{ \frakcm  x}{4}\right) \\
\text{ and }\\
 \bar{\frakD}_{2, i^*} \equiv   \frakD_{2, i^*} I\Bigg(\frac{9\frakc_m^2}{16}- 1 \leq  \frakD_{2, i^*}   \leq 1 \Bigg)
 +I\Bigg( \frakD_{2, i^*} >1\Bigg) + \Big(\frac{9\frakc_m^2}{16}- 1 \Big)I\Bigg( \frakD_{2, i^*}   < \frac{9\frakcm^2}{16}- 1\Bigg)
\end{multline*}
By letting
\begin{equation} \label{Delta_def}
\Delta \equiv \bar{D}_{1, x}  -\frac{x \bar{\frakD}_2 }{2} \text{ and }
\Delta_{i^*} \equiv \bar{D}_{1, i^*, x}  - \frac{x \bar{\frakD}_{2, i^*} }{2},
\end{equation}
one can write the difference in \eqref{bdd_with_stein} as 
\begin{align}
P(W_b+ \Delta \leq x) - \Phi(x) =  \; &\mathbb{E}\left[f_x'(W_b + \Delta)\right] - \mathbb{E}\left[ W_b   f_x(W_b  + \Delta)\right]  - \mathbb{E}\left[ \Delta  f_x(W_b + \Delta)\right] \notag\\
=  \; &E_1 + E_2 + E_3
\end{align}
where 
\begin{align*}
E_1 &\equiv   \sum_{i=1}^n   \mathbb{E} \left[  \int_{-1}^1  \left\{ f_x'(W_b+ \Delta) - f_x'(  W_b^{(i)}  + \Delta_{i^*} + t )  \right\} k_{b, i} (t)  dt \right] \\
E_2 &\equiv    \sum_{i=1}^n \mathbb{E} \left[ \xi_{b, i} \left\{  f_x(W_b + \Delta_{i^*}) - f_x(W_b + \Delta)  \right\}\right], \\
E_3 &\equiv    - \sum_{i =1}^n \mathbb{E} [\xi_{b, i} f_x( W_b^{(i)} +\Delta_{i^*}) ]  +  \mathbb{E}[ f_x'(W_b  + \Delta)] \sum_{i=1}^n \bE[(\xi_i^2 - 1) I(|\xi_i| >1)]- \mathbb{E}[ \Delta f_x(W_b + \Delta) ], 
\end{align*}
with $k_{b, i}$ defined as 
\begin{equation*} \label{bar_K_fn_def}
k_{b, i} (t) \equiv \mathbb{E}\left[\xi_{b, i} \{  I(0 \leq t \leq  \xi_{b, i}) - I(\xi_{b, i} \leq t <0)\}\right],
\end{equation*}
which has the properties
\begin{equation}\label{barK_property}
 \int_{-1}^1 k_{b, i}i (t) dt = \mathbb{E}[\xi_{b, i}^2] \leq \bE[\xi_i^2] \quad \text{ and } \quad \int_{-1}^1 |t| k_{b, i} (t) dt = \frac{   \mathbb{E}[|\xi_{b, i}|^3]}{2} \leq \frac{\bE[|\xi_i|^3]}{2}.
\end{equation}
We will establish that  
\begin{equation} \label{E1bdd}
|E_1| \leq   C e^{-c(m)x} \Bigg\{  \sum_{i=1}^n \bE[|\xi_i|^3]+ \sum_{i=1}^n \Big(  \|\xi_i\|_2  \| \bar{D}_{1,x} - \bar{D}_{1, x}^{(i)}\|_2 +  \|\xi_i\|_3 \| \bar{\frakD}_2 - \bar{\frakD}_2 ^{(i)}\|_{3/2}   \Big)\Bigg\},
\end{equation}
\begin{equation}
|E_2 |\leq C e^{-c x} \sum_{i=1}^n \Bigg( \|\xi_{b, i}\|_2 \|\bar{D}_{1,x} - \bar{D}_{1,x}^{(i)} \|_2 +  \|\xi_{b, i}\|_3  \| \bar{\frakD}_2 - \bar{\frakD}_2 ^{(i)}\|_{3/2}   \Bigg) \label{E2bdd}.
\end{equation}
and
\begin{equation}
|E_3 |\leq   C (m)e^{-cx} \Bigg(   \sum_{i=1}^n \bE[|\xi_i|^3]+ \|\barD_{1, x}\|_2 +  \bE[ (1 + e^{W_b})\frakD_2^2  ]\Bigg) + x \Big|\bE[\frakD_2  f_x(W_b)]\Big| \label{E3bdd},
\end{equation}
from which \lemref{lem2} can be concluded.

Define,
 for any pair $1 \leq i , j \leq n$,
\[
D_{1, i^*}^{(j)} \equiv   \begin{cases} 
   D_1^{(j)} ( X_1, \dots, X_{i - 1}, X_i^*, X_{i+1}, \dots, X_{j-1}, X_{j+1}, \dots, X_n ) & \text{if } i < j ; \\
  D_1^{(j)} ( X_1, \dots, X_{j-1}, X_{j+1},  \dots, X_{i - 1}, X_i^*, X_{i+1}, \dots,  X_n )      & \text{if } j < i ; \\
    D_1^{(j)} ( X_1,  \dots, X_{i - 1},  X_{i+1}, \dots,  X_n )      & \text{if } i = j, \\
  \end{cases}
\]
and
\[
\frakD_{2, i^*}^{(j)} \equiv   \begin{cases} 
   \frakD_2^{(j)} ( X_1, \dots, X_{i - 1}, X_i^*, X_{i+1}, \dots, X_{j-1}, X_{j+1}, \dots, X_n ) & \text{if } i < j ;\\
   \frakD_2^{(j)} ( X_1, \dots, X_{j-1}, X_{j+1},  \dots, X_{i - 1}, X_i^*, X_{i+1}, \dots,  X_n )      & \text{if }  j < i; \\
     \frakD_2^{(j)} ( X_1,  \dots, X_{i - 1},  X_{i+1}, \dots,  X_n )      & \text{if } i = j,  \\
  \end{cases}
\]
i.e., $D_{1, i^*}^{(j)}$ and $\frakD_{2, i^*}^{(j)} $ are versions of $D_1^{(j)}$ and $\frakD_{2}^{(j)}$ with $X_i^*$ replacing $X_i$ as input; likewise, they have their censored variants 
\[
\barD_{1,i^*, x}^{(j)} \equiv D_{1, i^*}^{(j)} I\left(|D_{1, i^*}^{(j)}| \leq \frac{\frakcm x}{4} \right) +\frac{\frakcm x}{4} I\left(D_{1, i^*}^{(j)}>\frac{\frakcm x}{4}\right) - \frac{\frakcm x}{4} I\left(D_{1, i^*}^{(j)} < - \frac{ \frakcm  x}{4}\right).
\]
and
\begin{equation*}
\bar{\frakD}_{2 , i^*}^{(j)} \equiv \frakD_{2, i^*}^{(j)}  I\Bigg(\frac{9\frakc_m^2}{16}- 1 \leq \frakD_{2, i^*}^{(j)}   \leq 1 \Bigg) 
 +I\Bigg(\frakD_{2, i^*}^{(j)}>1\Bigg) + \Big(\frac{9\frakc_m^2}{16}- 1 \Big)I\Bigg(\frakD_{2, i^*}^{(j)}   < \frac{9\frakcm^2}{16}- 1\Bigg).
\end{equation*}
We will first prove the two bounds in \eqref{E1bdd} and \eqref{E2bdd}  for $E_1$ and $E_2$, which will use the following two properties:

\begin{property} \label{property:indep_copy_equal_distr}
For any $i , j\in \{1, \dots, n\}$,
 \begin{equation*}
 \|\bar{D}_{1, i^*, x} -\bar{D}_{1,  i^*, x}^{(j)} \|_2 =  \|\bar{D}_{1, x} -\bar{D}_{1, x}^{(j)} \|_2 \text{ and } \|\bar{\frakD}_{2, i^*} -\bar{\frakD}_{2, i^*}^{(j)} \|_{3/2} =  \|\bar{\frakD}_2 -\bar{\frakD}_2^{(j)} \|_{3/2}.
 \end{equation*}
\end{property}

\begin{property} \label{property:indep_copy_tri_ineq} For  any $i \in \{1, \dots, n\}$, 
\begin{equation*} 
\|\bar{D}_{1, x} - \bar{D}_{1, i^*, x}\|_2 \leq   2\|\bar{D}_{1, x} - \bar{D}_{1, x}^{(i)}\|_2 \text{ and }
\|\bar{\frakD}_2 - \bar{\frakD}_{2, i^*} \|_{3/2} \leq   2\|\bar{\frakD}_2- \bar{\frakD}_{2}^{(i)}\|_{3/2}
\end{equation*}
\end{property}

\begin{proof}[Proof of \propertysref{indep_copy_equal_distr} and  \propertyssref{indep_copy_tri_ineq}]
Note that \propertyref{indep_copy_equal_distr} is true because  $X_1^*, \dots, X_n^*$ are independent copies of $X_1, \dots, X_n$, and \propertyref{indep_copy_tri_ineq} is true because of 
the triangular inequalities
\[
\|\bar{D}_{1, x} - \bar{D}_{1, i^*, x}\|_2 \leq  \|\bar{D}_{1, x} - \bar{D}_{1, x}^{(i)}\|_2  + \underbrace{ \| \bar{D}_{1, x}^{(i)} - \bar{D}_{1, i^*, x}\|_2}_{ = \| \bar{D}_{1, x}^{(i)} - \bar{D}_{1, x} \|_2  }
\]
 and 
 \[
 \|\bar{\frakD}_2 - \bar{\frakD}_{2, i^*}\|_{3/2} \leq  \|\bar{\frakD}_2 - \bar{\frakD}_2 ^{(i)}\|_{3/2}  + \underbrace{ \| \bar{\frakD}_2 ^{(i)} - \bar{\frakD}_{2, i^*}\|_{3/2}}_{\| \bar{\frakD}_2 ^{(i)}- \bar{\frakD}_2 \|_{3/2}},
 \] 
 as well as \propertyref{indep_copy_equal_distr}.
\end{proof}

\subsubsection{Proof of the bound for $E_1$,  \eqref{E1bdd}} \label{app:bdd_E1}

Recall 
\[
E_1  = \sum_{i=1}^n   \mathbb{E} \left[  \int_{-1}^1  \left\{ f_x'(W_b+ \Delta) - f_x'(  W_b^{(i)}  + \Delta_{i^*} + t )  \right\} k_{b, i} (t)  dt \right].
\]
Let $g_x(w) = (w f_x(w))'$ be as defined in \eqref{gx_def}, and let 
\begin{equation*}
\eta_1 =  t + \Delta_{i^*} \text{ and } 
\eta_2 = \xi_{b, i} + \Delta.
\end{equation*}
By Stein's equation
 \eqref{steineqt},
 one can write 
\[
E_1 = E_{11} + E_{12},
\]
where 
\begin{align*}
E_{11}  
&=  \sum_{i=1}^n  \int_{-1}^1 \bE \Big[ \int_{t + \Delta_{i^*}}^{\xi_{b, i} +  \Delta} g_x(W_b^{(i)} + u) du\Big] k_{b, i} (t) dt\\
&= \underbrace{\sum_{i=1}^n \int_{-1}^1 \bE\Bigg[ \int g_x(W_b^{(i)} + u) I( \eta_1 \leq u \leq \eta_2) du\Bigg] k_{b, i} (t) dt}_{E_{11.1}} \\
&\hspace{1cm} - \underbrace{\sum_{i=1}^n \int_{-1}^1 \bE\Bigg[ \int g_x(W_b^{(i)} + u) I(\eta_2\leq u \leq \eta_1) du\Bigg] k_{b, i} (t) dt}_{E_{11.2}}
\end{align*}
and
\begin{align*}
E_{12}  &= \sum_{i=1}^n \int_{-1}^1   
\Bigg\{P(W_b +  \Delta \leq x) - P(W_b^{(i)}  +\Delta_{i^*} + t \leq x)\Bigg\} k_{b, i}(t) dt.
\end{align*}

We first bound the integrand of $E_{11.1}$.
Using the identity
\begin{align*}
1 &= I(W_b^{(i)} + u \leq x - 1) + I(x -1 < W_b^{(i)} + u, u \leq 3x/4 ) + I(x -1 < W_b^{(i)} + u, u > 3x/4)\\
&\leq I(W_b^{(i)} + u \leq x - 1) + I(x -1 < W_b^{(i)} + u, W_b^{(i)} +1> x/4) + (x -1 < W_b^{(i)} + u, u > 3x/4)
\end{align*}
and the bounds for $g_x(\cdot)$ in \lemref{helping}, 
with $| \Delta| \leq \frac{x |\bar{\frakD}_2|}{2} + |\barNR| \leq   \underbrace{\frac{2 +\frakcm}{4}}_{< 3/4} x$ and
$1.6 \barPhi(x) \leq x e^{1/2-x}$,
\begin{align*}
&\Bigg|\bE\Big[\int g_x(W_b^{(i)} + u) I(\eta_1 \leq  u \leq \eta_2) du\Big]\Bigg|\\
&\leq  x e^{1/2 - x} \|\eta_2 - \eta_1\|_1 + (x + 2) \Big\{ \| I(W_b^{(i)} +1 > x/4) (\eta_2 - \eta_1) \|_1 +  \| I(\eta_2 > 3x/4)  (\eta_2 - \eta_1)\|_1 \Big\}\\
&\leq  x e^{1/2 - x} \|\eta_2 - \eta_1\|_1 + \frac{x+2}{e^{x/4 -1}} \|e^{W_b^{(i)}}(\eta_2 - \eta_1)\|_1 +
\frac{x+2}{e^{3x/4}} \|e^{\xi_{b, i} + \Delta} (\eta_2  - \eta_1) \|_1\\
&\leq  \Bigg(x e^{1/2 - x} + \frac{e (x+2)}{e^{(1 - \frakcm)x/4}} \Bigg)\|\eta_2 - \eta_1\|_1 + \frac{x+2}{e^{x/4 -1}} \|e^{W_b^{(i)}}(\eta_2 - \eta_1)\|_1 \\
&\leq \frac{C(x+2)}{e^{(1 - \frakcm)x/4}}\Bigg\{ |t| + \|\Delta_{i^*} - \Delta + \xi_{b, i}\|_1 +  \| e^{W_b^{(i)}}  (\Delta_{i^*} - \Delta + \xi_{b, i})\|_1 \Bigg\},
\end{align*}
where we have used the Bennett's inequality (\lemref{bennett_censored}) via $\|e^{W_b^{(i)}} t\|_1 \leq C |t|$ in the last line. Continuing, 
\begin{align}
&\Bigg|\bE\Big[\int g_x(W_b^{(i)} + u) I(\eta_1 \leq  u \leq \eta_2) du\Big]\Bigg|\notag\\
&\leq \frac{C(x+2)}{e^{(1 - \frakcm)x/4}}\Bigg\{ |t| + \Big\| x (\bar{\frakD}_{2, i^*} - \bar{\frakD}_2) + (\bar{D}_{1,  i^*, x}- \bar{D}_{1,x}) + \xi_{b, i}\Big\|_1\notag\\
&\hspace{4cm} +  \Bigg\| e^{W_b^{(i)}}  \Big[ x (\bar{\frakD}_{2, i^*} - \bar{\frakD}_2) + (\bar{D}_{1,  i^*, x}- \bar{D}_{1,x})+ \xi_{b, i}\Big]\Bigg\|_1 \Bigg\}\notag\\
& \leq Ce^{-c(m)x} \Bigg\{|t| + \|\xi_{b, i}\|_2   +  \|\bar{D}_{1, x}^{(i)}- \bar{D}_{1,x}\|_2 +   \|\bar{\frakD}_2^{(i)} - \bar{\frakD}_2\|_{3/2} \Bigg\}\label{integrand_R11_1_bdd},
\end{align}
where the last inequality uses $\|e^{W_b^{(i)}}\|_2 \vee  \|e^{W_b^{(i)}}\|_3 < C$ (by Bennett's inequality,  \lemref{bennett_censored}) and \propertyref{indep_copy_tri_ineq}. 
By a completely analogous argument, we also have 
\begin{multline}
\Bigg|\bE\Big[\int g_x(W_b^{(i)} + u) I(\eta_2 \leq  u \leq \eta_1) du\Big]\Bigg| \\
 \leq  Ce^{-c(m)x} \Bigg\{|t| + \|\xi_{b, i}\|_2   +  \|\bar{D}_{1, x}^{(i)}- \bar{D}_{1,x}\|_2 +   \|\bar{\frakD}_{2}^{(i)} - \bar{\frakD}_2\|_{3/2} \Bigg\} \label{integrand_R11_2_bdd}
\end{multline}
for the integrand of $E_{11.2}$. 
Combining  \eqref{integrand_R11_1_bdd} and \eqref{integrand_R11_2_bdd} and  integrating over $t$, we have
\begin{equation}
|E_{11}| 
 \leq C e^{-c(m)x}  \Bigg\{\sum_{i=1}^n \|\xi_{b, i}\|_3^3   +\sum_{i=1}^n \|\xi_{b, i}\|_2 \|\bar{D}_{1, x}^{(i)}- \bar{D}_{1,x}\|_2 +  \|\xi_{b, i}\|_3  \|\bar{\frakD}_2^{(i)}- \bar{\frakD}_2\|_{3/2} \Bigg\}  \label{R_11_bdd}
\end{equation}
where we have used \eqref{barK_property} and $\|\xi_{b, i}\|_2^3 \leq \|\xi_{b, i}\|^3_3$ and $\|\xi_{b, i}\|_2^2 \leq \|\xi_{b, i}\|_2 \leq \|\xi_{b, i}\|_3$.

For $E_{12}$, its integrand is bounded by
\begin{multline} \label{two_probs}
 P(x - \Delta - \xi_{b, i}\leq W_b^{(i)} \leq x  - \Delta_{i^*}- t )+ P( x - \Delta_{i^*}  - t \leq  W_b^{(i)} \leq  x - \Delta - \xi_{b, i})
\end{multline}
Since  $0 < \frakcm < 1$ implies that 
\[
\min(x - \Delta- \xi_{b, i}, x- \Delta_{i^*} - t ) \geq  x - \frac{(2+ \frakcm)x}{4}- 1  \geq \frac{x}{4}- 1\quad \text{ for } \quad  |t| \leq 1,
\]
by defining 
\begin{equation*}
W_b^{(i, j)} \equiv W_b - \xi_{b, i} - \xi_{b, j} \text{ and }
 \Delta^{(j)}_{i^*} \equiv \bar{D}_{1, i^*, x}^{(j)}  - \frac{x \bar{\frakD}_{2, i^*}^{(j)} }{2}  \text{ for } 1 \leq i\neq j \leq n,
\end{equation*}
 we can apply the randomized concentration inequality (\lemref{modified_RCI_bdd})    to bound \eqref{two_probs} as
\begin{align}
&C e^{-x/8} \Bigg\{\beta_2 + \beta_3 + 2 \sum_{\substack{j=1\\ j \neq i}}^n \bE\Big[ |\xi_{b, j}|e^{W_b^{(i, j)}/2} ( |\Delta - \Delta^{(j)}| +|\Delta_{i^*} - \Delta_{ i^*}^{(j)}|   )\Big] \notag\\
& \hspace{3cm}+ \bE\Big[ |W_b^{(i)}|e^{W_b^{(i)}/2} \Big( |\Delta - \Delta_{i^*}| + |\xi_{b, i}| + |t| + \beta_2 + \beta_3 \Big)\Big] \notag\\
& \hspace{3cm}+  \sum_{\substack{j=1\\ j \neq i}}^n  \Big|\bE[ \xi_{b, j}]\Big| \bE\Big[ e^{W_b^{(i,j)}/2} \underbrace{\Big(|t|  +|\xi_{b, i}| + |\Delta^{(j)} - \Delta_{i^*}^{(j)}| + \beta_2 + \beta_3\Big)}_{\leq C(1 +x)} \Big] 
\Bigg\}\notag\\
&\leq C e^{-x/8} \Bigg\{x\beta_2 + \beta_3  +  \sum_{ \substack{j=1\\j \neq i}}^n \Bigg[  \|\xi_{b, j}\|_2 ( \| \bar{D}_{1,x} - \bar{D}_{1, x}^{(j)}\|_2 )+ x \|\xi_{b, j}\|_3 ( \| \bar{\frakD}_2 - \bar{\frakD}_2^{(j)}\|_{3/2})  \Bigg] \notag\\
& \hspace{3cm}  +\| \bar{D}_{1,x} - \bar{D}_{1, i^*, x}\|_2 + x\|  \bar{\frakD}_2 - \bar{\frakD}_{2, i^*}\|_{3/2} + \|\xi_{b, i}\|_2 + |t| 
\Bigg\}\label{bdd_integrand_R12} \\
&\leq C e^{-x/8} \Bigg\{x\beta_2 + \beta_3  +  \sum_{ \substack{j=1\\j \neq i}}^n \Bigg[  \|\xi_{b, j}\|_2  \| \bar{D}_{1,x} - \bar{D}_{1, x}^{(j)}\|_2  + x \|\xi_{b, j}\|_3  \| \bar{\frakD}_2 - \bar{\frakD}_2^{(j)}\|_{3/2}  \Bigg] \notag\\
& \hspace{3cm} 2 \| \bar{D}_{1,x} - \bar{D}_{1, x}^{(i)}\|_2 +2 x\|  \bar{\frakD}_2 - \bar{\frakD}_{2}^{(i)}\|_{3/2} + \|\xi_{b, i}\|_2 + |t| 
\Bigg\}\label{bdd_integrand_R12_final},
\end{align}
where 
\begin{enumerate}
\item 
to attain \eqref{bdd_integrand_R12},  we have used that $ \Big|\bE[ \xi_{b, i}]\Big| \leq \bE[ \xi_i^2I(|\xi_i| > 1) ] $ from \lemref{exp_xi_bi_bdd}, $|W_b^{(i)}| e^{W_b^{(i)}/2} \leq 2(1 + e^{W_b^{(i)}})$, the Bennett's inequality (\lemref{bennett_censored}) and applied \propertyref{indep_copy_equal_distr} on $|\Delta_{i^*} - \Delta_{i^*}^{(j)}|$;
\item
to attain \eqref{bdd_integrand_R12_final}, we have used \propertyref{indep_copy_tri_ineq}.
\end{enumerate}
From \eqref{bdd_integrand_R12_final}, on integration with respect to $t$,   for absolute constants $C, c >0$,
\begin{multline} \label{R12_final_bdd}
|E_{12}| \leq C e^{-cx} \Bigg\{  \sum_{i=1}^n \bE[|\xi_i|^3]+ 
\sum_{i=1}^n  \|\xi_i\|_2   \| \bar{D}_{1,x} - \bar{D}_{1, x}^{(i)}\|_2 
+ \sum_{i=1}^n\|\xi_i\|_3 \| \bar{\frakD}_2 - \bar{\frakD}_2 ^{(i)}\|_{3/2} \Bigg\}
\end{multline}
by the properties of the K-function in  \eqref{barK_property}, $\|\xi_{b, i}\|_2^3 \leq \|\xi_{b, i}\|^3_3$ and $\|\xi_{b, i}\|_2^2 \leq \|\xi_{b, i}\|_{2} \leq \|\xi_{b, i}\|_3$.

Lastly, combining \eqref{R_11_bdd} and \eqref{R12_final_bdd}, we obtain \eqref{E1bdd}.

\subsubsection{Proof of bounds for $E_2$, \eqref{E2bdd}} \label{app:bdd_E2}

For $ x \geq 1$, given $|\Delta|\vee | \Delta_{i^*}| \leq (\frac{2 +\frakcm}{4}) x \leq \frac{3x}{4}$, by \eqref{fx'bdd} in \lemref{helping} and $|f'_x| \leq 1$ (\lemref{helping_unif}),
\begin{align*}
&|f_x(W_b +  \Delta_{i^*}) - f_x(W_b +  \Delta)| \\
&\leq  |f_x(W_b +  \Delta_{i^*}) - f_x(W_b +  \Delta)| \Bigg[ I\Bigg(W_b \leq \frac{x}{4} - 1\Bigg) + I\Bigg(W_b > \frac{x}{4} - 1\Bigg)\Bigg]\\
&\leq C \Big(e^{1/2-x} +I(W_b > x/4 - 1) \Big)\Big(|\bar{D}_{1, x} - \bar{D}_{1, i^* , x}| + x|\bar{\frakD}_2- \bar{\frakD}_{2, i^*} |\Big)\\
&\leq C\Big(e^{-x} + e^{-x/4} e^{W_b} \Big) \Big(|\bar{D}_{1, x} - \bar{D}_{1, i^*, x}| + x|\bar{\frakD}_2- \bar{\frakD}_{2, i^*} | \Big).
\end{align*}
Hence, 
\begin{align*}
|E_2| 
&\leq C_1e^{-x}  \sum_{i=1}^n ( \|\xi_{b, i}\|_2 \|\bar{D}_{1, x} - \bar{D}_{1, i^*, x} \|_2 +  x \|\xi_{b, i}\|_3 \| \bar{\frakD}_2- \bar{\frakD}_{2, i^*} \|_{3/2} ) +\\
&\hspace{2cm} C_2e^{-x/4}\sum_{i=1}^n (\|\xi_{b, i} e^{\xi_{b, i}}\|_2 \|\bar{D}_{1, x} - \bar{D}_{1, i^*, x} \|_2 +x \|\xi_{b, i} e^{\xi_{b, i}}\|_3  \| \bar{\frakD}_2- \bar{\frakD}_{2, i^*} \|_{3/2})\\
&\leq C e^{-c x} \sum_{i=1}^n \Bigg( \|\xi_{b, i}\|_2 \|\bar{D}_{1, x} - \bar{D}_{1, x}^{(i)} \|_2 +  \|\xi_{b, i}\|_3  \| \bar{\frakD}_2- \bar{\frakD}_2^{(i)} \|_{3/2} \Bigg),
\end{align*}
where  we have applied Bennett's inequality (\lemref{bennett_censored}) on $e^{W_b^{(i)}}$ in the first inequality, and  $e^{\xi_{b,i}} \leq e$, and \propertyref{indep_copy_tri_ineq} in the second. This establishes \eqref{E2bdd}.

\subsubsection{Proof of the bound for $E_3$,  \eqref{E3bdd}}
We will form bounds for each of
\begin{multline*}
   \sum_{i =1}^n \mathbb{E} [\xi_{b, i} f_x( W_b^{(i)} +\Delta_{i^*}) ]  \\
  \mathbb{E}[ f_x'(W_b  + \Delta)] \sum_{i=1}^n \bE[(\xi_i^2 - 1) I(|\xi_i| >1)] \\
 \text{ and } \;\
\mathbb{E}[ \Delta f_x(W_b + \Delta) ],
\end{multline*}
which can conclude \eqref{E3bdd}.

Bounding the first two terms is relatively simple. For the first term, by  the independence between $\xi_{b, i}$ and $W_b^{(i)} + \Delta_{i^*}$, Bennett's inequality (\lemref{bennett_censored}),  $\Delta_{i^*} \leq 3x/4$, $0 < f_x \leq 0.63$ in  \lemref{helping_unif}, and \eqref{fxbdd} in \lemref{helping},
 \begin{align} 
 &\sum_{i =1}^n |\mathbb{E} [\xi_{b, i} f_x( W_b^{(i)} + \Delta_{i^*}) ]| \notag\\
&\leq \sum_{i =1}^n  \mathbb{E} \Big[ (|\xi_i| - 1) I(|\xi_i| > 1) \Big]\E\left[ f_x( W_b^{(i)} +\Delta_{i^*}) \left( \underbrace{ I(W_b^{(i)} \geq x/4 -1)}_{\leq e^{W_b^{(i)} +1} \cdot e^{- x/4}} + I( W_b^{(i)} < x/4 - 1)\right) \right] \notag\\
&\leq \sum_{i =1}^n \mathbb{E} [ \xi_i^2 I(|\xi_i |> 1) ]\left(  C  e^{- x/4} +  1.7 e^{-x}\right) \leq C e^{-x/4}\beta_2\label{extra_term_1_bbdb}.
\end{align}
For the second term, by  Bennett's  inequality (\lemref{bennett_censored}), that $\Delta \leq 3x/4$, $|f_x'| \leq 1$ in \lemref{helping_unif}, and \eqref{fx'bdd} in  \lemref{helping},
\begin{align}
&\left|\mathbb{E}[ f_x'(W_b+ \Delta)] \sum_{i=1}^n \bE[(\xi_i^2 -1)I(|\xi_i| >1)] \right|\notag \\
&\leq \beta_2 \E \left[ |f_x'(W_b + \Delta)| \{ I(W_b < x/4 - 1) + I(W_b \geq x/4 -1)\} \right] \notag \\
&\leq \beta_2  ( e^{1/2 -x} +   C e^{-x/4}) \leq C e^{-x/4}\beta_2.   \label{extra_term_3_bbdb}
\end{align}
Both \eqref{extra_term_1_bbdb} and \eqref{extra_term_3_bbdb} are  less than $C e^{-cx}\sum_{i=1}^n \bE[|\xi_i|^3]$, forming a part of \eqref{E3bdd}.

To finish proving \eqref{E3bdd}, it remains to show the bound
\begin{equation} \label{3rd_term_bdd_E3}
 \Big|\mathbb{E}[ \Delta f_x(W_b + \Delta) ]\Big|  \leq C(m) e^{-cx} \Bigg(    \|\barD_{1, x}\|_2 + \bE[ (1 + e^{W_b}) \frakD_2^2  ]\Bigg) + x \Bigg|\bE[\frakD_2 f_x(W_b)]\Bigg|,
\end{equation}
for the last term,
which is more delicate to derive. We first write
\begin{multline}
 \Big|\mathbb{E}[ \Delta f_x(W_b + \Delta) ]\Big| 
 = \Bigg| \Delta \int_0^{\Delta} f_x'(W_b +t) dt +  \mathbb{E}[ \Delta f_x(W_b ) ]\Bigg|\\
  \leq  \Bigg| \Delta \int_0^{\Delta} f_x'(W_b +t) dt  \Bigg| +  \Bigg|\mathbb{E}[ \Delta f_x(W_b ) ] \Bigg|, \label{tri_ineq_on_deltafW}
\end{multline}
and will control the two terms on the right hand side separately.

For the first right-hand-side term in \eqref{tri_ineq_on_deltafW}, since $\Delta \leq  3x/4$, we have
\begin{align}
&\Bigg| \Delta \int_0^{\Delta} f_x'(W_b +t) dt  \Bigg| \notag\\
&\leq  2e^{1/2 - x}  \bE[\bar{D}_{1, x}^2 + x^2  \bar{\frakD}_2^2/4] +  2\bE[ \underbrace{I(W_b > x/4 - 1 )}_{\leq e^{W_b +1-x/4}} ( \bar{D}_{1, x}^2 + x^2  \bar{\frakD}_2^2/4)] \notag\\
& \hspace{3cm}\text{ by \eqref{fx'bdd} in \lemref{helping} and that $\Delta^2 \leq 2 (\barD_{
1, x}^2 + x^2 \bar{\frakD}_2^2/4)$}  \notag\\
&\leq C_1  e^{-x}  (\bE[ \bar{D}_{1, x}^2] +x^2 \bE[\bar{\frakD}_2^2])+ C_2 e^{-x/4}  x^2 \bE[e^{W_b} \barD_{1, x}^2/x^2] + C_3 e^{-x/4} x^2\bE[ e^{W_b}\bar{\frakD}_2^2] \notag\\
&\leq  C e^{-x/4} (  x  \|\barD_{1, x}\|_2 + x^2 \bE[ (1 + e^{W_b})\bar{\frakD}_2^2  ]), \label{1st_RHS_term}
\end{align}
where \eqref{1st_RHS_term} is true because, with  \lemref{bennett_censored} and $|\barNR|/x \leq \frakcm/4 \leq 1/4$,
\[
  \bE[e^{W_b} \barD_{1, x}^2/x^2] \leq \|e^{W_b}\|_{2} \|\barD_{1, x}^2/x^2 \|_2 = C \sqrt{\bE[\barD_{1, x}^4/x^4]} \leq  \frac{C \|\barD_{1, x}\|_2}{x}.\
\]
and 
\[
\bE[ \bar{D}_{1, x}^2] = (x/4)^2 \bE\Bigg[ \frac{\bar{D}_{1, x}^2}{(x/4)^2}\Bigg] \leq (x/4) \bE[\barNR] \leq C x \|\bar{D}_{1, x}\|_2
\]

For the second right-hand-side term in \eqref{tri_ineq_on_deltafW}, using $0 < f_x(w) \leq 0.63$ (\lemref{helping_unif}), 
\begin{align}
 &|\mathbb{E}[ \Delta f_x(W_b ) ]| \notag\\
& \leq \bE[ |\barD_{1, x} f_x(W_b)|]  + \frac{x}{2}  \Big|\bE[\bar{\frakD}_2 f_x(W_b)] \Big|\notag \notag\\
&\leq  0.63 e^{1-x}  \bE[ |\barD_{1, x}| e^{W_b}]  +  1.7e^{-x} \|\barD_{1,x}\|_2 +  \frac{x}{2}  \Big|\bE[\bar{\frakD}_2 f_x(W_b)] \Big| \notag \\
& \hspace{0.5cm}  \text{ by $0 < f_x(w) \leq 0.63$, \eqref{fxbdd} from \lemsref{helping_unif}, \lemssref{helping}  and $I(W_b > x-1) \leq e^{W_b +1 -x}$} \notag\\
&\leq C e^{-x } \|\barNR\|_2 +  \frac{x}{2}  \Big|\bE[\bar{\frakD}_2 f_x(W_b)] \Big| \text{ by Bennett's inequality (\lemref{bennett_censored})}  \notag\\
&\leq C (m) \Big( e^{-x } \|\barNR\|_2  + x e^{-x} \bE[\frakD_2^2 (1 + e^{W_b})]\Big) + \frac{x}{2} \bE[\frakD_2 f_x(W_b)], \label{2nd_RHS_term}
\end{align}
which can conclude \eqref{3rd_term_bdd_E3} in combination with  \eqref{tri_ineq_on_deltafW} and \eqref{1st_RHS_term}.
The last inequality \eqref{2nd_RHS_term} comes as follows: 
Write
\[
\bE[\bar{\frakD}_2 f_x(W_b)] = \bE[ (\bar{\frakD}_2 - \frakD_2)f_x(W_b)] + \bE[ \frakD_2 f_x(W_b)],
\]
Now,  defining $\frak{C}_m = 1 - \frac{9 \frakcm^2}{16}$ (where $0 < \frak{C}_m < 1$),
\begin{align*}
&| \bE[ (\bar{\frakD}_2 - \frakD_2)f_x(W_b)] |\\
&\leq \bE[|\bar{\frakD}_2 - \frakD_2|f_x(W_b) I(W_b \leq x -1)]  +  \bE[|\bar{\frakD}_2 - \frakD_2|f_x(W_b) I(W_b > x -1)]\\
&\leq 1.7e^{-x} \bE \Bigg[ | \frakD_2- \frak{C}_m |  I\Big(| \frakD_2| > \frak{C}_m\Big)\Bigg] + 0.63 e^{1- x} \bE \Bigg[ | \frakD_2- \frak{C}_m |  I \Big(| \frakD_2| > \frak{C}_m\Big) e^{W_b} \Bigg] \\
& \hspace{3cm}\text{ by $\eqref{fxbdd}$ and $0 < f_x(w) \leq 0.63$ from  \lemsref{helping_unif}}\\
&\leq 1.7e^{-x} \bE[| \frakD_2 |  I(| \frakD_2 | >  \frak{C}_m)] + 0.63 e^{1- x} \bE[ |  \frakD_2|  I(| \frakD_2 | > \frak{C}_m) e^{W_b}] \\
&\leq C(m) e^{-x} \bE[\frakD_2^2 (1 + e^{W_b})],
\end{align*}
where the last line uses that $ I(| \frakD_2| >  \frak{C}_m) \leq  \frak{C}_m^{-1} |\frakD_2 |$.

\section{Proof of \lemsref{chernoff_lower_tail_bdd_U_stat} and \lemssref{non_integral_moment_bdd_u_stat}} \label{app:otherproofs}

\begin{proof}[Proof of \lemref{chernoff_lower_tail_bdd_U_stat}]
As a useful fact, we first note that, for any $p \in (1, 2]$,
\begin{equation} \label{lower_exp_bdd_crucial_ineq}
e^{-s} \leq 1 - s + s^p/p \text{ for } s\geq 0.
\end{equation}
This is because the derivative of $1 - s + s^p/p - e^{-s}$ as a function in $s$ has the form
\begin{equation} \label{1st_derivative}
\frac{\partial }{ \partial s} (1 - s + s^p/p - e^{-s}) = s^{p-1} + e^{-s} -1,
\end{equation}
which can be seen to be  non-negative for all $s \in [0, \infty)$. (This is obvious for $s \in (1, \infty)$ since $s^{p-1} > 1 > 1 - e^{-s}$ for $1 \leq s < \infty$; and it is also true for $s \in [0, 1)$ since $1 - e^{-s} \leq s \leq s^{p-1}$ for $0  \leq s \leq 1$.)

Using the trick by \citet[Section 5, Eqn. $(5.4)$]{hoeffding1963probability}, one can write
\[
U_n= \frac{1}{n!} \sum W(X_{i_1}, \dots, X_{i_n}),
\]
where the summation is over all $n!$ permutation of $(i_1, \dots, i_n)$ of $(1, \dots, n)$ and 
\[
W(x_1, \dots, x_n) = \frac{h(x_1, \dots, x_m) + h(x_{m+1}, \dots, x_{2m}) + \cdots + h(x_{km-m +1}, \dots, x_{km})}{\kappa},
\]
where $\kappa \equiv [ n/m]$, the greatest integer $\leq n/m$. By the Chernoff bounding technique and Jensen's inequality, for any $t >0$,
\begin{align*}
P(U_n \leq x) 
&\leq e^{tx} \bE[e^{-t U_n}]\\
&\leq e^{tx} \bE[e^{-t W(X_1, \dots, X_n)}] = e^{tx} (\bE[ e^{-t h(X_1, \dots, X_m)/\kappa}])^\kappa.
\end{align*}
Using that $h(X_1, \dots, X_m) \geq 0$ and \eqref{lower_exp_bdd_crucial_ineq},
 we can continue and get that
\begin{align*}
P(U_n \leq x)  
&\leq e^{tx} \left\{1 - \frac{t}{\kappa} \bE[h] + \frac{t^p}{p\kappa^p} \bE[h^p]\right\}^\kappa\\
&\leq \exp\left\{ t(x- \bE[h]) + \frac{t^p}{p\kappa^{p-1}} \bE[h^p]\right\},
\end{align*}
where the last inequality uses that $1 + y \leq e^y$ for all $y \in \bR$. By minimizing the right hand side with respect to $t$, one can take $t = \kappa \Big(\frac{\bE[h] - x}{\bE[h^p]}\Big)^{1/(p-1)}$ and obtain
\begin{align*}
P(U_n \leq x)  &\leq   \exp\left( \frac{- \kappa(\bE[h] - x)^{p/(p-1)}}{ (\bE[h^p])^{1/(p-1)}} + \frac{\kappa(\bE[h] - x)^{p/(p-1)}  }{p(\bE[h^p])^{1/(p-1)}} \right) \\
&= \exp\left( - \frac{ (p-1)\kappa (\bE[h] - x)^{p/(p-1)} }{p (\bE[h^p])^{1/(p-1)} }\right)
\end{align*}
\end{proof}

\begin{proof}[Proof of \lemref{non_integral_moment_bdd_u_stat}]
Define  the \emph{canonical} functions \citep[p.20-21]{korolyuk2013theory} 
\begin{align*}
g_1(x_1) &= h_1(x_1)\\
g_2(x_1, x_2) &=h_2(x_1, x_2) - g_1(x_1) - g_1(x_2)\\
& \vdots\\
g_m(x_1, \dots, x_m)&= h_m(x_1, \dots, x_m) - \sum_{l=1}^m g_1(x_l) - \sum_{1 \leq l_1 < l_2 \leq m} g_2(x_{l_1}, x_{l_2}) - \\
&\hspace{3cm}\dots - \sum_{1 \leq l_1 < \dots < l_{m-1} \leq m}g_{m-1}(x_{l_1}, \dots, x_{l_{m-1}}).
\end{align*}
Note that $r$ can be alternatively defined as the first integer such that, as functions, 
\[
g_k(x_1, \dots,x_k) = 0\text{ for } k = 1, \dots, r-1, \text{ and } \quad g_r(x_1, \dots, x_r) \neq 0;
\]
see the discussion in \citet[p.32]{korolyuk2013theory} for instance. 
 Then  \citet[Theorem 2.1.3 \& 2.1.4]{korolyuk2013theory} suggest that
\[
 \bE[|U_n|^p] \leq  
  \begin{cases} 
  (m- r +1)^{p-1} \sum_{k=r}^m {m \choose k}^p  {n \choose k}^{-p+1} \alpha_p^{k+1 } \bE[|g_k|^p] & \text{if } 1 \leq p \leq 2; \\
     (m- r +1)^{p-1}   \sum_{k=r}^m  {m \choose k}^p {n \choose k}^{-p+1} n^{ ( (p-2)k)/2} \gamma_p^{k+1}  \bE[|g_k|^p]& \text{if }  p \geq 2,
  \end{cases}
\]
where $\alpha_p \equiv \sup_x  (|x|^{-p} ( |1 +x|^p- 1 - px))\leq 2^{2-p}$ and $\gamma_p \equiv \{8(p-1)\max(1, 2^{p-3})\}^p$. The bound
\eqref{simple_u_stat_moment_bdd} is a simple consequence of this  based on \eqref{Jensen}.\end{proof}

\bibliographystyle{plainnat}

\bibliography{BE_nonunif}

\end{document}